\documentclass[11pt, psamsfonts, reqno]{amsart}
\usepackage{amssymb,amsfonts}
\usepackage[margin=1in]{geometry}
\usepackage{amsthm, scrextend}
\usepackage{fancyhdr}
\usepackage{qtree} 
\usepackage{pgfplots}
\usepackage{amsmath}
\usepackage{eqnarray}
\usepackage[mathscr]{euscript}
 \let\mathscr\relax% just so we can load this and rsfs
\usepackage[scr]{rsfso}
\usepackage[all,arc]{xy}
\usepackage{enumerate}
\usepackage{mathrsfs}
\numberwithin{equation}{section}
\usepackage[mathscr]{euscript}
\usepackage{multicol}

\usepackage{times}
\usepackage[T1]{fontenc}
\usepackage{mathrsfs}
\usepackage{epsfig}
\usepackage{color}

\theoremstyle{plain} 
\newtheorem{thm}{Theorem}[section]
\numberwithin{thm}{section}
\newtheorem{cor}[thm]{Corollary}
\newtheorem{prop}[thm]{Proposition}
\newtheorem{lem}[thm]{Lemma}

\theoremstyle{definition}
\newtheorem{defn}[thm]{Definition}

\newtheorem{exmp}[thm]{Example}

\theoremstyle{remark}
\newtheorem{rem}[thm]{Remark}

\newcommand{\NN}{{\mathbb N}}

\newcommand{\bfA}{{\bf A}}

\newcommand{\ffD}{{\mathfrak D}}
\newcommand{\sA}{{\mathcal A}}

\newcommand{\sB}{{\mathcal B}}
\newcommand{\sC}{{\mathcal C}}
\newcommand{\sE}{{\mathcal E}}

\newcommand{\sF}{{\mathcal F}}
\newcommand{\sG}{{\mathcal G}}
\newcommand{\sH}{{\mathcal H}}

\newcommand{\sL}{{\mathcal L}}

\newcommand{\sP}{{\mathcal P}}
\newcommand{\sQ}{{\mathcal Q}}

\newcommand{\sV}{{\mathcal V}}
\newcommand{\sY}{{\mathcal Y}}

\newcommand{\rS}{{\rm S}}

\newcommand{\msrA}{{\mathcal{A}}}
\newcommand{\mitF}{{\mathit{F}}}
\newcommand{\vF}{{\mathit{F}}}
\newcommand{\fF}{{\mathfrak{F}}}
\newcommand{\bv}{{\bf v}}

\newcommand{\bx}{{\bf x}}
\newcommand{\by}{{\bf y}}
\newcommand{\Ker}{{\rm Ker}}

\newcommand{\pr}{{\circledast}}
\newcommand{\prn}{{\circledast_n}}
\newcommand{\Gr}{{\mathcal G}} 
\newcommand{\lcm}{{\rm lcm}} 

\usepackage{times}
\usepackage[T1]{fontenc}
\usepackage{mathrsfs}
\usepackage{epsfig}
\usepackage{color}
\usepackage{pstricks}
\usepackage{pst-node}
\usepackage{tkz-graph}
\newcommand{\stdnodesep}{3}
\newcommand{\doffset}{3pt}

\newcommand{\node}[3]{\rput{0}(#2){\ovalnode{#3#1}{\Large #1}}}

\newcommand{\bline}[3]{%
	\ncline[nodesepA=\stdnodesep,nodesepB=\stdnodesep]%
	{->}{#1}{#2}%
	\Bput{#3}%
}

\newcommand{\dline}[4]{%
	\ncarc[nodesepA=\stdnodesep,nodesepB=\stdnodesep,offset=\doffset]%
	{->}{#1}{#2}%
	\Aput{#3}%
	\ncarc[nodesepA=\stdnodesep,nodesepB=\stdnodesep,offset=\doffset]%
	{->}{#2}{#1}%
	\Aput{#4}%
}

\newcommand{\ddline}[4]{%
	\ncarc[nodesepA=\stdnodesep,nodesepB=\stdnodesep,offset=\doffset]%
	{->}{#1}{#2}%
	\Aput{#3}%
	\ncarc[nodesepB=\stdnodesep,nodesepA=\stdnodesep,offset=-\doffset]%
	{->}{#1}{#2}%
	\Bput{#4}%
}

\newcommand{\eeline}[5]{%
	\ncarc[nodesepA=3,nodesepB=3,offset=11pt]
	{->}{#1}{#2}%
	\Aput{#3}%
	\ncarc[nodesepA=3, nodesepB=3,offset=-10pt]%
	{->}{#1}{#2}%
	\Bput{#4}%
	\ncline[nodesepA=\stdnodesep,nodesepB=\stdnodesep]%
	{->}{#1}{#2}%
	\Aput{#5}%
}

\newcommand{\bcircle}[3]{%
	\nccircle[angleA=#2,nodesepA=\stdnodesep]{->}{#1}{20pt}%
	\Bput{#3}%
}
\numberwithin{equation}{section}
\title{Interleaving of Path Sets}

\author{William C.  Abram}
\address{ USAA , 9800 Fredericksburg Road, 
San Antonio, TX 78288}
\email{abramwc@gmail.com}
\author{ Jeffrey C. Lagarias}
\thanks{The research of the second author was supported by NSF grant 
 DMS-1701229, and a 2018 Simons Fellowship in Mathematics Award 555520.}
\address{Department of Mathematics, University of Michigan, 
Ann Arbor, MI 48109-1043,USA}
\email{lagarias@umich.edu}
\author{Daniel  Slonim}
\address{Department of Mathematics, Purdue University, 
West Lafayette, IN 47906}
\email{dslonim@purdue.edu}

\date{January 7, 2021}

\begin{document}

\begin{abstract}

Path sets  are spaces of one-sided infinite symbol sequences corresponding to the one-sided infinite walks beginning at a fixed initial 
vertex  in a directed labeled graph $\mathcal{G}$. 
Path sets are a generalization of one-sided sofic shifts.
 This paper  studies  decimation operations $\psi_{j, n}(\cdot)$ which extract symbol sequences in infinite arithmetic progressions $(\bmod\, n)$
starting with the symbol at position $j$. 
It also studies a family of $n$-ary interleaving operations $\prn$, one  for each $n \ge 1$, 
which act on an ordered set $(X_0, X_1, ..., X_{n-1})$ of one-sided symbol sequences $X_0 \pr X_1 \pr \cdots \pr X_{n-1}$ on an alphabet $\sA$,
by interleaving the symbols of each  $X_i$ in arithmetic progressions $(\bmod \, n)$,
It studies  a set of closure operations relating interleaving and decimation. 
This paper gives  basic algorithmic results on presentations of 
path sets and existence of a minimal right-resolving presentation.
It gives an algorithm for computing presentations of  decimations of path sets from 
presentations of path sets, showing the minimal right-resolving presentation of $\psi_{j,n}(X)$ has at most one more vertex
than a minimal right-resolving  presentation of $X$.
It shows  that a path set  has only finitely many distinct decimations.  
It shows the class of path sets on a  fixed alphabet is closed under interleavings and gives an algorithm to
compute presentations of interleavings of path sets. 
It studies interleaving factorizations and classifies path sets
that have infinite interleaving factorizations and gives an algorithm to recognize them.
It shows the finiteness of a  process of iterated interleaving
factorizations, which ``freezes" factors that have infinite interleavings. 
\end{abstract}

\maketitle

\tableofcontents

%%%%%%%%%%%%%%%%%%%%%%%%%%%%%%%%%%%%%%%%%%%%%%%%%%%%%%%%%%%%%%%%%%%%%%%%%%%%%%%%
% SECT 1. Introduction
%%%%%%%%%%%%%%%%%%%%%%%%%%%%%%%%%%%%%%%%%%%%%%%%%%%%%%%%%%%%%%%%%%%%%%%%%%%%%%%
\section{Introduction}\label{sec:1}
Let $\mathcal{A}$ be a finite alphabet with at least two elements, and let $\mathcal{A}^\mathbb{N}$ denote the full one-sided shift on $\mathcal{A}$: that is, the space of 
one-sided infinite strings of
symbols from $\sA$. One-sided symbolic dynamics studies the action of the one-sided shift map
 $S: \sA^{\NN} \to \sA^{\NN}$ given by
  $$
   \rS(a_0a_1 a_2 \cdots) = a_1a_2a_3 \cdots.
 $$
on subsets $X \subset \sA^{\NN}$. 

This paper considers  a class of such $X$ called {\em path sets}, studied in \cite{AL14a}. 
Consider   a  finite  labeled graph $\mathcal{G} = (G, \mathcal{E})$ whose 
underlying graph $G=(V, E)$ is a directed graph with  edge set $E$ (permitting loops and multiple edges), with labeling  $\sE: E \to \mathcal{A}$ 
specifying a labeling of the graph edges 
by elements of $\mathcal{A}$, and some fixed initial vertex $v$ of $\mathcal{G}$. From $(\sG, v)$ 
 we obtain  a set $\mathcal{P}=X(\mathcal{G}, v) \subset \mathcal{A}^\mathbb{N}$
  of symbol sequences  $x=a_0a_1a_2 \cdots $ giving  the edge labels of all one-sided infinite walks in $\mathcal{G}$ that begin at the 
  marked vertex $v$. We call any such $\mathcal{P}$ 
  a \emph{path set},  and 
  we call $(\mathcal{G},v)$ 
  a \emph{presentation} of $\mathcal{P}=X(\mathcal{G}, v)$. The name path set for these objects was given  in \cite{AL14a}, where previous work was reviewed. 
      We let $\mathcal{C}(\mathcal{A})$ denote the class of all path sets on the alphabet $\mathcal{A}$.   
      Path sets are a generalization of one-sided sofic shifts in symbolic dynamics; however in general they 
      are  not invariant under the one-sided shift map.
The main feature of path sets is  
specification  of a  fixed  initial vertex $v$, which encodes an initial condition
which can  break shift-invariance in the dynamics.

  The path set concept has recently found application to  number theory, fractal geometry and  to neural networks. 
  More specifically, Abram, Bolshakov, and Lagarias \cite{ABL17}  used path sets to study 
  Hausdorff dimension of intersections of multiplicative translates of $p$-adic Cantor sets, for application to a problem of Erd\H{o}s on 
  ternary expansions of powers of $2$, see \cite{AL14b, AL14c, Lagarias09}. It turns out that these intersections are always presentable
   as $p$-adic path set fractals, a kind of geometric 
  realization of path sets inside the $p$-adic integers $\mathbb{Z}_p$ viewed as the full shift on $\{0,1,\ldots, p-1\}$. 
  In another direction  Ban and Chang \cite{BC16} show that the mosaic solution space of the initial value problem of a multi-layer 
  cellular neural network is topologically conjugate to a path set. 
   Thus, according to Ban and Chang \cite{BC16}, ``The more properties we know about path sets, the more we know about the 
   topological structure of the solution spaces
 derived from differential equations with initial conditions, and vice versa.''

  This paper studies  the effect on path sets of two families  of operations defined on  one-sided symbol sequences $\sA^{\NN}$, which are {\em decimations} and {\em interleavings}.
  Decimation and interleaving operations are initially defined  using individual elements of $\sA^{\NN}$, extend to act on arbitrary subsets $X \subseteq \sA^{\NN}$ by set union,
and then are specialized to path sets $\sP \subseteq \sA^{\NN}$ in this  paper.
The  paper \cite{ALS21} studies these operations acting on general sets  $X \subseteq \sA^{\NN}$.
%%%%%%%%%%%%%%%%%%%%%%%%%%%%%%%%%%%%%%%%%%%%%%%%%%%%%%%%%%%%%
% SECT 1.1  Interleaving and Decimation operations
%%%%%%%%%%%%%%%%%%%%%%%%%%%%%%%%%%%%%%%%%%%%%%%%%%%%%%%%%%%%%\subsection{Interleaving and decimation operations}\label{subsec:11}

Decimation operations on path sets were 
 studied by the first two authors in \cite{AL14a}. 
%%%%%%%%%%%%%%%%%%%%%%%%%%%%%%%%%%%%%%%%%%%%%%%%%%%%%%%
% Definition 1.1
%%%%%%%%%%%%%%%%%%%%%%%%%%%%%%%%%%%%%%%%%%%%%%%%%%%%%%%
\begin{defn}  \label{decimationdef}
Let  $\mathcal{A}$ be a finite alphabet of symbols. 

(1) For $ j \ge 0$ the {\em $j$-th decimation operation at level $n$}, $\psi_{j,n}: \sA^{\NN} \to \sA^{\NN}$, is defined on an individual 
sequence  $\bx= x_0x_1x_2x_3 \cdots$ by
$$
\psi_{j,n}(\bx) = \by:=  x_j x_{j+n}x_{j + 2n}x_{j+3n}  \cdots.
$$

(2) The {\em $j$-th decimation  of level $n$}, denoted  $\psi_{j, n}(X)$, 
of the set $X \subset \sA^{\NN}$
is the set defined
$$
\psi_{j, n}(X) := \bigcup_{ \bx \in X} \psi_{j,n}(\bx) 
$$
\end{defn}

The definition applies for  all $j \ge 0$, but a special role is played by the operations with $0 \le j \le n-1$,
which we term {\em principal decimations}, see \eqref{eqn:basic-relation} below. The set of all  level $n$ decimation operators  for  $j \ge n$
are obtained from principal decimations by iterating the one-sided  shift operator $\rS$, 
noting that $\psi_{j+n, n}(\bx) = \rS \circ \psi_{j, n} (\bx)$. 
 
The main emphasis of the paper is the family of {\em interleaving operations}, which   comprise an  infinite  collection of 
 $n$-ary operations $(n \ge 1)$, one for each $n$,  defined for arbitrary subsets $X$ of the shift space  $\mathcal{A}^\mathbb{N}$.

%%%%%%%%%%%%%%%%%%%%%%%%%%%%%%%%%%%%%%%%%%%%%%%%%%%%%%%
% Definition 1.2
%%%%%%%%%%%%%%%%%%%%%%%%%%%%%%%%%%%%%%%%%%%%%%%%%%%%%%%
\begin{defn}  \label{interleavingdef}
Let  $\mathcal{A}$ be a finite alphabet of symbols. 

(1) The  {\em $n$ fold interleaving operation} $(\pr)_{j=0}^{n-1} \sA^{\NN} \times \sA^{\NN} \times \cdots \times \sA^{\NN} \to \sA^{\NN}$ 
is defined pointwise on individual sequences: 
$\bx_{j} = a_{j,0} a_{j,1}a_{j,2} \cdots $ for $0 \le j \le n-1$ by 
$$
(\bx_0, \bx_1, \cdots, \bx_{n-1}) \mapsto  \by :=(\prn)_{j=0}^{n-1}  \bx_j  = \bx_0 \pr \bx_1 \pr \cdots \pr \bx_{n-1}= b_0b_1b_2 \cdots
$$
with symbol sequence
$$
b_{ni+k}= a_{k, i}  \quad \mbox{for} \quad i \ge 0, \,\,  0 \le k \le n-1. 
$$

(2) The \emph{$n$-fold interleaving}  $X =  (\prn)_{i=0}^{n-1} X_i$
 of the the sets 
 $X_0, X_1, X_2,\ldots, X_{n-1} \subseteq \mathcal{A}^\mathbb{N}$
in the specified  order is 
\begin{align*}
\quad\quad  (\prn)_{i=0}^{n-1}X_i & := 
 X_0 \pr X_1\pr X_2\pr \cdots \pr X_{n-1} \\
& 
 = \{(x_i)_{i=0}^{\infty}\in\mathcal{A}^{\mathbb{N}} :\, x_jx_{j+n}x_{j+2n}\ldots \in X_j \text{ for all }0\leq j\leq n-1\}.
 \end{align*}
\end{defn}

The   notation $\pr$ above does not indicate
the arity of the $n$-fold interleaving operation $\prn$; the arity of composed operations is to be inferred 
via  groupings  of parentheses.
 The  $n$-fold interleaving operations of arity $2$ are not commutative, in general 
 $X_1 \pr X_2 \ne X_2 \pr X_1$, and not associative, in general $(X_1 \pr X_2) \pr X_3 \ne X_1 \pr(X_2 \pr X_3)$.

Interleaving and decimation operators  are related by  a pointwise  identit  stating that the 
 $n$-fold interleaving of the ordered set of principal decimations at  level $n$ gives  the identity map: 
\begin{equation}\label{eqn:basic-relation}
(\prn)_{j=0}^{n-1} \psi_{j, n}(\bx) = \bx \quad \mbox{for} \quad \bx \in \sA^{\NN}.
\end{equation} 
At the set level it shows that  the level $n$ principal decimations provides a
 right-inverse to recover the individual factors  of an  $n$-fold interleaving  
 \begin{equation}\label{eqn:basic-relation2}
 X_j = \psi_{j, n} ( X_0 \pr X_1 \pr \cdots \pr X_{n-1}).
\end{equation}

%%%%%%%%%%%%%%%%%%%%%%%%%%%%
% SECT 1.2 Symbolic Dynamics   
%%%%%%%%%%%%%%%%%%%%%%%%%%%%%%%%%%%%%%%%%%%%%%%%%%%%%%%%%%%%%%%%%%%%%%%%%%%%%%%
\subsection{Symbolic dynamics}\label{subsec:12}

This paper studies decimation and  interleaving operations on path sets  from the viewpoint of symbolic dynamics and coding theory,  with (one-sided)  symbol sequences 
viewed inside the full one-sided shift $\mathcal{A}^\mathbb{N}$ with the shift topology, which is the product  topology where each factor $\mathcal{A}$ has the discrete topology,
so that $\mathcal{A}^\mathbb{N}$ is a compact set. Much of symbolic dynamics studies closed  sets $X$ which are shift-invariant: $\rS X= X$. 

The interleaving concept is useful in coding theory as a method for improving the burst error correction capability of a code, cf. \cite[Section 7.5]{VO89}. 
The analogue of Definition ~\ref{interleavingdef} for finite codes is referred to by coding theorists as \emph{block interleaving at depth $n$}.
Each  interleaving operation  $\pr_n$ respects the shift topology in the sense that   if  $\{ X_i: 0 \le i \le n-1\}$ are closed subsets of the one-sided shift space $\mathcal{A}^\mathbb{N}$ 
then  $X= (\prn)_{i=0}^{n-1} X_i=  X_0 \pr X_1 \pr X_2\pr \cdots \pr X_{n-1}$ is also a closed subset of $\mathcal{A}^\mathbb{N}$.
\smallskip

A great deal of study has been given to the subclasses of {\em shifts of finite type}, and of the generalization to {\em sofic shifts}, studied 
for two-sided infinite sequences in  Lind and Marcus \cite{LM95}. These are special cases of path sets. General path sets $\sP$ are not shift-invariant,
but we show in Theorem \ref{thm:weak-shift-invariant-0} below that they satisfy a weak form of shift-invariance. 

%%%%%%%%%%%%%%%%%%%%%%%%%%%%%%%%%%%%%%%%%%%%%%
% SECT 1.3  Path sets 
%%%%%%%%%%%%%%%%%%%%%%%%%%%%%%%%%%%%%%%%%%%%%%%%%%%%%%%%%%%%%%%%%%%%%%%%%%%%%%%
\subsection{Path sets}\label{subsec:13}

We recall the definition of path set from \cite{AL14a}. 
A \emph{pointed graph} $(\mathcal{G},v)$
over a finite alphabet $\mathcal{A}$   comprises a
finite edge-labeled directed graph
  $\mathcal{G}= (G, \sE)$  and  a distinguished vertex $v$ of 
the underlying directed graph $G$. The directed graph  $G=(V, E) $ is specified by its vertex set $V$
and (directed)  edge set $E$ with edges $e=(v_1, v_2) \in V \times V$, and the data 
$\sE\subset E \times \sA$  
specifies the set of labeled edges $(e, a)$,
with  labels 
% $a \in \sA$.
drawn from the  alphabet $\sA$. 
We allow loops and  multiple edges, but require that all triples $(e, a) = (v_1, v_2, a)$ be distinct.
We  use interchangeably   the terms {\em vertex} and {\em state} of $G$, as in Lind and Marcus \cite[Sect.  2.2]{LM95}.
The  results of this paper regard
the finite alphabet $\sA$ as fixed, unless specifically noted otherwise.

%%%%%%%%%%%%%%%%%%%%%%%%%%%%%%%%%%%%%
%
% Defn 1.3 path set
%
%%%%%%%%%%%%%%%%%%%%%%%%%%%%%%%%%%%%%%
\begin{defn} \label{de111}
% (path set) 
A  \emph{path set}\footnote{In \cite{AL14a} the notation $X_{\sG}(v)$ was used.}
 (or {\em pointed follower set}) 
   $\sP = X(\mathcal{G}, v) $ specified by a pointed labeled graph $(\sG, v)$ 
   and a distinguished vertex $v$
  is the subset of $\sA^{\NN}$ made up  of  the symbol sequences
of  successive edge labels of all possible one-sided infinite walks in $\mathcal{G}$ 
issuing from the distinguished vertex $v$. 
\end{defn}

We let $\sC(\sA)$ denote the collection of all path sets using labels from  the alphabet $\sA$.
We call  the data $(\sG, v)$ a  {\em presentation} of the path set $\sP= X(\mathcal{G},v)$. A path  set $\sP$ typically has many different presentations.
The set $\sP$ by definition includes every symbolic path sequence with multiplicity one, although the graph $\sG$ could potentially contain many paths starting from $v$
with identical symbol sequences.

The paper \cite{AL14a} showed that  the class $\mathcal{C}(\mathcal{A})$ of all path sets $\sP$ with fixed (finite) alphabet $\mathcal{A}$ is closed under 
all decimation operations $\psi_{j,n}(\sP)$  We will give a second proof of this result in Section \ref{sec:decimation}.

This paper shows in addition that  the class $\mathcal{C}(\mathcal{A})$ of all path sets with fixed (finite) alphabet $\mathcal{A}$ is closed under 
all  the interleaving operations $\pr_n$. Conversely it shows that if a path set $\sP$ has an $n$-fold interleaving factorization, then each factor
in the factorization is necessarily a path set. 
%(Theorem ~\ref{thm:intalgthm}).
  
In contrast, the smaller classes of one-sided sofic shifts and  of shifts of finite type are not closed under $n$-interleaving. 
interleaving can break one-sided shift-invariance even for the most well-behaved shift spaces. 
Path sets therefore appear to be a natural level of generality at which to study  interleaving operations.

%%%%%%%%%%%%%%%%%%%%%%%%%%%%%%%%%%%%%%%%%%%%%%%%%%%%%%%%%%%%%%%%%%%%%%%%%%%%%%%%
%
% SECT 1.4. Main Results
%
%%%%%%%%%%%%%%%%%%%%%%%%%%%%%%%%%%%%%%%%%%%%%%%%%%%%%%%%%%%%%%%%%%%%%%%%%%%%%%%
\subsection{Main results}\label{subsec:14}

In this paper we study the decimation and interleaving operations restricted to
path sets. 

 %%%%%%%%%%%%%%%%%%%%%%%%%%%%%%%%%%%%%%%%%%%%%%%%%%%%%%%%%%%%%%%%%%%%%%%%%%%%%%%%
% SECT 1.4.1   Presentations 
%%%%%%%%%%%%%%%%%%%%%%%%%%%%%%%%%%%%%%%%%%%%%%%%%%%%%%%%%%%%%%%%%%%%%%%%%%%%%%%
\subsubsection{Presentations of path sets}\label{subsec:141}

 In Section \ref{sec:prelim} we prove basic properties of presentations of path sets, 
 and discuss algorithms to test for these properties, in the language of symbolic dynamics.

 The  paper \cite[Section 3]{AL14a} showed that  every path set $\sP$ has a presentation with 
several additional properties: {\em right-resolving, reachable, and pruned}.\footnote{A labeled directed graph $\mathcal{G}$ is called \emph{right-resolving} if any two edges emanating from the same vertex have distinct labels.  A pointed graph $(\mathcal{G},v)$
 is called \emph{reachable} if there is a directed path in $\mathcal{G}$ from the initial vertex $v$ to every other vertex of $\mathcal{G}$; for a finite automaton this property is called
 {\em accessible}.  
 A reachable pointed graph 
$(\mathcal{G},v)$ is called \emph{pruned} if in  it has no sinks, meaning every vertex has out-degree at least one; for finite automata this property is called 
{\em coaccessible}. A finite automaton is \emph{trim} if it is both accessible and coaccessible, see \cite{Eilenberg74}, \cite[page 27]{PP04}.}
The right-resolving property  guarantees uniqueness of  symbolic paths:  any two distinct   paths in the graph $\sG$ starting from the initial vertex $v$ have different  symbol sequences;
 that is, for such a presentation the (finite or infinite) symbol sequence uniquely specifies the path.

 In contrast in symbolic dynamics   sofic systems need not have unique
 minimal right-resolving presentations (in the sense of sofic system presentations)
 if the sofic system is reducible, 
 compare \cite[Example 3.3.21]{LM95}.

 %%%%%%%%%%%%%%%%%%%%%%%%%%%%%%%%%%%%%%%%%%%%%%%%%%%%%%%%%%%%%%%%%%%%%%%%%%%%%%%%
% SECT 1.3.2   Decimation 
%%%%%%%%%%%%%%%%%%%%%%%%%%%%%%%%%%%%%%%%%%%%%%%%%%%%%%%%%%%%%%%%%%%%%%%%%%%%%%%
\subsubsection{ Decimations and the shift}\label{subsec:142}
 %%%%%%%%%%%%%%%%%%%%%%%%%%%%%%%%%%%%%%%%%%%%%%%%%%%%%%%
 
 The paper \cite{AL14a} showed that all decimations $\psi_{j,n}(\sP)$ of a path set are path sets,
 giving a constructive algorithm to compute a presentation of $\psi_{j,n}(\sP)$ from that of $\sP$.
 Section \ref{sec:decimation}  presents   a different construction, {\em the modified $n$-th higher power presentation},
 to obtain  presentations $\psi_{j,n}(\sP)$ which 
 require at most one vertex more than the number of vertices $m$  of the input presentation
 of $\sP$.  (The $n$-th higher power construction  is well known,  cf. \cite{[Sect. 1.4]LM95}. )

 As a first consequence of the modified higher power presentation we 
 show that 
 path sets satisfy a weak version of shift-invariance.  We say that a  set $X$ is {\em weakly shift-invariant}
 if there exist $k >  j \ge 0$ with $\rS^k X = \rS^jX$ .
 
 %%%%%%%%%%%%%%%%%%%%%%%%%%%%%%%%%
% THM 1.4
%%%%%%%%%%%%%%%%%%%%%%%%%%%%%%
\begin{thm}\label{thm:weak-shift-invariant-0}
{\rm (Weak shift-invariance of path sets)}
For any path set $\sP$ there exist integers $k > j \ge 0$
giving the  equality  of iterated shifts
$\rS^k \sP =  \rS^j \sP.$
\end{thm} 

This result is proved as Theorem \ref{thm:weak-shift-invariant}.
It comes from the property that iterations of the shift operator $\rS$ are given
by $1$-decimations. We have in general 
 $\psi_{j,n}(\sP) =  \psi_{0,n}(\rS^j \sP)$ and  particular  $\psi_{j,1}(\sP)= \rS^j \sP.$ . 
\smallskip

  %

%%%%%%%%%%%%%%%%%%%%%%%%%%%%%%%%%%%%%%%%%%%%%%%%%%%%%%%
% Definition 1.5
%%%%%%%%%%%%%%%%%%%%%%%%%%%%%%%%%%%%%%%%%%%%%%%%%%%%%%%
\begin{defn} \label{def:full-decimation-set}
The {\em full decimation set}  $\ffD(X)$ of an arbitrary set $ X \subset \sA^{\NN}$ is
the set of all (principal and non-principal) decimations:
$$
\ffD(X) := \{ \psi_{j,m}(X): \, \mbox{for} \quad m \ge 1 \quad \mbox{and} \quad j \ge 0 \}. 
$$
\end{defn}

The  modified higher power construction implies finiteness for
 the set  of all decimations of a path set.

%%%%%%%%%%%%%%%%%%%%%%%%%%%%%%%
% THM  1.6 //1.8 [formerly 1.11)
%%%%%%%%%%%%%%%%%%%%%%%%%%%%%
\begin{thm}\label{thm:decimation_set_bound}
{\rm ( Full decimation set bound)} 
For each path set $\sP$ its full decimation set   $\ffD(\sP)$ is a finite set. 
If $\sP$ has a presentation having $m$ vertices, then $\ffD(\sP)$ has cardinality 
bounded by
$$
|\ffD(\sP) | \le 2^{ (m+1)^2|\sA|}.
$$
\end{thm} 
 
 This result is proved as Theorem \ref{thm:decimation_set_bound2}.
 In contrast, there are closed set $X \subset \sA^{NN}$ such that  $\ffD(X)$ is an infinite set, cf.  \cite[Section 7]{ALS21}.

 Theorem \ref{thm:decimation_set_bound}.
  answers a question raised in \cite[page 113]{AL14a}. That paper defined the
 $n$-kernel of a path set $\sP$ by
 $$\Ker_n(\sP) := \{ \psi_{j,n^k}(\sP) : j \ge 0, k \ge 0 \}.$$
 It called a path set {\em $n$-automatic} if $\Ker_n(\sP)$  is finite. 
  This notion was 
 proposed in analogy with the definition of {\em $n$-automatic sequences}
 made by Allouche and Shallit \cite{AS92}, \cite{AS03}. Since $\Ker_{n} (\sP) \subset \ffD(\sP)$,
 Theorem \ref{thm:decimation_set_bound} implies  that every
 path set is $n$-automatic in this sense for every $n \ge 1$.

The output presentation of the $n$-th higher power construction is not necessarily right-resolving, 
 even if the input presentation is right-resolving. However a  standard construction, the subset construction,
 shows  that if $\sP$ has a presentation with $m$ vertices,
 then each decimation $\psi_{j,n}(\sP)$ has a right-resolving presentation with at most $2^{m+1} -1$ vertices,
 an upper  bound which is independent of $n$. 
 
 %%%%%%%%%%%%%%%%%%%%%%%%%%%%%%%%%
% Theorem 1.7 /3.4
%%%%%%%%%%%%%%%%%%%%%%%%%%%%%%
\begin{thm}\label{thm:decimation_presentation0}
{\rm (Right-resolving presentations of decimation sets of a path set.)}
Given a path set $\sP$ on alphabet $\sA$ with at least two letters, having a 
(not necessarily right-resolving) presentation $\sP = X(\sG, v)$  with $m$ vertices.
Then for each $n \ge 1$ and each  $j \ge 0$ the  decimation set $\psi_{j,n} (\sP)$ has a right-resolving presentation having at most $2^{m+1}-1$ vertices.
\end{thm}
 
This result is proved as Theorem \ref{thm:decimation_presentation}.

 %%%%%%%%%%%%%%%%%%%%%%%%%%%%%%%%%%%%%%%%%%%%%%%%%%%%%%%%%%%%%%%%%%%%%%%%%%%%%%%%
% SECT 1.3.3
%%%%%%%%%%%%%%%%%%%%%%%%%%%%%%%%%%%%%%%%%%%%%%%%%%%%%%%%%%%%%%%%%%%%%%%%%%%%%%%
\subsubsection{ Closure of $\mathcal{C}(\sA)$ under interleavings}\label{subsec:143}

In Section \ref{sec:intalgsec} we start the study of interleaving operations on the class 
$\mathcal{C}(\sA)$ of all path sets on a fixed alphabet $\sA$.
We first show  that $\mathcal{C}(\sA)$  is closed under the interleaving operations.

%%%%%%%%%%%%%%%%%%%%%%%%%%%%%%%%%%%%%%%%%%%%%%%%%%%%%%%
% Theorem 1.8
%%%%%%%%%%%%%%%%%%%%%%%%%%%%%%%%%%%%%%%%%%%%%%%%%%%%%%%
\begin{thm}\label{thm:intalgthm} 
  {\rm ($\sC(\sA)$ is closed under interleaving )}
If $\mathcal{P}_0, \ldots, \mathcal{P}_{n-1}$ are path sets on the alphabet $\mathcal{A}$, then their $n$-fold interleaving 
$$X := (\prn)_{i=0}^{n-1} \sP_i= \mathcal{P}_0 \pr \mathcal{P}_1 \pr \cdots \pr \mathcal{P}_{n-1}$$
 is a path set; i.e., $X \in \sC(\sA)$,   
\end{thm}

This result is proved as Theorem \ref{thm:intalgthm2}.
We establish Theorem \ref{thm:intalgthm} 
as a corollary of an effective   algorithmic construction of 
the $n$-fold interleaving $\sP$ at the level of presentations of path sets.

%%%%%%%%%%%%%%%%%%%%%%%%%%%%%%%%%%%%%%%%%%%%%%%%%%%%%%%
% Theorem 1.9
%%%%%%%%%%%%%%%%%%%%%%%%%%%%%%%%%%%%%%%%%%%%%%%%%%%%%%%

\begin{thm}\label{thm:algthm} 
{\rm(Interleaving pointed graph product construction)}
 Let $n \ge 2$  and suppose $\mathcal{P}_0, \ldots, \mathcal{P}_{n-1}$ 
are path sets with given presentations $(\mathcal{G}_0,v_0), \ldots, (\mathcal{G}_{n-1},v_{n-1})$, respectively.
There exists a construction  taking as inputs these presentations and giving as output
a presentation $(\sH, \bv)$ of the $n$-fold interleaving
$X  := \mathcal{P}_0 \pr \mathcal{P}_1 \pr \cdots \pr \mathcal{P}_{n-1}.$ In particular $X= X(\sH, \bv)$
is a path set.  This construction has the following properties:
\begin{enumerate}
\item[(i)]
 If $\mathcal{G}_i$ has $k_i$ vertices for each $0 \leq i \leq n-1$, then $\mathcal{H}$ will have at most $n\prod_{i=0}^{n-1}k_i$ vertices. 
 \item [(ii)]
 If the pointed graphs $(\mathcal{G}_i, v_i)$ are right-resolving for all $0 \leq i \leq n-1$, then 
 the output pointed graph $\mathcal{H}$ will also be right-resolving.
 \item [(iii)]
 If the pointed graphs 
 $(\mathcal{G}_i, v_i)$ are pruned  for all $0 \leq i \leq n-1$, then 
 the output pointed graph $\mathcal{H}$ will also be pruned.
 \end{enumerate}
\end{thm}

This result is proved as Theorem \ref{thm:algthm2}.
Theorem \ref{thm:algthm} constructs $\sH$ by  a graph-theoretic construction, termed here the {\em $n$-fold interleaved (pointed)  graph product},
which takes as input   pointed graphs $(\sG_i, v_i)$, for $0 \le i \le n-1$  and produces  as output  a pointed graph
\begin{equation*}
(\sH, \bv) := (\prn)_ {i=0}^{n-1}( \sG_i, v_i),
\end{equation*}
such that the underlying path set $\sP =X(\sH, \bv) $ is the $n$-fold interleaving of the path sets $\sP_i = X(\sG_i, v_i)$.
The  presentation 
 $(\sH, v)$  found by the construction depends  on the input presentations of $\sP_i$.
We provide examples showing  that minimality of  right-resolving presentations is not always preserved under the  interleaved pointed graph product.

 %%%%%%%%%%%%%%%%%%%%%%%%%%%%%%%%%%%%%%%%%%%%%%%%%%%%%%%%%%%%%%%%%%%%%%%%%%%%%%%%
% SECT 1.3.4   Decimation and interleaving Factorization
%%%%%%%%%%%%%%%%%%%%%%%%%%%%%%%%%%%%%%%%%%%%%%%%%%%%%%%%%%%%%%%%%%%%%%%%%%%%%%%
\subsubsection{ Decimations and Interleaving Factorizations  }\label{subsec:144}
 %%%%%%%%%%%%%%%%%%%%%%%%%%%%%%%%%%%%%%%%%%%%%%%%%%%%%%%
 
 Decimations and Interleaving operations together define a set of closure operations on symbol sets.
 We define the {\em $n$-fold interleaving closure} $X^{[n]} $ of a general set $X \subset \sA^{\NN}$ by
 \begin{equation}
 X^{[n]} = \psi_{0,n}(X) \pr \psi_{1,n}(X) \pr \cdots \pr \psi_{n-1, n}(X).
 \end{equation}
 These closure operations were studied in \cite{ALS21}. One always has $X \subseteq X^{[n]}$,
  the operation is idempotent: $(X^{[n]})^{[n]} = X^{[n]}$, and the operation takes closed sets to closed sets. 
  
 We say a set has an {\em $n$-fold interleaving factorization } if $X= X^{[n]}$.
 In that case  its associated {\em interleaving factors}
 are given by $X_{j,n} = \psi_{j,n}(X)$ for $0 \le j \le n-1$. These  interleaving  factors are unique, i.e.
 $X = X_0 \pr X_1 \pr \cdots \pr X_{n-1}$ and $X= Y_0 \pr Y_1 \pr \cdots \pr Y_{n-1}$ then $X_j= Y_j$
 for $0 \le j \le n-1$ as sets.

%%%%%%%%%%%%%%%%%%%%%%%%%%%%%%%%%%%%%%%%%%%%%%%%%%%%%%
% Theorem 1.10
%%%%%%%%%%%%%%%%%%%%%%%%%%%%%%%%%%%%%%%%%%%%%%%%%%%%%%%
\begin{thm}\label{thm:factorthm}
{\rm ($\sC(\sA)$  is stable under $n$-fold  interleaving closure operations)} 
If $\sP$ is a path set, then for each $n \ge 1$ the $n$-fold interleaving closure $\sP^{[n]}$ is a path set.
In addition, if $\sP$ is $n$-factorizable then each of its $n$-fold intereaving factors $\sP_i= \psi_{i,n}(\sP)$  for $0 \le i\le n-1$
are path sets.
\end{thm}. 
  
 This result is proved as Theorem \ref{thm:factorthm2}.
 All interleaving factors are decimations $\psi_{j,n}(\sP)$ with $0 \le j \le n-1$.
We show that  if a  path set $\sP$ has an $n$-fold interleaving presentation, then there is
an improved upper bound on the size of the associated decimations  $\psi_{j,n}(\sP)$
relative to that given in Theorem \ref{thm:decimation_presentation0}.

%%%%%%%%%%%%%%%%%%%%%%%%%%%%%%%%%%%%%%%%%%%%%%%%%%%%%%%%%%
%%%%%%%%%%%%%%%%%%%%%%%%%%%%%%%%%
% THM 1.11/ 5.8
%%%%%%%%%%%%%%%%%%%%%%%%%%%%%%
\begin{thm}\label{thm:56a}
{\rm ( Upper bound on minimal presentation size  of $n$-fold interleaving factors ) } 
Let   $\sP$ be a path set  having $m$ vertices in its minimal right-resolving presentation. 
Suppose that $\sP$ has   an  $n$-fold interleaving factorization
$\sP = (\pr)_{j=0}^{n-1} \sP_j$. Then  each $n$-fold interleaving factor $\sP_j = \psi_{j,n}(\sP)$
has a minimal right-resolving presentation  having
at most $m$  vertices.
\end{thm} 

This result is proved as Theorem \ref{thm:68}.
%Theorem 6.6. 
 
 %%%%%%%%%%%%%%%%%%%%%%%%%%%%%%%%%%%%%%%%%%%%%%%%%%%%%%%%%%%%%%%%%%%%%%%%%%%%%%%%
% SECT 1.3.5   Classifications o   interleaving Factorization
%%%%%%%%%%%%%%%%%%%%%%%%%%%%%%%%%%%%%%%%%%%%%%%%%%%%%%%%%%%%%

\subsubsection{Classification of  interleaving factorizations of path sets  }\label{subsec:145}

 The factorization  problem 
is the problem of finding all the possible interleaving factorizations of a path set $\sP$ under $n$-fold interleaving.

This problem is 
interesting and complicated due to the  fact  that some general sets $X \subset \sA^{\NN}$
may have factorizations for infinitely many $n \ge 1$, which we
call {\em infinitely factorizable} sets. The simplest example is  the   one-sided shift $\mathcal{A}^{\NN}$,
which  factors for all $n \ge 1$ as 
 $$
 \mathcal{A}^\mathbb{N} = (\mathcal{A}^\mathbb{N})^{(\pr n)},
  $$
  It  is a path set, and  Section \ref{sec:6} gives a structural characterization of infinitely
 factorizable path sets.

 The paper \cite{ALS21} gave a classification of the possible pattern
 of interleaving factorizations for a general set $X$, and
 a separate classification  for a general  closed set $X \subset \sA^{\NN}$,
 see Section \ref{subsec:55}.
 For a  closed set $X \subset \sA^{\NN}$, exactly one of the following holds. 
 \begin{enumerate}
\item[(i)]
$X$ is factorizable for all $n \ge 1$.
\item[(ii)]
 $X$ is $n$-factorizable for
a finite set of $n$, which are all the divisors of an integer $f= f(X) \ge 1$. 
\end{enumerate}. 
\noindent All path sets are closed sets so this classification applies to them. 
One can easily show that all allowed patterns of
 interleaving factorizations in this dichotomy  occur for path sets.

The main results of this paper about interleaving factorizations
concern the structure of factorizations of a path set $\sP$.  
They are given in Sections \ref{sec:6} and \ref{sec:factorsec}. 

\begin{enumerate}
\item[(2)]
We  characterize  infinitely factorizable path sets in terms of  the  form of their minimal
right-resolving presentation, which we call ``leveled."  (Theorem \ref{thm:63})
\item[(2)]   For the remaining finitely factorizable path sets, we obtain an upper
bound on the size of $n$ for which $n$-factorability can occur, in terms of the size of their
minimal right-resolving presentation: one has $n \ge m-1$.
Thus $f(\sP) \le m-1$. (Theorem \ref{thm:non_leveled_fact}). 
\item[(3)] We show that the process of iterated interleaving  factorization of any path set always
terminates in finitely many steps if we agree to ``freeze" any infinitely factorizable path set encountered in the
process. (Theorem \ref{thm:complete_fact0}). 
\end{enumerate}

The bound in (2)  leads to an effective algorithm for determining if a path set is infiniitely factorizable or not.
The finiteness result in (3) also follows from (2).
For  general closed sets $X$, the paper \cite{ALS21} gave  examples having 
an infinite  depth tree of recursively refined iterated factorizations.

%%%%%%%%%%%%%%%%%%%%%%%%%%%%%%%%%%%%%%%%%%%%%%%%%%%%%%%%%%%%%%%%%%%%%%%%%%%%%%%%
% SECT 1.5. Prior Work
%%%%%%%%%%%%%%%%%%%%%%%%%%%%%%%%%%%%%%%%%%%%%%%%%%%%%%%%%%%%%%%%%%%%%%%%%%%%%%%
\subsection{Related work}\label{sec:15}

 There is a large literature of related work in automata theory, semigroups and symbolic dynamics;  see the discussion in \cite[Sect. 1.2]{AL14a}.  
 
The path set concept has  previously been studied in automata theory and formal language theory,
given in different terminology.
 In \cite{AL14a} we observed that path sets  are characterized as  those $\omega$-sets recognizable by
a finite  (deterministic) B\"{u}chi automaton having  one initial state and having  every state be a  terminal state.
These sets were   characterized in Perrin and Pin \cite[Chapter III, Proposition 3.9]{PP04}
as the set of (B\"{u}chi) recognizable sets that are closed in the product topology on $\sA^{\NN}$.
The name path set is consistent with the term ``path'' in a finite  automaton
used in Eilenberg \cite[page 13]{Eilenberg74}.

The set of finite initial blocks of a path set $\sP$ forms a 
rational language (also called a regular language)$L(\sP)$  in $\sA^{\star}$, the set of all finite words in the alphabet $\sA$.
The  formal language $L(\sP)$ uniquely
characterizes the path set. 
We may call the set of formal languages 
$$
\sL(\sA)  : = \{ L(\sP): \, \sP \quad  \mbox{ a path set on alphabet} \,\, \sA\}
$$
the set of {\em path set languages}.
These languages may be  characterized  as being the prefix-closed regular languages, see  Appendix \ref{sec:A0}.
The set $\sL(\sA)$ forms 
a  strict subset of all rational  languages on the finite alphabet $\mathcal{A}$.

Special cases of  interleaving operations have  been considered in automata theory and formal languages.
Eilenberg \cite[Chapter II.3, page 20]{Eilenberg74} introduced a notion of
{\em internal shuffle product} $A \coprod B$ of two recognizable sets (= regular language)
which corresponds to $2$-interleaving. In Proposition 2.5 in Chapter II of that book, Eilenberg proved that 
the collection of recognizable sets is closed under internal shuffle product. 

Interleaving operations have been used in coding theory  in various code constructions,
see for example Vanstone and van Oorschot \cite{VO89}, Chapters 5 and 7.

One-sided path sets also appear in studies in aperiodic order. The 1989 paper of de Bruijn \cite[Sect. 5 and 6]{DeB:89}
deals mainly with two-sided infinite connector sequences, but has one-sided ``singular" examples as well. He studied these sequences in
connection with rewriting rules describing inflation and deflation for aperidoic tilings, in particular the Penrose tilings,
topics which he previously studied in \cite{DeB:81a},  \cite{DeB:81b}, \cite{DeB:81c}.

One  may consider extensions of interleaving to infinite alphabets.
The notion  of full one-sided shift based on the product topology does not
give a compact space for  infinite alphabets.
In 2014 Ott, Tomforde, and Willis \cite{OTW14} formulated a definition of full shifts on infinite alphabets
which gives a compact shift space, which may be useful for this purpose.

%%%%%%%%%%%%%%%%%%%%%%%%%%%%%%%%%%%%%%%%%%%%%%%%%%%%%%%%%%%%%%%%%%%%%%%%%%%%%%%%
% SECT 1.4 Contents of paper
%%%%%%%%%%%%%%%%%%%%%%%%%%%%%%%%%%%%%%%%%%%%%%%%%%%%%%%%%%%%%%%%%%%%%%%%%%%%%%%
\subsection{Contents}\label{sec:14}

The contents of the paper are 
as follows:
\begin{enumerate}  
\item
%Section 2
Section ~\ref{sec:prelim} collects together  preliminary results about path sets. In particular, it  shows they are uniquely determined by their set of allowed finite initial blocks.
It gives  an effective algorithm to tell whether  two presentations $(\sG_1, v_1)$ and $(\sG_2, v_2)$ give the same path set.
  It shows the 
uniqueness (up to isomorphism) of minimal right-resolving presentations of a path set $\sP$, a fact which does not parallel the theory for sofic shifts (whose definition of right-resolving
does not require the presentation to be pointed.) 
\item
%Section 3
Section \ref{sec:decimation} presents an
algorithm for finding a presentation of a decimation $\psi_{j,n}(\sP)$ from a presentation of $\sP$.
It proves that all path sets are weakly shift-invariant.
 It shows that the  full set of decimations $\ffD(\sP)$ of a path set $\sP$ is a finite set.
 It gives a upper bound on the size of right resolving presentation of  all decimations $\psi_{j,n}(\sP)$ 
 that depends only
 on the size of the presentation of $\sP$.
\item
%section 4
Section ~\ref{sec:intalgsec} shows that any 
$n$-fold interleaving of path sets is a path set. It gives a construction, the (pointed) interleaving graph product,
which when given presentations of the individual $\sP_i$, yields  a presentation of  the $n$-fold interleaving 
of these path sets.
It  gives examples showing the output presentations need not be
right-resolving, even when the presentations of the $\sP_i$ are minimal right-resolving.

\item
%Section 5
Section ~\ref{sec:5} first reviews results for decimation  and interleaving operations
acting on general sets $X \subset \sA^{\NN}$ which were shown in \cite{ALS21}. 
That paper defines  a hierarchy of closure operations
$X \mapsto X^{[n]}$, the {\em $n$-th interleaving closure} of $X$, and shows that
a set $X$ is $n$-factorizable if and only if $X^{[n]} = X$.
Second,  it restricts to path sets, and  proves  that if $\sP$ is a path set, then
so is $\sP^{[n]}$ for all $n \ge 1$. 

\item
% Section 6
Section ~\ref{sec:5B} recalls results from \cite{ALS21} on the structure of
the allowed sets of possible $n$-factorizations that a general set $X \subset \sA^{\NN}$
may have. It deduces that the set $\sC^{\infty}(\sA)$ of infinitely factorizable path sets is
stable under all decimation and interleaving operations. 
It proves that if a path set $\sP$ is $n$-factorizable, the
minimal right resolving presentation of such $\psi_{j, n}(\sP)$ requires no more nodes than
that of a minimal right-resolving presentation of $\sP$.
\item
%Section 6
Section ~\ref{sec:6} determines  the structure of infinitely factorizable closed subsets $X$ of $\sA^{\NN}$.
 It characterizes infinitely factorizable path sets in two ways: the first is a syntactic property of the infinite words in $\sP$,
 and the second characterizes the  form of their minimal right-resolving presentations.
\item
% Section 7
Section \ref{sec:factorsec} analyzes the structure of factorizations of finitely factorizable path sets. 
An iterated interleaving factorization is complete if all of its factors are either infinitely factorizable
or indecomposable.  It proves that every finitely factorizable path set has at least one  complete factorization. 
\item
% Section 8
Section \ref{sec:concluding}  discusses open questions and further work.
\item
%Appendix A
Appendix ~\ref{sec:A0} discusses path sets from the viewpoint of  automata theory.
\item
%Appendix B
Appendix~\ref{sec:B0} gives  a sufficient condition on a presentation of a path set $\sP$
for all of its interleaving factorizations  to be {\em self-interleaving factorizations}.
An $n$-fold  self-interleaving factorization is  one having  all factors $X_i= Z$  identical,
 for $0 \le i \le n-1$, where $Z$ may depend on $n$.
\end{enumerate} 

\medskip

{\bf Acknowledgments.}
The work of  W. Abram and D. Slonim was  facilitated by the Hillsdale College LAUREATES program, 
done by  D. Slonim under the supervision of W. Abram.

%%%%%%%%%%%%%%%%%%%%%%%%%%%%%%%%%%%%%%%%%%%%%%%%%%%%%%%%%%%%%%%%%%%%%%%%%%%%%%%%
%
% SECT 2. Preliminary Definitions and Results
%
%%%%%%%%%%%%%%%%%%%%%%%%%%%%%%%%%%%%%%%%%%%%%%%%%%%%%%%%%%%%%%%%%%%%%%%%%%%%%%%
\section{Structure and presentations  of  path sets} \label{sec:prelim}

In this section we establish  basic  structural results about path sets, formulated  in the 
  terminology of symbolic dynamics in Lind and Marcus \cite{LM95}.

 We  fix a finite alphabet $\mathcal{A}$, and let $\mathcal{C}(\mathcal{A})$ denote the collection of all path sets on $\mathcal{A}$. 
 We use interchangeably the term  {\em word} or {\em block} to mean  a finite string of consecutive symbols from $\sA$, often
 viewed inside an infinite word.
 
 In this section we  show that path sets have
 a minimal right-resolving presentation, unique up to isomorphism of path set presentations.
 Equivalent results can be found in the automata theory literature, 
 which uses different terminology,  see Appendix B. 
We have included proofs to  provide a self-contained treatment in the language of symbolic dynamics. 
 In particular, we  introduce two  notions  of finite and infinite follower sets and characterize path sets
 in terms of finiteness properties of both finite and infinite follower sets. These notions  are needed
 for  later proofs.

 Path sets also  have an important characterization in terms of automata theory,
which uses different terminology. 
In automata theory path sets $\sP$ 
are  characterized as  those recognizable sets in $\sA^{\NN}$ for non-deterministic B\"{u}chi
automata which are closed sets in its product topology.
We  discuss  automata theory results in  Appendix \ref{sec:A0}. 

%%%%%%%%%%%%%%%%%%%%%%%%%%%%%%%%%%%%%%%%%%%%%%%%%%%%%%
%
% Section 2.1
%
%%%%%%%%%%%%%%%%%%%%%%%%%%%%%%%%%%%%%%%%%%%%%%%%%%%%%%%

\subsection{Closure properties of  path sets}\label{subsec:21}

We recall basic results on the closure of path sets in the symbol topology,
and under set operations,  shown in \cite{AL14a}. 

%%%%%%%%%%%%%%%%%%%%%%%%%%%%%%%%%%%%%%%%%%%%%%%%%%%%%%%
% Theorem 2.1
%%%%%%%%%%%%%%%%%%%%%%%%%%%%%%%%%%%%%%%%%%%%%%%%%%%%%%%
\begin{thm}\label{thm:operations}
(\cite[Theorem 1.2]{AL14a})

(1)  Each path set $\sP$ in $\mathcal{C}(\mathcal{A})$ 
is a closed subset  in the product topology on $\mathcal{A}^\mathbb{N}$.

(2) If $\sP_1$ and $\sP_2$ are path sets, then so is $\sP_1 \cap \sP_2$.

(3) If $\sP_1$ and $\sP_2$ are path sets, then so is $\sP_1 \cup \sP_2$.
\end{thm}

%%%%%%%%%%%%%%%%%%%%%%%%%%%%%%%%%%%%%%%%%%%%%%%%%%%%%%%
% Remark 2.2
%%%%%%%%%%%%%%%%%%%%%%%%%%%%%%%%%%%%%%%%%%%%%%%%%%%%%%%

\begin{rem}\label{rem22} 
  The collection   $\sC(\mathcal{A})$ of all path sets in $\mathcal{A}^\mathbb{N}$
 is not closed under complementation inside  $\mathcal{A}^\mathbb{N}$. 
 See \cite[Example  2.3 ]{AL14a}.
\end{rem}

%%%%%%%%%%%%%%%%%%%%%%%%%%%%%%%%%%%%%%%%%%%%%%%%%%%%%%
%
% Section 2.2
%
%%%%%%%%%%%%%%%%%%%%%%%%%%%%%%%%%%%%%%%%%%%%%%%%%%%%%%%
\subsection{Presentations of path sets}\label{subsec:presentations}

Each path set has infinitely many presentations.  
We recall the following properties such a presentation may have.

%%%%%%%%%%%%%%%%%%
%%% Defn 2.3
%%%%%%%%%%%%%%%%%
\begin{defn}\label{def:properties}

(1) A labeled directed graph $\mathcal{G}$ is called \emph{right-resolving} if any two edges emanating from the same vertex have distinct 
labels.

(2) A  pointed graph $(\mathcal{G},v)$ is called \emph{reachable} if there is a directed path in $\mathcal{G}$ from the initial vertex $v$ to every other vertex of $\mathcal{G}$. 

(3) A pointed graph  $(\mathcal{G},v)$ is called \emph{pruned} if it has no sinks, meaning every vertex has  an exiting edge; i.e., it has out-degree at least one. 
\end{defn}

%%%%%%%%%%%%%%%%%%%%%%%%%%%%%%%%%%%%%%%%%%%%%%%%%%%%%%
% Proposition 2.4
%%%%%%%%%%%%%%%%%%%%%%%%%%%%%%%%%%%%%%%%%%%%%%%%%%%%%%%
\begin{prop}\label{thm:3.2frompaper1} 
Every path set $\sP$ has a  presentation that is right-resolving, pruned and reachable.
\end{prop}

\begin{proof} Theorem 3.2 of \cite{AL14a} gives (in its proof) an algorithm which when given an arbitrary presentation $\sP= X(\sG, v)$ of a path set 
will compute
another presentation $(\sG', v')$ for $\sP$ which is right-resolving. The pruning operation  given in Section 3 of \cite{AL14a}
then iteratively  removes stranded vertices while retaining the right-resolving property. A vertex is {\em stranded} if it either
has no entering edges or  no exit edges, or both. When a stranded vertex is removed, new vertices  may become stranded, and
the operation repeats until no stranded vertices remain.
Finally a pruned, right-resolving presentation can be 
converted to a  right-resolving, pruned and reachable presentation  by further removing all vertices not reachable from $v'$. 
\end{proof}

%%%%%%%%%%%%%%%%%%%%%%%%%%%%%%%%%%%%%%%%%%%%%%%%%%%%%%%
% Remark 2.5 
%%%%%%%%%%%%%%%%%%%%%%%%%%%%%%%%%%%%%%%%%%%%%%%%%%%%%%%

\begin{rem}\label{rem:trim}
A presentation $(\sG, v)$ specifies a finite automaton  in the sense of Eilenberg \cite[Chapter II]{Eilenberg74}, having  a single initial state $I=\{v\}$,
where we impose the requirement that all states be terminal states: i.e., $T= V(\sG)$. (The terminal states are not specified in the path set definition.) 
In automata theory {\em right-resolving} is equivalent to the automaton being {\em deterministic}.
For a single initial state automaton {\em reachable}  is equivalent to being {\em accessible}.
For an  automaton having all states being terminal states, {\em pruned} is equivalent to being {\em co-accessible.}
Therefore the  presentation produced by Proposition \ref{thm:3.2frompaper1} is  a {\em trim} automaton,
i.e. deterministic, accessible and co-accessible.
\end{rem}

 %%%%%%%%%%%%%%%%%%%%%%%%%%%%%%%%%%%%%%%%%%%%%%%%%%%%%%
%
% Section 2.3
%
%%%%%%%%%%%%%%%%%%%%%%%%%%%%%%%%%%%%%%%%%%%%%%%%%%%%%%%
\subsection{Word follower sets and vertex follower sets}\label{subsec:23} 

The internal structure of presentations of path sets is determined by  their patterns of initial words over all paths.
We  formulate two notions of  follower set which capture this internal structure. 
These definitions are adapted from  definitions in Lind and Marcus  \cite{LM95} for two-sided infinite sequences.

The first of these notions  applies to 
general subsets $X$ of the one-sided shift, parallel to \cite[Defn. 3.2.4]{LM95}.

%%%%%%%%%%%%%%%%%%%%%%%%%%%%%%%%%%%%%%%%%%%%%%%%%%%%%%
% Defn. 2.6
%%%%%%%%%%%%%%%%%%%%%%%%%%%%%%%%%%%%%%%%%%%%%%%%%%%%%%%
\begin{defn} \label{def:symbolic_follower}
(Word follower set) 
Let $X$ be a subset of the one-sided shift $\msrA^{\NN}$, and let $w=b_0b_1\cdots b_{k-1}$ be
a finite word of length $|w|= k$, allowing the empty word $\emptyset$ of length $0$. 

(1) The {\em word follower set  $\mitF_X(b)$ of an initial finite word $b$} of $X$ 
is the set of all finite blocks 
$$ \mitF_{X}(b) = \{ a\, |~ a \text{ is a finite block
such that} \,\, ba \,\, \text{is an initial block of some} \,\,x= ba x' \in X\}.$$

(2) A set $\mitF \subseteq \msrA^{\NN}$ is a {\em word follower set of $X$}  if there exists some initial block $b$ of $\mathcal{P}$ such that $\mitF= \mitF_X(b) $. 
\end{defn}

A   closed set $X$ in $\sA^{\NN}$ may possess  infinitely many different word follower sets $\mitF_X(b)$, as $b$ varies.
We  show below that   path sets $\sP=X(\sG, v)$   have only finitely many different word follower sets as $b$ varies. 
Note that the initial block set $\sB^{I}(\sP) = \mitF_{X} (\emptyset)$.  

\smallskip

The  second of these notions applies  to  presentations  $(\sG, v)$ 
 of  path sets $\sP$, parallel to  \cite[Defn. 3.3.7]{LM95}.

%%%%%%%%%%%%%%%%%%%%%%%%%%%%%%%%%%%%%%%%%%%%%%%%%%%%%%
% Defn. 2.7
%%%%%%%%%%%%%%%%%%%%%%%%%%%%%%%%%%%%%%%%%%%%%%%%%%%%%%%
\begin{defn} \label{def:vertex_follower} 
(Vertex follower set)
 The  \emph{vertex follower set  $\vF(\sG,  v')$  of a vertex $v'$}
 in a labeled directed graph $\mathcal{G}$ 
 is the set of all finite words $b=b_0b_1 \ldots b_{k-1}$ that can be presented by labels of paths on $\mathcal{G}$ beginning at vertex $v'$.
\end{defn}

The next result shows for a right-resolving presentation $(\sG, v)$ 
of a path set $\sP$ that all its possible  word follower sets 
occur as vertex follower sets of the directed labeled graph $\sG$.

%%%%%%%%%%%%%%%%%%%%%%%%%%%%%%%%%%%%%%%%%%%%%%%%%%%%%%
% Proposition 2.8
%%%%%%%%%%%%%%%%%%%%%%%%%%%%%%%%%%%%%%%%%%%%%%%%%%%%%%%
\begin{prop}\label{onevertexforeachfollowerset}
Let $(\mathcal{G}, v) $ be a right-resolving, 
%reachable and 
pruned presentation of a path set $\mathcal{P}$, with $\mathcal{G}$
having $m$ vertices. 

(1) For each finite initial word $b \in \sB^{I}(\sP)$,
the word follower set $\mitF_{\sP}( b)$ equals the vertex follower set $\mitF(\sG, v')$ for some reachable vertex  $v'$ of $\mathcal{G}$. 
In particular, there are at most $m$ different word follower sets.

(2) Conversely, each vertex follower set $\mitF(\sG, v')$
of a reachable vertex occurs as the word follower set $\mitF_{\sP}( b)$ of some initial word $b \in \sB^{I}(\sP)$.
The word $b$ can be chosen to be  of length at most $m-1$. If $v'=v$ we choose $b = \emptyset$.
\end{prop}

\begin{proof}
(1) Let $b= b_1 b_2 \cdots b_k$ be a finite initial word of $\mathcal{P}$, and let $\mitF_{\sP}(b)$ be its word follower set. Since $\mathcal{G}$ is right-resolving, 
there is exactly one path on $\mathcal{G}$ beginning at
 $v$ presenting  label set  $b$. Let $v'$ be the (reachable) vertex where this path terminates. A finite word $w$ will be 
 in the word follower set of $b$ if and only if there is a directed path beginning at vertex $v'$ 
 that presents the label set  $w$, and since ($\mathcal{G}, v)$ is pruned, there exists  an infinite word $x= bwx' \in \sP$. 
 Consequently the word follower set $\mitF_{\sP}( b)$  coincides with  the vertex follower set $\mitF(\sG, v')$. 

(2)  Let $v'$ be a reachable vertex in $(\mathcal{G}, v)$, so that
there exists   a directed path $\pi$ from $v$ to $v'$, which can be chosen to have length at most $m-1$,
since there are $n$ vertices in $\sG$.
 Let $b$ be the block labeling this path, which uniquely determines $v'$ since the presentation is
 right-resolving.  As above, since $\mathcal{G}$ is pruned,  the vertex follower set $\mitF(\sG, v')$ equals the word follower set $\mitF_{\sP}( b)$.
 \end{proof}

%%%%%%%%%%%%%%%%%%%%%%%%%%%%%%%%%%%%%%%%%%%%%%%%%%%%%%
%
% Section 2.4 Finiteness follower
%
%%%%%%%%%%%%%%%%%%%%%%%%%%%%%%%%%%%%%%%%%%%%%%%%%%%%%%%
\subsection{Finiteness of follower sets for path sets}\label{sec24}

Proposition \ref{onevertexforeachfollowerset} 
 implies  the finiteness of the number of distinct word follower sets for any path set.

%%%%%%%%%%%%%%%%%%%%%%%%%%%%%%%%%%%%%%%%%%%%%%%%%%%%%%
% Theorem  2.9
%%%%%%%%%%%%%%%%%%%%%%%%%%%%%%%%%%%%%%%%%%%%%%%%%%%%%%%
\begin{thm}\label{thm:finite_symbolic}
{\em ($\sC(\sA)$ characterized by finiteness of  word follower sets)}
A set $X \subseteq \mathcal{A}^{\NN}$ is a path set if and only if it
is closed and has a finite number of distinct word follower sets.
\end{thm}

\begin{proof}
Suppose first that $X= \sP$ is a path set. Then  
 $\sP$ is  closed  by Theorem \ref{thm:operations}.  
Now  $\sP$  has a right-resolving, pruned presentation $(\sG, v)$  by Theorem \ref{thm:3.2frompaper1};  let $n$ be the
number of its  vertices.  Proposition \ref{onevertexforeachfollowerset}  implies it has at most $n$ distinct word follower sets
$\mitF_{X}(b)$ as $b$ varies. Thus  $X$ is closed and has a finite number of distinct word follower sets. 

%To show  the ``if'' direction, 
Conversely, suppose $X$ is closed and has a finite number $n$ of distinct word follower sets.
We  construct a   right-resolving presentation $X(\sG, v)$ of a path set  having $n$ vertices, and check that $X= X(\sG, v)$. 
This construction will yield  a minimal right-resolving presentation of $X$.
We give a name to each follower set in the finite list by assigning to it the minimal prefix $b$ defining it as $\mitF_{\sP}(b)$.
(We may put a total order on the alphabet $\msrA$ and use the lexicographic  order on prefixes to define ``minimal''.) 
These word follower sets $\mitF_{\sP}(b)$ will name the states of $\sG$. 
For $b= \emptyset$ we select $v= \mitF_{\sP}(\emptyset) = \sB^{I}(\sP)$ to  be the initial vertex of $\sG$. For each vertex 
$v_i := \mitF_{\sP} (b_i)$ and 
each letter
$a$ that occurs as an initial letter of some word of the follower set $ \mitF(\sP, b_i)$,    we add an exit edge with label $a$
which goes from $v_i$ to the follower set $\mitF_{\sP}(b_i a)$. Since the vertices enumerate the complete list of possible word follower
 sets, it 
 will be some vertex  $v_j= \mitF_{\sP}( b_j)$ for some $b_j$. 
(We have $\mitF_{\sP}( b_ia)=\mitF_{\sP}(b_j)$ but we may have $b_ia \ne b_j$.)  We have constructed $(\sG, v)$, and
it is right-resolving as there is at most one exit edge with a given symbol from each vertex. It is also pruned and
reachable by construction.

It remains to  show that  $X= X(\sG, v)$, which will verify  that it is a path set. 

(1) We show  the inclusion $X \subseteq X(\sG, v)$.  Take $x=a_0a_1a_2\cdots\in X$
and construct a path in $(\sG, v)$ realizing this symbol sequence. Given a finite sequence $b_k=a_0a_1\ldots a_k \in \sB^{I}(X)$,
we prove by induction on $k \ge 0$, that at the $k$-th step finds a path in $(\sG, v)$ moving  from vertex labeled $\mitF_X(b_{k-1})$  to vertex labeled $\mitF_{X}( b_{k})$
with edge symbol $a_k$. Here $b_{-1} = \emptyset$ is the base case, and both the base case and the induction step follow
by the definition of edges in $\sG$.

(2) We  show the reverse inclusion $X(\sG, v) \subseteq X$ using  the fact  that $X$ is closed. 
Each infinite symbol sequence $y= a_0a_1... a_k\cdots $ in $X(\sG, v)$, starting at vertex labeled $\mitF_{X}(\emptyset)$, 
appears on a unique path (by right-resolving property) which at step $k$ is at vertex corresponding to $\mitF_{X}(a_0a_1\cdots a_k)$ of $\sG$.  
We may prove by induction on $k$ that the finite path $b_k=a_0a_1 \ldots a_k$ is  an initial block in $X$.
The word follower set property of $\mitF_{X} (a_0a_1\cdots a_k)$ permits inductively adding the symbol $a_{k+1}$, for both
the base case  $k= -1$ and the induction step.
Since $X$ is closed by hypothesis, the infinite word $y$ belongs to $X$.
   \end{proof}

%%%%%%%%%%%
%Corollary   2.10
%%%%%%%%%%
\begin{cor}\label{thm:initialblocks}
{\rm (Path sets are  characterized by initial words)} 
A path set $\sP$ is characterized by its  set $ \sB^{I}(\sP)$ of all finite initial words. That is, if two path sets have the same set of initial words, then they are identical.
\end{cor}

\begin{proof}
The limit set of the set of initial words $ \sB^{I}(X)$ of an arbitrary set $X\subseteq \sA^{\NN}$ is its 
topolgical closure $\overline{X}$. Since path sets $\sP$ are closed sets by Theorem \ref{thm:finite_symbolic}, they are detemined by  $\sB^{I}(\sP)$. 
\end{proof} 

%%%%%%%%%%%%
% Remark 2.11
%%%%%%%%%%%%%
\begin{rem}\label{rem:sofic_characterize}
 The characterization of Theorem \ref{thm:finite_symbolic} parallels a characterization of sofic shifts found by 
Ashley, Kitchens and Stafford \cite{AKS92} (cf. \cite[Theorem B.1]{AL14a}) in 1992,  
which assumes the extra condition of  one-sided  shift-invariance: 
 {\em Any  shift-invariant subset $X$ of $\msrA^{\NN}$ is a sofic shift if and only if it has only finitely many different word
follower sets $\mitF(X, b)$.}  
\end{rem}

%%%%%%%%%%%%%%%%%%%%%%%%%%%%%%%%%%%%%%%%%%%%%%%%%%%%%%
%
% Section 2.5
%
%%%%%%%%%%%%%%%%%%%%%%%%%%%%%%%%%%%%%%%%%%%%%%%%%%%%%%%
\subsection{Minimal presentations of path sets}\label{subsec:minimal_presentations}

%%%%%%%%%%%%%%%%%%%%%%%%%%%%%%%%%%%%%%%%%%%%%%%%%%%%%%%
% Definition 2.12. minimal presentation
%%%%%%%%%%%%%%%%%%%%%%%%%%%%%%%%%%%%%%%%%%%%%%%%%%%%%%%
\begin{defn} \label{defn:minimal-pres} 
(1) A {\em minimal presentation} for a path set $\sP$ is a presentation with a minimal number of vertices.

(2) A {\em minimal right-resolving presentation}
is a right-resolving presentation having a minimal number of vertices among all right-resolving presentations. 
\end{defn}

A minimal presentation is always pruned and reachable, but it need not be right-resolving. Minimal right-resolving
presentations are sometimes not minimal presentations.

%%%%%%%%%%%%
%theorem 2.13
%%%%%%%%%%%%%
\begin{thm}\label{thm:minimal-to-RRminimal}
Let $\sP$ be a path set having a  minimal presentation having  $m$ vertices. Then $\sP$ has a minimal
right-resolving presentation having at most $2^m-1$ vertices. 
\end{thm}

\begin{proof} This result is proved by the well-known subset construction in the automata theory literature,
for obtaining a minimal deterministic finite state automaton matching a finite state deterministic automaton.
It appears \cite[Chapter III, Sect. 5, Theorem 5.2]{Eilenberg74}. 
\end{proof}

Below we characterize minimal right-resolving presentations, showing
uniqueness in the process.
We first recall a  definition from symbolic dynamics, which is a one-sided shift version of a definition in  Lind and Marcus \cite[page 71, page 78]{LM95}

%%%%%%%%%%%%%%%%%%%%%%%%%%%%%%%%%%%%%%%%%%%%%%%%%%%%%%%
% Definition 2.14 (Follower separated}
%%%%%%%%%%%%%%%%%%%%%%%%%%%%%%%%%%%%%%%%%%%%%%%%%%%%%%%
\begin{defn}\label{def:follower_separated} 
 A directed labeled graph $\mathcal{G}$ 
 is called {\em follower-separated}  if all vertices have distinct vertex follower sets.
\end{defn}

In characterizing  the number of states in a minimal right-resolving presentation, we 
formulate two definitions that view path sets as ``infinite follower sets''.

%%%%%%%%%%%%%%%%%%%%%%%%%%%%%%%%%%%%%%%%%%%%%%%%%%%%%%%
% Definition 2.15  (Word path set; vertex path set}
%%%%%%%%%%%%%%%%%%%%%%%%%%%%%%%%%%%%%%%%%%%%%%%%%%%%%%%
\begin{defn}\label{def:word_vertex_path_set}
(1) Given a path set $\sP$ and a finite word $w \in \sA^{\star}$, the {\em word path set } $\sP^w$  of $\sP$ is the set of all infinite
words $x$ such that $wx \in \sP$. 

(2) Given a  presentation $(\sG, v)$ of a path set $\sP$, an associated {\em vertex path set} is any path set $X(\sG, v')$
where $v'$ is a vertex of $\sG$.
\end{defn}

A nonempty word path set $\sP^w$ is a path set. For a right-resolving presentation $(\sG, v)$ of $\sP$ an 
argument parallel to  Proposition \ref{onevertexforeachfollowerset} (1) 
shows that $\sP^{w}$  equals the vertex path set $X(\sG, v')$ for the vertex $v'$ of $\sG$ that is 
the final vertex on  the unique path with symbol labels $w$ from the initial vertex $v$.

%%%%%%%%%%%%%%%%%%%%%%%%%%%%%%%%%%%%%%%%%%%%%%%%%%%%%%
% Theorem  2.16
%%%%%%%%%%%%%%%%%%%%%%%%%%%%%%%%%%%%%%%%%%%%%%%%%%%%%%%
\begin{thm}\label{thm:minimal_presentation}
{\rm (Minimal right-resolving presentation)} 

(1) A path set $\sP$ has a minimal right-resolving presentation $(\sG, v)$, which  is unique up to
isomorphism of pointed labeled graphs. This presentation  is pruned, reachable and follower-separated.

(2) Conversely, if a right-resolving presentation of $\sP$ is pruned, reachable and follower-separated,
then it is minimal. 

(3) The number $m$ of vertices in a minimal right-resolving presentation of $\sP$ is the number of distinct 
word follower sets 
$\mitF_{\sP}(b)$
 of $\sP$. It also equals the number of distinct vertex follower sets $\mitF(\sG, v)$ of $\sP$ in any right-resolving, reachable presentation
 $\sP =X(\sG, v)$.

(4) The number $m$ of vertices in a minimal right-resolving presentation of $\sP$ is the number of distinct nonempty
word path sets $\sP^{w}$ of $\sP$.  It also equals the number of distinct nonempty vertex path sets $X(\sG, v')$  in any right-resolving presentation
$(\sG, v)$ of $\sP$.
\end{thm}

\begin{proof}
(1) Proposition \ref{onevertexforeachfollowerset} implies that the number of vertices of any right-resolving
presentation of $\sP$ must equal or exceed the number of distinct  word follower sets $\mitF_{\sP}( b)$.
The proof of Theorem \ref{thm:finite_symbolic} constructed a right-resolving presentation $(\sG, v)$ for $\sP$,
which has one vertex for each distinct word follower set,
which must therefore be minimal. By construction it is pruned and the  vertex follower sets in this presentation
are distinct, so it is follower-separated.

It remains to show uniqueness. We know that any minimal right-resolving presentation necessarily
has  vertices  labeled by  all of the possible word follower sets. 
The exit edges from the pointed vertex $v$ have different labels $a'$ (from right-resolving property) and
the edge labeled $a'$ must go to the vertex labeled by  the  word follower set $\mitF_{\sP} (a')$.  
This assignment  is the only way  to permit  the initial  word follower set $\mitF_{\sP}( \emptyset)$
to reach all words for it that begin with prefix $a'$.
Similarly the  exit edges from each vertex $\mitF_{\sP}( b)$ must take the allowed
prefix labels $a'$ of words in $\mitF_{\sP}( b)$ and for each $a'$ must map to vertex 
with follower set label $\mitF_{\sP} (ba')$.  Every labeled edge is forced, so the construction is unique. 

(2) Suppose that a  right-resolving presentation $(\sG, v)$ of $\sP$ is pruned,  follower-separated and reachable.
By Proposition \ref{onevertexforeachfollowerset} ( 2) each  vertex follower sets is a symbolic
follower set. 
By Proposition \ref{onevertexforeachfollowerset} (1)
the vertex follower sets  includes all distinct  word follower sets. 
The follower-separation property implies
each distinct follower set occurs exactly once  in $\sG$ as a vertex follower set, so 
$(\sG, v)$ is minimal.

(3) The two assertions are a consequence of Proposition  \ref{onevertexforeachfollowerset}
saying that every vertex follower set of every right-resolving presentation is some word
follower set, and that all distinct word follower sets appear 
as vertex follower sets in every right-resolving presentation.

(4) The two assertions follow from (3) using Theorem \ref{thm:initialblocks},
which implies  that word path sets are uniquely determined by their  word follower sets, and vice-versa.
Similarly  vertex path sets are uniquely determined by their vertex follower sets, and vice-versa.
\end{proof}

%%%%%%%%%%%%
% Remark 2.17
%%%%%%%%%%%%%
\begin{rem}\label{rem:minimal-not}
For  any  minimal presentation of $\sP$ that has  fewer vertices than  in a minimal right-resolving presentation, 
there will necessarily be  word follower sets $\sP^w$ that are not vertex follower sets of such a presentation. 
\end{rem} 

%%%%%%%%%%%%%%%%%%%%%%%%%%%%%%%%
% Remark 2.18
%%%%%%%%%%%%%%%%%%%%%%%%%%%%%%%%
\begin{rem}  
%It was observed in \cite{AL14a} that  
For sofic shifts  minimal right-resolving presentations (in the two-sided sofic shift sense, which is not the path set sense)  
are not necessarily unique. A counterexample is given in \cite[Example 3.3.21]{LM95}.
\end{rem}

%%%%%%%%%%%%%%%%%%%%%%%%%%%%%%%%%%%%%%%%%%%%%%%%%%%%%%%%%%%%%
%
% 2.6  Path set identity algorithms
%
%%%%%%%%%%%%%%%%%%%%%%%%%%%%%%%%%%%%%%%%%%%%%%%%%%%%%%%%%%%%%
\subsection{Recognizing and distinguishing  path sets} \label{sec:recognition}

We describe algorithms  for testing identity of path sets and for finding minimal right-resolving presentations. 
They are based on the following effective bound for telling when two given vertex follower sets
in a (possibly disconnected) presentation are equal. 

%%%%%%%%%%%
%Proposition 2.17
%%%%%%%%%%
\begin{prop}\label{thm:exercise}
Let $\mathcal{G}$ be a right-resolving pruned labeled graph,  not necessarily connected, that has $m$ vertices.
 If two vertices, $v_1$ and $v_2$, have distinct vertex follower sets $\mitF(\sG, v_i)$, then there is some word $w=a_1a_2\cdots a_r$ of length $m$ in $\msrA$ which belongs to 
 exactly one of the two follower sets $\sF(\mathcal{G}, v_1)$ and $\sF(\mathcal{G}, v_2)$. 
 \end{prop}

\begin{proof}
The proof of this theorem\footnote{One   can improve the length bound on $w$ to $m-1$, and there exist examples showing that the upper bound $m-1$ is the best possible.} 
is outlined in Exercise 3.4.10 of \cite{LM95}.
Similar results appear in Conway \cite[Chapter 1, Theorems 6 and 7]{Co71}. 
\end{proof}

%%%%%%%%%%%
%Proposition 2.18
%%%%%%%%%% 
\begin{prop}\label{cor:n+m}

Let $\mathcal{P}_1= X(\sG_1, v_1)$ and $\mathcal{P}_2=X(\sG_2, v_2)$ be path sets with right-resolving presentations, 
where $\mathcal{G}_1$ has $m_1$ vertices and $\mathcal{G}_2$ has $m_2$ vertices. Then
 $\mathcal{P}_1=\mathcal{P}_2$ if and only if $\mathcal{P}_1$ and $\mathcal{P}_2$ share the same set of initial $(m_1+m_2)$-blocks; i.e., $\sB^{I}_{m_1+m_2}(\sG_1, v_1) = \sB^{I}_{m_1+m_2}(\sG_2, v_2)$. 
\end{prop}

\begin{proof}
This result follows  directly from Proposition ~\ref{thm:exercise}. We form a (disconnected) graph $\sG = \sG_1 \sqcup \sG_2$, which has $m_1+m_2$ vertices.
The contrapositive of Proposition  ~\ref{thm:exercise} says that the follower sets $\sF(\sG, v_1)$ and $\sF(\sG. v_2)$ are equal if and only if they contain the same set of words of length $m_1+m_2$; i.e.,
if $\sF_{m_1+m_2}(\sG, v_1) = \sF_{m_1+m_2}(\sG, v_2).$ Because the graph $\sG$ is disconnected in two pieces, these follower sets are $\sF(\sG, v_1)=\sB^{I}(\sG_1, v_1)$ and $\sF(\sG. v_2)=\sB^{I}(\sG_2, v_2)$. 
For the same reason, we have that the length $m_1$ follower sets are $\sF_{m_1+m_2}(\sG, v_1) = \sB^{I}_{m_1+m_2}(\sG_1, v_1)$
and $\sF_{m_1+m_2}(\sG, v_2)= \sB^{I}_{m_1+m_2}(\sG_2, v_2)$, whence $\sB^{I}_{m_1+m_2}(\sG_1, v_1)= \sB^{I}_{m_1+m_2}(\sG_2, v_2)$.
Therefore we conclude that the latter equality implies equality of the initial follower sets $\sB^{I}(\sG_1, v_1) = \sB^{I}(\sG_2, v_2)$. 
By Theorem \ref{thm:initialblocks} we conclude $\sP_1=\sP_2$. 
\end{proof}

%%%%%%%%%%%
%Proposition 2.19
%%%%%%%%%%
\begin{prop}\label{thm:identity}
{\rm (Testing Identity of Path Sets)} 
There is an effective algorithm which when given two pointed graphs 
$(\sG_1, v_1)$ and $(\sG_2, v_2)$ determines whether the path sets $\sP_1=X(\sG_1, v_1)$ and $\sP_2= X(\sG_2, v_2)$
are identical. 
 \end{prop}

\begin{proof}
Proposition  \ref{cor:n+m} yields an effective algorithm to tell if two path sets $\sP_1= X(\sG_1, v_1)$ and $\sP_2=X(\sG_2, v_2)$ are equivalent. 
We first use the method of \cite[Theorem 3.2]{AL14a}
to convert the  given presentations to pointed graphs $(\sG_1^{'}, v_1^{'})$ and $( \sG_2^{'}, v_2^{'}) $ that are right-resolving and reachable.
Suppose these two graphs $\sG_1^{'}$ and $\sG_2^{'}$   have  $m_1$ and $m_2$ vertices respectively.
It now  suffices to exhaustively determine all members of the  finite sets $\sB^{I}_{m_1+m_2}(\sG_1^{'}, v_1^{'})$ and $\sB_{m_1+m_2}^{I}(\sG_2^{'}, v_2^{'})$
of initial blocks of length $m_1+m_2$ by tracing paths through the graphs, and to check whether these sets are identical. 
\end{proof}

%\%%%%%%%%%%%%%%%%%%%%%%%%%%%%%%%%%%%%%%%%%%%%%%%%%%%%%%%%%%%%%
%
% SECT 3. Decimations of Path Sets
%
%%%%%%%%%%%%%%%%%%%%%%%%%%%%%%%%%%%%%%%%%%%%%%%%%%%%%%%%%%%%%%%%%%%%%%%%%%%%%%%
\section{Decimations  of  path sets} \label{sec:decimation}

We study the effect of  decimation operations on  path sets.  The following result was originally 
established as Theorem 1.5 in \cite{AL14a}.

%%%%%%%%%%%%%%%%%%%%%%%%%%%%%%%%%%%%%%%%%%%%%%%%%%%%%%%
% Theorem 3.1
%%%%%%%%%%%%%%%%%%%%%%%%%%%%%%%%%%%%%%%%%%%%%%%%%%%%%%%
\begin{thm}\label{thm:decimation} 
{\rm ( $\sC(\sA)$ is closed under decimation) }
If  $\sP \in \sC(\sA)$ is a path set, then  for any $(j, n)$ with $j \ge 0$ and $n \ge 1$, 
the $j$-th decimation set $\sP_{j,n}$ of $\sP$ at depth $n$, given by 
$$
\sP_{j,n} = \psi_{j, n}(\sP) := \bigcup_{ \bx \in \sP} \{  \psi_{j,n}(\bx) \}, 
$$
is a path set. 
\end{thm}

The  proof  given in \cite{AL14a} formulated an algorithm
which when given as input a right-resolving presentation $(\sG, v)$ of a path set  $\sP$, 
and the values $(j,n)$ produced  as output a  (not necessarily right-resolving)  presentation $(\sG_{j,n} , v_{j,n})$ for
 the $j$-th decimation at level $n$ of $\sP$, $\psi_{j,n}(\sP)$,   for $0 \le j \le n-1$.
The algorithm was outlined in the discussion in  \cite[Section 7]{AL14a}. It has the feature  that 
number of vertices of the output presentation it produces can be much larger than the number of vertices
in the input presentation.   

Here we present algorithms which produce  presentations
of $\psi_{j,n}(\sP)$ which are smaller: they  increase the number of vertices of the input presentations by at most $1$.

The paper \cite{AL14a} showed that any iterated shift $S^j(\sP)$ of a path set is a path set.
In Section \ref{subsec:31} we give an algorithm which shows that from any presentation of a path
set with $m$ vertices one can constructively find a presentation of any $S^j(\sP)$ having at
most $m+1$ vertices. Note that $S^j(\sP) = \psi_{j,1}(\sP).$
In Section \ref{subsec:32N}  we  present  a second constructive  algorithm that
 finds  a presentation of $\psi_{j,n}(\sP)$ for $0 \le j \le n-1$, the higher power presentation,
 having no more than $m$ vertices. Combining it with
 the algorithm of Section \ref{subsec:31} we obtain presentation for each $\psi_{j, n}(\sP)$ with $j \ge n$.
 In Section \ref{subsec:33N} we use this result to
prove finiteness of the set of all decimations $\psi_{j,n}(\sP)$ of a path set $\sP$.

For general sets $X \subseteq \mathcal{A}^{\mathbb{N}}$, we will term the decimations $\psi_{j,n}(X)$ with $0 \le j \le n-1$ {\em principal decimations} and call the remaining  $\psi_{j,n}(X)$
with $j \ge n$ {\em subsidiary decimations}. This terminology reflects the fact that, when  acting on a single word $a_0a_1\cdots$,
the principal decimations at level $n$  supply
enough information to reconstruct $X$  word by word, using the identity \eqref{eqn:basic-relation}. 

%%%%%%%%%%%%%%%%%%%%%%%%%%%%%%%%%%%%%%%%%%%%%%%%%%%%%%%%%%%%%%
% SUBSECT 3.1 . (Iterated Shift algorithm. 
%
%%%%%%%%%%%%%%%%%%%%%%%%%%%%%%%%%%%%%%%%%%%%%%%%%%%%%%%%%%%%%
\subsection{Iterated shift operators on path sets}\label{subsec:31}

We have $\psi_{0,1}(X) = X$ and $\psi_{j, 1}(X) = S^j (X)$, where $S^j$ is the $j$-fold iteration of the left shift operator,
which operates on individual symbol sequences $a_0a_1a_2 \cdots \in \sA^{\NN}$ by
\begin{equation*}
\rS^j( a_0a_1a_2a_3 \cdots ) = a_j a_{j+1}a_{j+2} a_{j+3} \cdots.
\end{equation*}
 
%%%%%%%%%%%%%%%%%%%%%%%%%%%%%%%%%
% THM 3.2 
%%%%%%%%%%%%%%%%%%%%%%%%%%%%%%
\begin{thm}\label{thm:decimation_alg}
{\rm (Iterated shift operator)}
Given a path set with presentation $\sP = X(\sG, v)$ having $m$ vertices, which is reachable.   Then for  each $j \ge 1$  there exists a presentation
of the $j$-th iterated shift 
$$\rS^j \sP := \psi_{j,1}(\sP) = X(\sG_{j,1}, v)$$
 which has   at most $m+1$ vertices. 
\end{thm} 

\begin{proof} 
The path set $\rS^j(\sP)$ is exactly the set of infinite words in $\sG$ emanating from the set $V^{(j)}(\sG,v)$ of vertices of $\sG$  that can be reached
from the initial vertex $v$  after traversing  a  path with $j$ edges. We create a new graph $\sG_{j,1}$  from $\sG$ by adding a new vertex $w$,
so that $V(\sG_{j,1} )= V(\sG) \cup \{ w\}$. 
The directed labeled graph $\sG_{j,1}$  has the same directed labeled edges as $\sG$ on the vertices $V(\sG)$, 
the new vertex $w$ has no entering edges and  is defined to have labeled exit edges from $w$ to vertex $v_2 \in V(\sG)$ whenever there is 
an entering edge $v_1 \to v_2$ from some $v_1 \in V(\sG)$ in having the given label. 
Any duplicate labeled edges obtained this way are to be discarded. The new vertex will be the marked vertex in
the presentation $ X(\sG', w)$.
We claim that  the presentation $ X(\sG', w)$ is $\psi_{j,1}(\sP)$. Indeed, after one step each path from $w$ enters $\sG$ and stays 
there forever after.
\end{proof} 
%%%%%%%%%%%%%%%%%%%%%%%%%%%%%%%%%%%%%
% Remark 3.3
%%%%%%%%%%%%%%%%%%%%%%%%%%%%%%%%%%%%%%%%%
\begin{rem}
The presentation $\rS^j(\sP) =X(\sG_{j,1}, w)$ obtained in this construction need not be right-resolving; there may be multiple edges with
the same label emanating from $w$. 
\end{rem}

%%%%%%%%%%%%%%%%%%%%%%%%%%%%%%%%%%%%%%%%%%%%%%%%%%%%%%%
% Definition 3.4
%%%%%%%%%%%%%%%%%%%%%%%%%%%%%%%%%%%%%%%%%%%%%%%%%%%%%%%
\begin{defn}  \label{weakshift} (Shift-invariance, weak shift-invariance, weak shift-stability)  

(1) A set $X \subset \sA^{\NN}$ is {\em shift-invariant } if $SX = X$.

(2)  A set $X$ is  {\em weakly shift-invariant}  if  there are integers  $k > j \ge 0$ such
that the iterated shifts  $S^kX = S^jX$. 

(3) A set $X \subseteq \sA^{\NN}$ is {\em weakly shift-stable}  if there are integers  $k > j \ge 0$ such that $S^k X \subseteq S^j X$.
\end{defn} 

The concept of weak shift-stability was introduced and studied in \cite{ALS21}. 
Weak shift-invariance implies weak shift-stability of a set $X$.

%%%%%%%%%%%%%%%%%%%%%%%%%%%%%%%%%
% THM 3.5 
%%%%%%%%%%%%%%%%%%%%%%%%%%%%%%
\begin{thm}\label{thm:weak-shift-invariant}
{\rm (Weak shift-invariance of path sets)}
For any path set $\sP$ there exist integers $k > j \ge 0$
giving the  equality  of iterated shifts
$S^k \sP =  S^j \sP.$
That is, all path sets $\sP$ are  weakly shift-invariant. 
\end{thm} 

\begin{proof}
Given  $\sP$, take a reachable presentation for it, letting $m$ be its number of vertices.  
According to Theorem \ref{thm:decimation_alg}, all iterated shifts $S^j(\sP)$ for $j \ge 1$ have
 reachable presentations (on the same alphabet) having  at
most $m+1$ vertices. The number of such presentations is finite. 
(There are at most $(m+1)(2^{|\sA|(m+1)(m+2)}$ of them,
noting that presentations do not allow multiple directed edges with the same label between two vertices.)
By the pigeonhole principle, there must exist two integers $0\le j<k$ giving the same presentation, 
so  $S^k \sP =  S^j \sP$.
\end{proof} 

%%%%%%%%%%%%%%%%%%%%%%%%%%%%%%%%%%%%%%%%%%%%%%%%%%%%%%%%%%%%%%
% SUBSECT 3.2 . (formerly 5.3) Decimation algorithm
%
%%%%%%%%%%%%%%%%%%%%%%%%%%%%%%%%%%%%%%%%%%%%%%%%%%%%%%%%%%%%%
\subsection{Higher power presentation for decimations of path sets}\label{subsec:32N}

We  present an algorithm which constructs from a given  presentation of $\sP= X(\sG, v)$ 
a  presentation $\sG_{j,n}$ of $\psi_{j,n}(\sP)$ for principal decimations, 
called here the {\em modified $n$-th higher power presentation},   
which has the vertex bound  $|V(\sG_{j,n})| \le   |V(\sG)|$.
It is based on the well known  {\em $n$-th higher power construction},  cf. \cite[Sect. 1.4]{LM95}. 
which presents $\sP$ in blocks using labels from a larger symbol alphabet $\sA^n$. 
The modified algorithm replaces   the $n$-block word labels produced by this construction to labels using the original
alphabet $\sA$ in such a way as to obtain a presentations $\sG_{j,n}$ of all  principal decimations $\psi_{j,n}(\sP)$, for $0 \le j \le n-1$.
We then apply  the shift construction in Theorem \ref{thm:decimation_alg} to get a presentation  of each  $\psi_{j,n}(\sP)$
for $j \ge n$ having at most one extra vertex: $|V(\sG_{j,n})| \le   |V(\sG)| +1$.

%%%%%%%%%%%%%%%%%%%%%%%%%%%%%%%%%
% THM 3.6 
%%%%%%%%%%%%%%%%%%%%%%%%%%%%%%
\begin{thm}\label{thm:decimation_alg2}
{\rm (Higher powers of a path set)}
Given a path set with presentation $\sP = X(\sG, v)$ on alphabet $\sA$.  For  $n \ge 2$  there exists a presentation
$\psi_{j, n}(\sP) = X(\sG_{j.n}, v)$ of the $(j, n)$-th decimation of $\sP$, for $0 \le j \le n-1$, such that each $\sG_{j,n}$ has the same vertex set
as $\sG$ and has the same marked vertex $v$. 
 For $j \ge n$ there exists  a presentation $\psi_{j, n}(\sP) = X(\sG_{j.n}, w)$, where $\sG_{j,n}$ has  the same vertex set
 of $\sG$, plus one extra vertex $w$, which will be  the marked vertex. 
\end{thm} 

\begin{proof}
We associate  to the presentation $(\sG, v)$ a construction  $(\sG_n,v)$ called the {\em $n$-th higher power presentation},
in which $\sG_n$ has
the same vertex set as $\sG$ and the same initial vertex $v$, but its edges are labeled by the product alphabet $\sA^n$.
(This construction parallels that in  \cite{LM95}, Defn. 2.3.10.)
 In $\sG_m$ we draw a directed edge between vertices $v_1$ and $v_2$
with edge label $b_0b_1\cdots b_{n-1} \in \sA^n$ if there is a directed path of length $n$ in $\sG$ starting at $v_1$ and ending at $v_2$,
having successive edge labels $b_0, b_1 , \cdots, b_{n-1}$. 
 It is straightforward to see that 
$\sP= X(\sG_n, v)$, viewed in the enlarged alphabet $\sA^n$, generates the output infinite words in blocks of $n$ symbols.

We now obtain a presentation $(\sG_{j, n}, v)$  from $(\sG_{n}, v)$ by relabeling  edges, replacing each edge  symbol $b_0b_1\cdots b_{n-1} \in \sA^n$  by a single symbol $b_j\in \sA$, 
its  $j$-th symbol.  After this is done, there may exist  pairs of vertices $v_1$ and $v_2$ being connected by multiple edges labeled with
the same symbol $b_j$;  we  delete duplicate edges. In addition, the  resulting graph might be disconnected; we retain the induced subgraph having the set of 
 vertices reachable starting from $v$  using this set of edges.

We claim that   $\psi_{j,n}(\sP) = X(\sG_{j,n}, v)$. To prove the claim we show inclusions hold in both directions. 
 Suppose  $x=x_0x_1\ldots\in\psi_{(j,n)}(\sP)$. Then there is some word $y=y_0y_1\ldots\in\sP$ such that $x_i=y_{j+in}$ for all $i$. Since $y$ is presented by $(\mathcal{G},v)$, there is an infinite path in $\mathcal{G}$ starting at $v$, presenting $y$. Therefore, there is an edge in $\mathcal{G}_n$ from $v$ to a vertex $v'$ of $\sG$ labeled with the first $n$ letters of $y$, 
 another edge from $v'$ to another vertex $v''$ labeled with the next $n$ letters, and so on. Take a corresponding path in $\mathcal{G}_{j,n}$.
 The word presented will begin with the $j$th letter of the first block of $n$ letters from $y$, followed by the $j$th letter of the second block of $n$ letters, and so on. 
 Thus, the word presented will be $y_jy_{j+n}y_{j+2n}\ldots$. This word  is $x$ so  $x \in X(\sG_{j,n},v)$, whence 
$\psi_{j,n}(\sP) \subseteq X(\sG_{j,n}, v)$.
For the other inclusion,  suppose  $x\in X(\mathcal{G}_{j,n},v)$. Then there is a path on $\sG_{j,n}$ starting at $v$ presenting $x$. 
A corresponding path on $\sG_n$ will present a word whose $(j,n)$th decimation is $x$. Thus, $x\in\psi_{j,n}(\sP)$, and we have $\psi_{j,n}(\sP) \supseteq X(\sG_{j,n}, v)$.

We have completed the construction for principal decimations. For the remaining decimations $\psi_{j,n}(\sP)$ with $j \ge n$, 

we apply the  higher power construction  to the presentation obtained in  Theorem \ref{thm:decimation_alg}, in which 
$\psi_{j, n}(\sP)$ becomes the initial 
principal decimation $\psi_{0,n}(\rS^j \sP)$ of  $\rS^j \sP$.   
\end{proof}

%%%%%%%%%%%%%%%%%%%%%%%%%%%%%%%%%%%%%%%%%%%%%%%%%%%%%%%%%%%%%
%
% SUBSECT 3.3 Finiteness of decimations
%
%%%%%%%%%%%%%%%%%%%%%%%%%%%%%%%%%%%%%%%%%%%%%%%%%%%%%%%%%%%%%
\subsection{Finiteness of full decimation set  }\label{subsec:33N}

 Definition \ref{def:full-decimation-set} states that the {\em full decimation set} $\ffD(\sP)$ of a path set is defined by
$$
\ffD(\sP) := \{ \psi_{j,n}(\sP): \, \mbox{all} \quad n \ge 1 \quad \mbox{and all } \quad j \ge 0\}. 
$$
For the special case of path sets we show finiteness of the full decimation set.

%%%%%%%%%%%%%%%%%%%%%%%%%%%%%%%
% THM .3.7
%%%%%%%%%%%%%%%%%%%%%%%%%%%%%
\begin{thm}\label{thm:decimation_set_bound2}
{\rm (Full decimation set bound)} 
For each path set $\sP$ its full decimation set   $\ffD(\sP)$ is a finite set. 
If $\sP$ has a presentation having $m$ vertices, then $\ffD(\sP)$ has cardinality 
bounded by
$$
|\ffD(\sP) | \le 2^{ m^2|\sA|}.
$$
\end{thm}

%%%%%%%%%%%%%%%%%%%%%%%%%%%%%%%%%%%%
%
%   Proof of Theorem 1.6
%
%%%%%%%%%%%%%%%%%%%%%%%%%%%%%%%%%%% 

\begin{proof}
 Theorem \ref{thm:decimation_alg} implies 
a upper bound for $|\ffD(\sP)|$  given by the number of distinct labeled directed graphs on $m$ vertices having
the property that
there is at most one directed edge going from one given vertex $v_1$ to another given vertex $v_2$,
with a given symbol $a\in \sA$. The number of such directed vertex pairs is $m^2$ and the number of
possible directed edge patterns from  a fixed vertex $v_1$ to another fixed vertex $v_2$ is 
exactly $2^{|\sA|}$, so we obtain $|\ffD(\sP) | \le 2^{m^2|\sA|}.$
\end{proof} 

%%%%%%%%%%%%%%%%%%%%%%%%%%%%%%%
% Remark 3.8
%%%%%%%%%%%%%%%%%%%%%%%%%%%%%
\begin{rem}\label{rem:37}
In contrast to Theorem \ref{thm:decimation_set_bound}
 there exist closed $X \subset \sA^{\NN}$ for  which all members of the infinite collection  $\{\psi_{j, n}(X)\}$  for $j \ge 0$  are distinct,
 see \cite[Example 6.5]{ALS21}.
\end{rem}

%%%%%%%%%%%%%%%%%%%%%%%%%%%%%%%%%%%%%%%%%%%%%%%%%%%%%%%%%%%%%%%%%%%%%%%%%%%%%%%%
%
% SUBSECT 3.4 . (formerly 5.3) Decimation algorithm
%
%%%%%%%%%%%%%%%%%%%%%%%%%%%%%%%%%%%%%%%%%%%%%%%%%%%%%%%%%%%%%%%%%%%%%%%%%%%%%%%
\subsection{Right resolving  presentations for decimations of path sets}\label{subsec:34N}

%%%%%%%%%%%%%%%%%%%%%%%%%%%%%%%%%%%%%%%%%%%%%%%%%%%%%%%
% Remark  
%%%%%%%%%%%%%%%%%%%%%%%%%%%%%%%%%%%%%%%%%%%%%%%%%%%%%%%

 The output presentation  of $\psi_{j,n}(\sP)$ produced by Theorem \ref{thm:decimation_alg}
 need not be right-resolving, even if the 
 given input  presentation of $\sP$ were right-resolving. Using the subset construction for obtaining
 a right-resolving presentation from a general presentation, we obtain
 the following result.
 
 %%%%%%%%%%%%%%%%%%%%%%%%%%%%%%%%%
% Theorem 3.9
%%%%%%%%%%%%%%%%%%%%%%%%%%%%%%
\begin{thm}\label{thm:decimation_presentation}
{\rm (Right-resolving presentations of decimation sets of a path set)}
Given a path set $\sP$ on alphabet $\sA$ with at least two letters, having a 
(not necessarily right-resolving) presentation $\sP = X(\sG, v)$  with $m$ vertices.
Then for each $n \ge 1$ and each  $j \ge 0$ the  decimation set $\psi_{j,n} (\sP)$ has a right-resolving presentation having at most $2^{m+1}-1$ vertices.
\end{thm}

 This bound on the number of vertices implies a finiteness result for the 
 number of distinct decimation sets; see Theorem \ref{thm:decimation_set_bound2}.

%%%%%%%%%%%%%%%%%%%%%%%%%%%%
%
% SECT 4 Interleavings
%
%%%%%%%%%%%%%%%%%%%%%%%%%%%%%%%%%%%%%%%%%%%%%%%%%%%%%%%%%%%%%%%%%%%%%%%%%%%%%%%
\section{Interleaving  of  path sets} \label{sec:intalgsec}

Our object is to constructively prove the following result. 
%%%%%%%%%%%%%%%%%%%%%%%%%%%%%%%%%%%%%%%%%%%%%%%%%%%%%%%
% Theorem 4.1
%%%%%%%%%%%%%%%%%%%%%%%%%%%%%%%%%%%%%%%%%%%%%%%%%%%%%%%
\begin{thm}\label{thm:intalgthm2} 
  {\rm ($\sC(\sA)$ is closed under interleaving)}
If $\mathcal{P}_0, \ldots, \mathcal{P}_{n-1}$ are path sets on the alphabet $\mathcal{A}$, then their $n$-fold interleaving 
$$X := (\prn)_{i=0}^{n-1} \sP_i= \mathcal{P}_0 \pr \mathcal{P}_1 \pr \cdots \pr \mathcal{P}_{n-1}$$
 is a path set; i.e., $X \in \sC(\sA)$.  
\end{thm}

To do this, we give an effective procedure  for computing a presentation $(\sG, v)$ of the $n$-interleaving set
$X:= \mathcal{P}_0 \pr \mathcal{P}_1 \pr \cdots \pr \mathcal{P}_{n-1}$, given presentations of each input factor 
$\mathcal{P}_i= (\sG_i, v_i) $. This presentation certifies that $X$ is a path set.
We  give  examples. We also prove the converse result that every interleaving factor of a path set $\sP$ is a path set
given by some decimation of $\sP$.
 
%%%%%%%%%%%%%%%%%%%%%%%%%%%%%%%%%%%%%%%%%%%%%%%%%%%%%%%%%%%%%
%
% SUBSECT 4.1. The Algorithm
%
%%%%%%%%%%%%%%%%%%%%%%%%%%%%%%%%%%%%%%%%%%%%%%%%%%%%%%%%%%%%%
\subsection{ $n$-fold interleaving construction}\label{subsec:41}

%%%%%%%%%%%%%%%%%%%%%%%%%%%%%%%%%%%%%%%%%%%%%%%%%%%%%%%
% Theorem 4.2
%%%%%%%%%%%%%%%%%%%%%%%%%%%%%%%%%%%%%%%%%%%%%%%%%%%%%%%

\begin{thm}\label{thm:algthm2} 
{\rm(Interleaving pointed graph product construction)}
 Let $n \ge 2$  and suppose that $\mathcal{P}_0, \ldots, \mathcal{P}_{n-1}$ 
are path sets with given presentations $(\mathcal{G}_0,v_0), \ldots, (\mathcal{G}_{n-1},v_{n-1})$, respectively.
There exists a construction  taking as inputs these presentations and giving as output
a presentation $(\sH, \bv)$ of the $n$-fold interleaving
$X  := \mathcal{P}_0 \pr \mathcal{P}_1 \pr \cdots \pr \mathcal{P}_{n-1}.$ In particular $X= X(\sH, \bv)$
is a path set.  This construction has the following properties:
\begin{enumerate}
\item[(i)]
 If $\mathcal{G}_i$ has $k_i$ vertices for each $0 \leq i \leq n-1$, then $\mathcal{H}$ will have at most $n\prod_{i=0}^{n-1}k_i$ vertices. 
 \item [(ii)]
 If the pointed graphs $(\mathcal{G}_i, v_i)$ are right-resolving for all $0 \leq i \leq n-1$, then 
 the output pointed graph $\mathcal{H}$ will also be right-resolving.
 \item [(iii)]
 If the pointed graphs 
 $(\mathcal{G}_i, v_i)$ are pruned  for all $0 \leq i \leq n-1$, then 
 the output pointed graph $\mathcal{H}$ will also be pruned.
 \end{enumerate}
\end{thm}

%%%%%%%%%%%%%%%%%%%%%%%%%%%%%%%%%%%%%%%%%%%%%%%%%%%%%%%
% Proof of Interleaving Construction Theorem 1.4= Theorem 
%%%%%%%%%%%%%%%%%%%%%%%%%%%%%%%%%%%%%%%%%%%%%%%%%%%%%%%

\begin{proof}[Proof of Theorem ~\ref{thm:algthm}]
Suppose $\mathcal{P}_i$ is presented by the pointed  graph $(\mathcal{G}_i, v_i^0)$, which has vertex set $\mathcal{V}_i$ having  $k_i$ vertices, for all $0 \leq i \leq n-1$. 
We construct a new pointed  labeled graph $(\mathcal{H}, \bv)$, which we term the {\em  $n$-fold interleaving pointed graph product}  of  $\tilde{\sG}_i := (\mathcal{G}_i, v_i^0)$.

The underlying directed labeled  graph $\sH$ is the {\em  $n$-fold interleaving  graph product} of the labeled graphs $\sG_i$: 
$$
\sH := \pr_{i=0}^{n-1} \sG_i := \sG_0 \pr \sG_1 \pr \cdots \pr \sG_{n-1},
$$
using as  input  an  ordered set of $n$ directed labeled  graphs $\mathcal{G}_i$.

The vertices of $\sH$ consist of   a union of  products of the vertices of the $\sG_i$,  for $n$ cyclically
rotated copies of the $\sG_i$. 
To begin, choose (for convenience) a numbering  to the vertices $\sV_i$ from each $\mathcal{G}_i$, and let $v_i^j$ be the $j$th element of $\mathcal{V}_i$, with $v_i^0$ 
the marked vertex of $\mathcal{V}_i$. 
We define the $i$-th vertex set in $\sH$ to be 
$$
\sV^{i} := \sV_{i} \times \sV_{i+1} \times \cdots \times \sV_{n-1} \times \sV_0 \times \sV_1 \times \cdots \sV_{i-1}.
$$
and let
$\sV(\sH) = \bigcup_{i=0}^{n-1} \sV^{i}$ be the vertex set of $\sH$.
 Here a vertex in $\sV^{i}$ is a vector, 
$$
(v_i^{j_{i}}, v_{i+1}^{j_{i+1}} ,\ldots , v_{n-1}^{j_{n-1}}, v_0^{j_0},  v_1^{j_1}, \ldots v_{i-1}^{j_{i-1}})\,:  \, \, \mbox{where each  } \, 0 \le j_m \le k_m -1.
$$

 The labeled edges of $\mathcal{H}$    all connect vertices of $\sV^{i}$ to vertices of the next  set $\sV^{i+1}$, with 
 indices taken modulo $n$. 
Whenever there is an edge from $v_i^{\ell_1}$ to $v_i^{\ell_2}$ in $\mathcal{G}_i$ which has  label $a$,
draw edges in $\mathcal{H}$ labeled $a$ from each vertex   in $\sV^i$ that is of the form\footnote{There are $\frac{1}{k_i} \prod_{j=0}^{n-1} k_j$ such edges.} 
\begin{equation*}
(v_i^{\ell_1},  v_{i+1}^{j_{i+1}}, \ldots ,    v_{n-2}^{j_{n-2}}, v_{n-1}^{j_{n-1}},v_0^{j_0},  v_1^{j_1}, \ldots  ,   v_{i-1}^{j_{i-1}}). 
\end{equation*}
to that vertex in $\sV^{i+1}$ given by
\begin{equation*}
( v_{i+1}^{j_{i+1}}, \ldots ,    v_{n-2}^{j_{n-2}}, v_{n-1}^{j_{n-1}},v_0^{j_0},  v_1^{j_1}, \ldots ,    v_{i-1}^{j_{i-1}}, v_i^{\ell_2}). 
\end{equation*}
We use the cyclic ordering for superscripts $i ~(\bmod ~n)$ of $\sV^{i}$,  so that $\sV^{n} \equiv \sV^0$.

For  the pointed graph version of this construction, we add as  the pointed vertex  of $\sH$  vertex of $\sV^{0}$ given by $\bv^0= (v_0^0, v_{1}^0  \cdots  v_{n-2}^{0} , v_{n-1}^{0})$
determined by the pointed vertices $v_i^0$ of the individual $\sG_i$. 

Now define   $\widetilde{\mathcal{P}}$ to be the path set presented by
 $(\mathcal{H},\bv^0)$.\medskip
 
{\bf Claim.} {\em \,\,
$\widetilde{\mathcal{P}}=(\prn)_{i=0}^{n-1}\mathcal{P}_i :=\mathcal{P}_0 \pr \mathcal{P}_1 \pr \cdots \pr \mathcal{P}_{n-1}.$}\\

To prove the claim, 
we first  show  the inclusion  \, $(\prn)_{i=0}^{n-1}\mathcal{P}_i \subseteq \widetilde{\mathcal{P}}$.
Let  $(x_t)_{t=0}^{\infty}\in (\prn)_{i=0}^{n-1}\mathcal{P}_i$. By definition  there  exist  edge label elements separately in each factor $\mathcal{P}_i$
that traverse an infinite edge path, 
$(e_{i,t})_{t=0}^{\infty} $
in $\sG_i$ which  has  associated symbol sequence 
\begin{equation*}
(x_{i,t})_{t=0}^{\infty}\in \mathcal{P}_i\text{ for all }0\leq i\leq n-1,
\end{equation*}
which visits an  associated sequence of vertices 
$$
(v_{i, t})_{t=0}^{\infty}
$$
in the graph $\sG_i$. We call $(v_{i,t})_{i=0}^{\infty}$ a {\em vertex path} associated to the edge path $(x_{i, t})_{i=0}^{\infty}$.
(A vertex path is uniquely determined by the edge path, requiring that the initial vertex be the marked vertex. 
There could be several edge paths giving the same marked vertex, if there are multiple edges.)

 The edge symbol sequences $(x_{i, t})_{i=0}^{\infty} $  interleave  to reconstruct
the sequence $(x_t)_{t=0}^{\infty}$ via
$$
x_{i,t}=x_{i+ nt} \quad \mbox{ for all} \,\, 0\leq t<\infty.
$$
We show these elements give a sequence of update edges for an edge path in $\sH$ realizing this symbol sequence $x_t$. We start at 
$t=0$ at initial vertex $\bv^0 =(v_{0}^0,v_{1}^0, \ldots ,v_{n-1}^0) \in \sV^0$.
We proceed in ``rounds" of $n$ steps. At the beginning of ``round'' $k$ at $t=kn$ we will be at a vertex
$\bv^{k} =(v_0^{j_{0,k}}, v_1^{j_{1,k}} , \ldots , v_{n-1}^{j_{n-1, k}}) \in \sV^{0})$. During the round, with steps numbered $0 \to n-1$, the vertices cyclically rotate (to the left)
and at the start of $i$-th step the vector is initially in  $\sV^{i}$, the leftmost vertex $v_i^{j_i, k}$ is updated to $v_i^{j_{i+1}, k+1}$, by moving on an edge on $\sG_i$ between these two vertices, which has   edge
label $x_{nk+i}$, and the new vertex in $\sG_i$ is moved  all the way to the right, to get a vector in $\sV^{i+1}$. 
 At the end of the round we back in $\sV^0$ at a new vertex vector $\bv^{k+1}.$
We proceed by induction on the number of rounds $k$. The base case $k=0$ starts with  all $j_{i, 0} =0$. 
The induction hypothesis is that at the end of the $r$-th round  we have followed a path on $\sH$ that  incremented motion on each of the $\sG_i$ by one symbol of $x_{i, r}$
and has  moved in the $i$-th vector coordinate 
from vertex $v_{i,r}$ to the  vertex $v_{i, r+1}$,  corresponding to $\sG_i$, That is, at step $nr+i$ we 
produced symbol $x_{i, r+1}$ and at time $t= n(r+1)$ the $i$-th vector component of 
 $\bv^{r+1}= (v_0^{j_{0,r+1}}, \cdots , v_{n-1}^{j_{n-1, r-1}})$, has entry  $v_{i, r+1} \in \sG_i$, for $0 \le i \le n-1$. 
The induction step is completed using 
the fact that each $(e_{i, t} : t \ge 0 )   \in \sP_i$ is a legal edge path in $\sG_i$ that permits taking the next step at $\sG_i$  in the next round. 
We conclude that  there is an infinite path in $\mathcal{H}$ originating at $\bv^{0}$ 
with edge labels $(x_0,x_1,x_2,\ldots)$, producing $(x_t)_{t=0}^{\infty}$ is an 
element of $\widetilde{\mathcal{P}}$. Thus 
$(\prn)_{i=0}^{n-1}\mathcal{P}_i \subseteq \widetilde{\mathcal{P}}.$

It remains to show  the reverse inclusion $\widetilde{\mathcal{P}} \subseteq (\prn)_{i=0}^{n-1}\mathcal{P}_i $. 
Suppose given $(x_k)_{k=0}^{\infty}\in \widetilde{\mathcal{P}}$. Then in the first $n$ steps there is a vertex path
\begin{equation*}
\begin{array}{c}
( v_0^0, v_1^0, \ldots,  v_{n-2}^0, v_{n-1}^0) ; 
\quad ( v_1^{0}, v_{2}^0, \ldots,  v_{n-1}^0, v_0^{j_{0,1}});
\quad (v_{2}^0, v_{3}^0, \ldots ,v_{n-1}^0 ,  v_0^{j_{0,1}}, v_1^{j_{1,1}});
\\
\ldots ; 
\quad (v_{n-1}^0, v_{0}^{j_{0,1}}, v_1^{j_{1,1}}, \ldots, v_{n-2}^{j_{n-2,1}}); 
\quad (v_0^{j_{0,1}}, v_1^{j_{1,1}}, \ldots, v_{n-2}^{j_{n-2,1}}, v_{n-1}^{j_{n-1,1}}); 
% \quad 
% (v_{1}^{j_{1,1}}v_{2}^{j_{2,1}}\ldots v_{n-1}^{j_{1,1}}v_0^{j_{0,2}};\quad\ldots
\end{array}
\end{equation*}
in $\mathcal{H}$ which can be traversed by edges labeled $x_0,x_1,\ldots, x_{n-1}$. Notice that the first coordinate of a vertex will be the 
last coordinate of the vertex that follows after $n-1$ steps. Since the initial vertex is $(v_{1}^0, v_{2}^0\ldots v_{n-1}^0,v_n^0)$, we know 
that for each $0\leq i\leq n-1$, there is a matching edge in $\mathcal{G}_i$ from $v_i^0$ to $v_i^{j_{i,1}}$. 

For any $k<\infty$, an edge in $\mathcal{H}$ with edge label $x_k$ from vertex
\begin{equation*}
(v_i^{j_{i,r}}, \ldots, v_{n-1}^{j_{n-1,r}}, v_{0}^{j_{0,r+1}}, v_{1}^{j_{1,r+1}}, \ldots . v_{i-1}^{j_{i-1,r+1}} ) \in \sV^{i} , 
\end{equation*}
to vertex
\begin{equation*}
(  v_{i+1}^{j_{i+1,r}}, \ldots v_{n-1}^{j_{n-1,r}}, v_{0}^{j_{0,r+1}}, v_{1}^{j_{1,r+1}}, \ldots , v_{i-1}^{j_{i-1,r+1}}, v_i^{j_{i,r+1}}) \in \sV^{i+1} 
 \end{equation*}
corresponds to a directed  edge in $\mathcal{G}_i$   from $v_i^{j_{i,r}}$ to $v_i^{j_{i,r+1}}$ that has label $x_k$. 
Following our given vertex path in $\mathcal{H}$ for $n-1$ more steps gets us to a vertex in $\sH$ 
whose last coordinate is $v_i^{j_{i,r+1}}$. There is an  edge in $\mathcal{H}$ labeled $x_{k+n}$ emanating from this vertex 
which corresponds to an edge in $\mathcal{G}_i$ labeled $x_{k+n}$ emanating from $v_i^{j_{i,r+1}}$ and going 
to $v_i^{j_{i,r+2}}$. Thus, for each $0\leq i\leq n-1$, the labels $x_i,x_{i+n},x_{i+2n}\ldots$ are the labels of an infinite path in
 $\mathcal{G}$ originating at $v_i^0$, so $(x_k)_{k=0}^{\infty}\in (\prn)_{i=0}^{n-1}\mathcal{P}_i$. We conclude  that 
$\widetilde{\mathcal{P}} \subseteq (\prn)_{i=0}^{n-1}\mathcal{P}_i.$

The claim follows. 

The claim shows  that the interleaving $(\prn)_{i=0}^{n-1}\sP_i$ is the path set $\widetilde{\sP}$, having a presentation $(\sH, \bv^0)$.
Since $\sH$ has $n \prod_{i=0}^{n-1} k_i$ vertices, this  proves (i) .
For (ii), if each of the $\sG_i$  is right-resolving, then it is evident from the interleaving graph product construction that $\sH$ is right-resolving.
Each vertex $\bv$ of $\sV^{i}$  has at most one exit edge having  a given label $a$, inherited from $\sG_i$.
For (iii), if the graph $\sG_i$ is pruned, then each vertex has at least one exit edge. The construction  of edges for $\sH$ then shows that each vertex
in $\sV^{i}$ has an exit edge if and only if each vertex of $\sG_i$ has an exit edge. 
\end{proof} 

%%%%%%%%%%%%%
% Proof of Theorem 4.1
%%%%%%%%%%%%%

\begin{proof}[Proof of Theorem ~\ref{thm:intalgthm}]
Theorem \ref{thm:intalgthm}  is  is an immediate consequence of Theorem ~\ref{thm:algthm}. 
\end{proof}

%%%%%%%%%%%%%%%%%%%%%%%
%
%Remark 4.3
%
%%%%%%%%%%%%%%%%%%%%%%%
\begin{rem}\label{rmk:graph-product-cmt}
(1) The $n$-fold interleaving graph product operation does not always produce minimal right-resolving presentations, even when all the input presentations are minimal right-resolving, see Example \ref{exmp1}.  
 
 (2) The $n$-fold interleaving graph product operation does not always produce reachable presentations when
 all the input presentations are reachable. 
 \end{rem}

%%%%%%%%%%%%%%%%%%%%%%%%%%%%%%%%%%%%%%%%%%%%%%%%%%%%%%%%%%%%%%%%%%%%%%%%%%%%%%%%
%
% SUBSECT 4.2. Examples 
%
%%%%%%%%%%%%%%%%%%%%%%%%%%%%%%%%%%%%%%%%%%%%%%%%%%%%%%%%%%%%%%%%%%%%%%%%%%%%%%%
\subsection{Examples}

We present  examples showing that the $n$-fold interleaving pointed graph product, given minimal right-resolving presentations $(\sG_i, v_i)$ as input, 
may not  produce a right-resolving presentation of their interleaving as output.

%%%%%%%%%%%%%%%%%%%%%%%%%%%%%%%%%%%%%%%%%%%%%%%%%%%%%%%
% Example 4.1
%%%%%%%%%%%%%%%%%%%%%%%%%%%%%%%%%%%%%%%%%%%%%%%%%%%%%%%
\begin{exmp}\label{exmp1}
(Non-preservation of minimal right-resolving property: extra automorphisms)
Let $\mathcal{P}_0 = X(\mathcal{G}_0, v_0)$ and $\mathcal{P}_1 = X(\mathcal{G}_1,v_1)$, where $\mathcal{G}_0$ and $\mathcal{G}_1$ are the graphs given in Figure 4.1. 
Evidently $\mathcal{P}_0 = \mathcal{P}_1 = \{0,1\}^\mathbb{N}$, the full shift on two letters, 
and $(\mathcal{G}_0,v_0)$, $(\mathcal{G}_1,v_1)$ are (isomorphic) minimal right-resolving presentations. It is easy to see that $\mathcal{P}_0 \pr \mathcal{P}_1 = \{0,1\}^\mathbb{N}$ 
as well. Figure 4.2 shows the presentation of $\mathcal{P}_0 \pr \mathcal{P}_1$ given by our algorithm. This presentation is right-resolving but  non-minimal; it is a double-covering
of the minimal right-resolving representation.
\end{exmp}

\begin{figure}[ht]\label{fig1}
	\centering
	\psset{unit=1pt}
	\begin{pspicture}(-80,25)(80,70)
		\newcommand{\noden}[2]{\node{#1}{#2}{n}}
		\noden{$v_0$}{-60,50}
		\noden{$v_1$}{60,50}
		\bcircle{n$v_1$}{90}{0}
		\bcircle{n$v_0$}{90}{0}
		\bcircle{n$v_1$}{270}{1}
		\bcircle{n$v_0$}{270}{1}
	\end{pspicture}
	\newline
\hskip 0.5in {\rm FIGURE 4.1.} Presentations of $\mathcal{P}_0$ and $\mathcal{P}_1$ of Example ~\ref{exmp1}.
\newline
\end{figure}

\begin{figure}[ht]\label{fig2}
	\centering
	\psset{unit=1pt}
	\begin{pspicture}(-50,-25)(50,60)
		\newcommand{\noden}[2]{\node{#1}{#2}{n}}
		\noden{$v_1v_0$}{0,0}
		\noden{$v_0v_1$}{0,60}
		\dline{n$v_1v_0$}{n$v_0v_1$}{0,1}{0,1}
	\end{pspicture}
	\newline
\hskip 0.5in {\rm FIGURE 4.2.} Presentation of $\mathcal{P}_0 \pr \mathcal{P}_1$ of Example ~\ref{exmp1}. The  initial vertex is $v_0v_1$.
\newline
\end{figure}

The non-minimality of the presentation constructed in Example ~\ref{exmp1} is a result of the fact that 
the $n$-fold interleaving graph product
construction  keeps track of which input path set each digit comes from. 
If all the input path sets have the same presentation, then the graph product  has a cyclic automorphism of
order $n$.
For  any path set $\mathcal{P}$, the  presentation of the $n$-fold self-interleaving $\mathcal{P} \pr \mathcal{P} \pr \cdots \pr \mathcal{P}$ 
given by the construction  of Theorem ~\ref{thm:algthm} is an $n$-fold covering of another presentation, the one constructed in  \cite[Proposition 3.4]{ABL17}.

%%%%%%%%%%%%%%%%%%%%%%%%%%%%%%%%%%%%%%%%%%%%%%%%%%%%%%%
% Example 4.2
%%%%%%%%%%%%%%%%%%%%%%%%%%%%%%%%%%%%%%%%%%%%%%%%%%%%%%%
\begin{exmp}\label{exmp2} 
(Non-preservation of minimal right-resolving property: failure of follower-separation)
Consider the path sets $\mathcal{Q}_0=\{(0^{\infty})\}\cup\{(0^n12^\infty)|n\in\mathbb{N}\}$ and $\mathcal{Q}_1=\{(32^{\infty})\}$. Figure 4.3 gives minimal right-resolving presentations $(\mathcal{H}_0,v_0)$ of $\mathcal{Q}_0$ and $(\mathcal{H}_1, v_2)$ of $\mathcal{Q}_1$. The presentation of $\mathcal{Q}_0 \pr \mathcal{Q}_1$ given by Theorem \ref{thm:algthm} is shown in Figure 4.4. This presentation is not minimal, since $v_1v_3$ and $v_3v_1$ have the same follower sets. 
However, identifying the vertices $v_1v_3$ and $v_3v_1$ and replacing the edges between them with a single self-loop labeled $2$ will give a minimal right-resolving presentation. This
presentation is shown in Figure 4.5.
\end{exmp}

% Figures 4.3 through 4.5
\begin{figure}[ht]\label{fig3}
	\centering
	\psset{unit=1pt}
	\begin{pspicture}(-130,-30)(130,20)
		\newcommand{\noden}[2]{\node{#1}{#2}{n}}
		\noden{$v_0$}{-100,0}
		\noden{$v_1$}{-50,0}
		\noden{$v_2$}{50,}
		\noden{$v_3$}{100,0}
		\bcircle{n$v_0$}{90}{0}
		\bcircle{n$v_1$}{270}{2}
		\bcircle{n$v_3$}{270}{2}
		\bline{n$v_0$}{n$v_1$}{1}
		\bline{n$v_2$}{n$v_3$}{3}
	\end{pspicture}
\newline
\hskip 0.5in {\rm FIGURE 4.3.} Presentations of $\mathcal{Q}_0$ and $\mathcal{Q}_1$ of Example ~\ref{exmp2}. The initial vertices are
$v_0$ and $v_2$, respectively.
\newline
\end{figure}

\begin{figure}[ht]\label{fig4}
	\centering
	\psset{unit=1pt}
	\begin{pspicture}(-150,30)(100,135)
		\newcommand{\noden}[2]{\node{#1}{#2}{n}}
		\noden{$v_0v_2$}{0,125}
		\noden{$v_2v_0$}{-60,105}
		\noden{$v_0v_3$}{-80,45}
		\noden{$v_3v_0$}{-140,45}
		\noden{$v_2v_1$}{60,105}
		\noden{$v_1v_3$}{80,45}
		\noden{$v_3v_1$}{0,45}
		\bline{n$v_0v_2$}{n$v_2v_1$}{1}
		\bline{n$v_0v_2$}{n$v_2v_0$}{0}
		\bline{n$v_2v_0$}{n$v_0v_3$}{3}
		\bline{n$v_0v_3$}{n$v_3v_1$}{1}
		\bline{n$v_2v_1$}{n$v_1v_3$}{3}
		\dline{n$v_0v_3$}{n$v_3v_0$}{0}{2}
		\dline{n$v_1v_3$}{n$v_3v_1$}{2}{2}
	\end{pspicture}
\newline
\newline
\hskip 0.5in {\rm FIGURE 4.4.} Presentation of $\mathcal{Q}_0 \pr \mathcal{Q}_1$ of Example ~\ref{exmp2}. The  initial vertex is $v_0v_2$.
\newline
\end{figure}

\begin{figure}[ht]\label{fig5}
	\centering
	\psset{unit=1pt}
	\begin{pspicture}(-150,20)(100,135)
		\newcommand{\noden}[2]{\node{#1}{#2}{n}}
		\noden{$v_0v_2$}{0,125}
		\noden{$v_2v_0$}{-60,105}
		\noden{$v_0v_3$}{-60,45}
		\noden{$v_3v_0$}{-140,45}
		\noden{$v_2v_1$}{60,105}
		\noden{$v_3v_1$}{60,45}
		\bline{n$v_0v_2$}{n$v_2v_1$}{1}
		\bline{n$v_0v_2$}{n$v_2v_0$}{0}
		\bline{n$v_2v_0$}{n$v_0v_3$}{3}
		\bline{n$v_0v_3$}{n$v_3v_1$}{1}
		\bline{n$v_2v_1$}{n$v_3v_1$}{3}
		\dline{n$v_0v_3$}{n$v_3v_0$}{0}{2}
		\bcircle{n$v_3v_1$}{270}{2}
	\end{pspicture}
\newline
\hskip 0.2in {\rm FIGURE 4.5.} Minimal right-resolving presentation of $\mathcal{Q}_0 \pr \mathcal{Q}_1$ of Example ~\ref{exmp2}. The  initial vertex is $v_0v_2$.
\newline
\end{figure}
% End Figures 4.3 to 4.5

%%%%%%%%%%%%%%%%%%%%%%%%%%%%%%%%%%%%%%%%%%%%%%%%%%%%%%%%%%%%%%
% SECT 5. Interleaving closures  of Path Sets
%
%%%%%%%%%%%%%%%%%%%%%%%%%%%%%%%%%%%%%%%%%%%%%%%%%%%%%%%%%%%%%
\section{Interleaving  closure operations } \label{sec:5}

The paper  \cite{ALS21} shows that interleavings of principal $n$-decimations 
define a series of closure operations $X \mapsto X^{[n]}$ for  arbitrary subsets $X \subset \sA^{\NN}$.
We recall properties of these operations established in \cite{ALS21} which relate them to
$n$-fold decimations. 
Then we  show that the class $\sC(\sA)$ of path sets is stable  under these closure operations.

%%%%%%%%%%%%%%%%%%%%%%%%%%%%%%%%%%%%%%%%%%%%%%%%%%%%%%%%%%%%%
%
% SUBSECT 5.1)  Closure under factorization 
%
%%%%%%%%%%%%%%%%%%%%%%%%%%%%%%%%%%%%%%%%%%%%%%%%%%%%%%%%%%%%%
\subsection{Interleaving  closure operations}

Decimations combined with $n$-interleavings define a series of closure operations on path sets.
The closure operations are defined for arbitrary subsets $X \subset \sA^{\NN}$, as described in \cite{ALS21}.

%%%%%%%%%%%%%%%%%%%%%%%%%%%%%%%%%%%%%%%%%%%%%%%%%%%%%%%
% Definition  5.1//
%%%%%%%%%%%%%%%%%%%%%%%%%%%%%%%%%%%%%%%%%%%%%%%%%%%%%%%
\begin{defn} \label{def:dec-closed} 
(Interleaving closure operations) Given a subset $X$ of $\msrA^{\NN}$ the 
  {\em  $n$-fold interleaving closure $X^{[n]}$ of $X$}  is given by the $n$-fold interleaving 
\begin{equation*}
X^{[n]} := \psi_{0, n}(X) \pr\psi_{1,n}(X) \ast \cdots \pr \psi_{n-1, n}(X).
%(a_ia_{i+m}a_{i+2m}\cdots) 
\end{equation*}
\end{defn}

We recall some results  from \cite{ALS21}. The following result parts  (1) and (2) are consequences of
Theorem 4.2 of \cite{ALS21}), and parts (3) and (4) are consequences of Theorem 4.12 of \cite{ALS21}.

%%%%%%%%%%%%%%%%%%%%%%%%%%%%%%%%%%%%%%%%%%%%%%%%%%%%%%%
% Theorem 5.2
%%%%%%%%%%%%%%%%%%%%%%%%%%%%%%%%%%%%%%%%%%%%%%%%%%%%%%%
\begin{thm}\label{thm:inclusion}
{\rm ($n$-fold interleaving closure)} \,
Given a subset $X$ of $\msrA^{\NN}$ one has the set inclusion 
\begin{equation}\label{eqn:interleave-inclusion}
X \subseteq X^{[n]},
\end{equation}
where $X^{[n]}$ is the $n$-fold interleaving closure of $X$.
If $X \subseteq Y $ then $X^{[n]} \subseteq Y^{[n]}$. 
In addition:

(1) The operation $X \mapsto X^{[n]}$ is idempotent; i.e., 
$(X^{[n]})^{[n]} = X^{[n]}$ for all $X \subset \sA^{\NN}$. 

(2) The  $n$-fold interleaving closure $X^{[n]}$ has the property that  it is the maximal set $Y$ such
that $X \subseteq Y$ and 
\begin{equation*}
\psi_{j,n}(Y) = \psi_{j, n}(X) \quad \mbox{for} \quad 0\le j \le n-1.
\end{equation*}

(3) If $X$ is a closed set in $\msrA^{\NN}$
then each decimation set $X_{j,n}= \psi_{j,n}(X)$ is a closed set.
The n-th interleaving closure $X^{[n]}$ is  a closed set in $\sA^{\NN}$.

(4) The $n$-fold interleaving closure operation  commutes with the closure operation on 
the product topology on $\sA^{\NN}$,  in the sense that
$$
(\overline{X})^{[n]}  = \overline{X^{[n]}}. 
$$
\end{thm}

The following result is  is a consequence of Theorem 2.8 of \cite{ALS21}.

%%%%%%%%%%%%%%%%%%%%%%%%%%%%%%%%%%%%%%%%%%%%%%%%%%%%%%%
% Theorem  5.3// 
%%%%%%%%%%%%%%%%%%%%%%%%%%%%%%%%%%%%%%%%%%%%%%%%%%%%%%%
\begin{thm}\label{thm:DIF2} 
%(\cite[Theorem 2.5]{ALS21}) 
{\rm (Decimations and interleaving factorizations)}
A  general subset $X$ of $\msrA^{\NN}$  
has an {$n$-fold interleaving factorization} 
$X = X_{0} \pr X_{1} \pr \cdots \pr X_{n-1}$
if and only if $X= X^{[n]}$.
In this case,   each 
$$X_{i} = \psi_{j,n}(X) \quad \mbox{for} \quad 0 \le j \le n-1,$$
so when  they exist,  $n$-fold interleaving factorizations are unique.
\end{thm} 

%%%%%%%%%%%%%%%%%%%%%%%%%%%%%%%%%%%%%%%%%%%%%%%%%%%%%%%%%%%%%%
% SUBSECT  5.2 (formerly 5.5)  Closure under factorization 
%
%%%%%%%%%%%%%%%%%%%%%%%%%%%%%%%%%%%%%%%%%%%%%%%%%%%%%%%%%%%%%
\subsection{Interleaving closures of path sets}
%Proof of Theorem \ref{factorthm}.}
\label{subsec:52}

We specialize to path sets, and show all the 
 $n$-fold interleaving closures $\sP^{[n]}$, $(n \ge 1)$, of a path set $\sP$ are path sets.
(Theorem \ref{thm:factorthm}). 
 
 %%%%%%%%%%%%%%%%%%%%%%%%%%%%%%%%%%%%%%%%%%%%%%%%%%%%%%
% Theorem 5.4
%%%%%%%%%%%%%%%%%%%%%%%%%%%%%%%%%%%%%%%%%%%%%%%%%%%%%%%
\begin{thm}\label{thm:factorthm2}
{\rm ($\sC(\sA)$  is stable under $n$-fold  interleaving closure operations)} 
If $\sP$ is a path set, then for each $n \ge 1$ the $n$-fold interleaving closure $\sP^{[n]}$ is a path set.
In addition, if $\sP$ is $n$-factorizable then each of its $n$-fold intereaving factors $\sP_j= \psi_{j,n}(\sP)$  for $0 \le j\le n-1$
are path sets.
\end{thm}.

%%%%%%%%%%%%%%%%%%%%%%%%%%%%%%%%%%%%%%%%%%%%%%%%%%%%%%%
% Proof of Theorem 1.8(factor theorem) 
%%%%%%%%%%%%%%%%%%%%%%%%%%%%%%%%%%%%%%%%%%%%%%%%%%%%%%%

\begin{proof}
If $\sP$ is a path set then by  Theorem \ref{thm:decimation} 
each $\psi_{j,n}(\sP)$ is a path set.
Thus
$$
\sP^{[n]} := \psi_{0, n}(\sP) \pr \psi_{1, n}(\sP) \pr \cdots \pr \psi_{n-1,n}(\sP), 
$$
is a path set by Theorem \ref{thm:intalgthm2}.

Now suppose $\sP$ has an $n$-fold interleaving factorization 
$$
\sP = X_{0,n} \pr X_{1, n} \pr \cdots \pr X_{n-1, n}.
$$
Then  by Theorem \ref{thm:DIF2} $X_{j,n}= \psi_{j,n}(\sP)$
for $0 \le j\le n-1$. But $\psi_{j,n}(\sP)$ is a path set by Theorem \ref{thm:decimation}.
\end{proof}

%%%%%%%%%%%%%%%%%%%%%%%%%%%%%%%%%%%%%%%%%%%%%%%%%%%%%%%%%%%%%
%
% SECTION 6: INTERLEAVING FACTORIZATION
%
%%%%%%%%%%%%%%%%%%%%%%%%%%%%%%%%%%%%%%%%%%%%%%%%%%%%%%%%%%%%%
\section{Interleaving factorizations } \label{sec:5B}

We recall results on interleaving factorizations of general sets $X \subset \sA^{\NN}$ from \cite{ALS21},
relating to closure operations. We deduce that the set $\sC^{\infty}(X)$ of infinitely factorizable path sets
is stable under $n$-fold interleavings of its members, for all $n \ge 1$. Finally we obtain a bound on the
size of minimal right-resolving presentations of interleaving factors of path sets, which are ncessarily
principal decimations, which improves on
the bound of Theorem \ref{thm:decimation_alg} for general decimations of path sets.

%%%%%%%%%%%%%%%%%%%%%%%%%%%%%%%%%%%%%%%%%%%%%%%%%%%%%%%%%%%%%
%
% SUBSECT 6.1  interleaving factors 
%
%%%%%%%%%%%%%%%%%%%%%%%%%%%%%%%%%%%%%%%%%%%%%%%%%%%%%%%%%%%%%
\subsection{Structure of interleaving factors:arbitrary sets $X$}\label{subsec:55}

We recall from \cite{ALS21} results on the structure of possible interleaving factorizations for a general set $X \subset \sA^{\NN}$,
and derive corollaries of one of them.  The following result is a consequence of Theorem 2.12 of \cite{ALS21}.

%%%%%%%%%%%%%%%%%%%%%%%%%%%%%%%%%%%%%%%%%%%%%%%%%%%%%%%
% Theorem  6.1/./theorem 1.6 [AL20a]
%%%%%%%%%%%%%%%%%%%%%%%%%%%%%%%%%%%%%%%%%%%%%%%%%%%%%%%
\begin{thm}   \label{thm:lcm-factorization} 
{\rm (Divisibility for interleaving factorizations) }

(1) Let   $X \subseteq \sA^{\NN}$ have an $n$-fold interleaving factorization. If $d$ divides $n$, then $X$ also has an 
$d$-fold interleaving factorization.

(2) Let $X$ have $m$-fold and $n$-fold interleaving factorizations. Then $X$ has an
$\ell$-fold interleaving factorization, where $\ell= \lcm(m,n)$ is the least common multiple of $m$ and $n$.
\end{thm}

An immediate consequence of this result is a dichotomy. 
For a general set $X \subset \sA^{\NN}$, exactly one of the following hold.
\begin{enumerate}
\item[(1)]
$X$ is factorizable for infinitely many $n$, 
\item[(2)]
 $X$ is $n$-factorizable for
a finite set of $n$, which are exactly the divisors of a fixed integer $f= f(X)$. 
\end{enumerate}.

The following result shows that if $X$ is a closed set, then 
infinite factorizability  implies $n$-factorizability for all $n \ge 1$. 
It is a consequence of Theorem 2.13 of \cite{ALS21}.
%%%%%%%%%%%%%%%%%%%%%%%%%%%%%%%%%%%%%%%%%%%%%%%%%%%%%%%
% Theorem  6.2/ Theorem 1.7 [AL20a}
%%%%%%%%%%%%%%%%%%%%%%%%%%%%%%%%%%%%%%%%%%%%%%%%%%%%%%%

\begin{thm} \label{thm:61} 
{\rm ( Classification of infinitely factorizable closed $X$)} 
 For a closed set $X \subseteq \msrA^{\NN}$ where $\msrA$ is a finite alphabet,
the following properties are equivalent.
\begin{enumerate}
\item[(i)]
$X$ is infinitely factorizable; i.e., $X$ has an $n$-interleaving factorization for infinitely many $n \ge 1$.
\item[(ii)]
$X$ has an $n$-interleaving factorization for all $n \ge 1$. 
\item[(iii)]
For each $k \ge 0$  there are nonempty subsets $\msrA_k \subset \msrA$ such that $X = \prod_{k=0}^{\infty} \msrA_k$
is a countable  product of finite sets  with the product topology.
\end{enumerate}
\end{thm}

%%%%%%%%%%%%
% Remark 6.3  [Algorithm]
%%%%%%%%%%%%
\begin{rem}\label{rem:64a} 
(1) If $|\sA| \ge 2$, then there are uncountably many infinitely factorizable closed sets $X \subset \sA^{\NN}$,
while there are only countably many path sets.

(2) For   $k, \ell \ge 1$ and a
finite set of consecutive $\sA_k, \sA_{k+1}, \sA_{k+2}, ... \sA_{k+\ell}$  such that there is a block
$a_k a_{k+1} \cdots a_{k+\ell}$ with each $a_{k+i} \in \msrA_{k+i}$ for $0\le i \le \ell$  that does not occur in any element of $X$,
we say that $X$ has   a  \emph{ $(k , \ell)$-missing-configuration}. 
The proof  shows that   existence of a $(k, \ell)$-missing-configuration certifies that $X$ has no 
$n$-fold interleaving factorization with $n \ge k+ \ell +1$. The proof of Theorem \ref{thm:61} given in \cite{ALS21}
shows that if $X$ is not infinitely factorizable then it has a $(k, \ell)$-missing configuration for some finite $k, \ell \ge 1$. 
\end{rem}

%%%%%%%%%%%%%%%%%%%%%%%%%%%%%%%%%%%%%%%%%%%%%%%%%%%%%%%
% Corollary 6.4 (5.9 /6.3)
%%%%%%%%%%%%%%%%%%%%%%%%%%%%%%%%%%%%%%%%%%%%%%%%%%%%%%%

\begin{cor}\label{cor:62}
Let $X$ be an infinitely factorizable closed subset of $\mathcal{A}^{\NN}$. Then it  is factorizable for all $n \ge 1$, so  its
factor set $\fF(X)$ contains all decimations
$\psi_{j, n}(X)$ for  $n \ge 1$ and $0 \le j \le n-1$. 
Each decimated set $\psi_{j, n}(X)$ is also infinitely factorizable. 
\end{cor}

\begin{proof} 
By property (ii) of Theorem \ref{thm:61},  $X$ is factorizable for each $n \ge 1$,
and its  $n$-fold factors are $\psi_{j, n}(X)$ for $0 \le j \le n-1$.
Now the property (iii) is preserved under decimations of all orders, hence all $\psi_{j, n}(X)$
must be infinitely factorizable.
\end{proof}

For an  infinitely factorizable $X$ it is possible that  all decimations $\psi_{j,n}(X)$  are pairwise distinct.
In  such cases  the factor set $\fF(X)$ would be   infinite.

%%%%%%%%%%%%%%%%%%%%%%%%%%%%%%%%%%%%%%%%%%%%%%%%%%%%%%%
% Corollary 6.5 (5.10)
%%%%%%%%%%%%%%%%%%%%%%%%%%%%%%%%%%%%%%%%%%%%%%%%%%%%%%%

\begin{cor}\label{cor:63}
The set  $\sY( \mathcal{A})$
of all infinitely factorizable closed subsets   $X \subseteq \mathcal{A}^{\NN}$ is closed under the operation of 
$n$-fold interleaving for all $n \ge 1$. That is, if $X_0, X_1, \cdots , X_{n-1} \in \sY(\mathcal{A})$, then
$$
(\pr)_{i=0}^{n-1} X_i = X_0 \pr X_1 \pr \cdots \pr X_{n-1} \in \sY(\mathcal{A}).
$$
\end{cor}

\begin{proof} 
This fact follows using the characterization of infinitely factorizable by property (iii) of  Theorem \ref{thm:61}.
This property is inherited under $n$-fold interleaving of sets $X_i$ that have it.
 \end{proof}
Combining Theorem \ref{thm:lcm-factorization} and Theorem \ref{thm:61}   yields the following result.

%%%%%%%%%%%%%%%%%%%%%%%%%%%%%%%%%%%%%%%%%%%%%%%%%%%%%%%
% Theorem 6.6  ( 5.11 /Theorem 1.8 [AL20a]
%%%%%%%%%%%%%%%%%%%%%%%%%%%%%%%%%%%%%%%%%%%%%%%%%%%%%%%
\begin{thm} \label{thm:classification} (\cite[Theorem 2.10]{ALS21}) 
{\rm (Dichotomy theorem)}
Let $X$ be a closed subset of $\sA^{\NN}$. Then exactly one of the following holds for $X$.

(1) (Infinitely factorizable) The set of $n$ where $X$ has an $n$-interleaving factorization is the set of all positive integers $\NN^{+}$.

(2) (Finitely factorizable) The set of $n$ where $X$ has an $n$-interleaving factorization is the set of all divisors of
some  integer $f = f(X)$.
%depending on $X$.
\end{thm}
%%%%%%%%%%%%%%%%%%%%%%%%%%%%%%%%%%%%%%%%%%%%%%%%%%%%%%%%%%%%%
%
% SUBSECT 5.5. interleaving factors 
%
%%%%%%%%%%%%%%%%%%%%%%%%%%%%%%%%%%%%%%%%%%%%%%%%%%%%%%%%%%%%%
\subsection{Infinitely factorizable  path sets}\label{subsec:56}

We deduce that the collection $\sC^{\infty}(\sA)$ of all infinitely factorizable path sets on $\sA$  is closed under all interleaving operations. 

%%%%%%%%%%%%%%%%%%%%%%%%%%%%%%%%%%%%%%%%%%%%%%%%%%%%%%%
% Theorem 6.7 (5.12)
%.
%%%%%%%%%%%%%%%%%%%%%%%%%%%%%%%%%%%%%%%%%%%%%%%%%%%%%%%

\begin{thm}\label{thm:infinite_factorizable_clsd} 
{\em ($\sC^{\infty} (\sA)$ is closed under interleaving)}
    
       (1) If $\sP$ is a path set on the alphabet $\sA$ having an $n$-fold interleaving factorization
for all $n \ge 1$, then each interleaving factor $\psi_{j,n}(\sP)$ is itself infinitely factorizable.

      (2) Conversely,  if the $n$ path sets  $\{ \sP_i : \, 0 \le i \le n-1\}$ are each infinitely factorizable then
the $n$-fold interleaving  $\sP := \sP_0 \pr \sP_1 \pr \cdots \pr \sP_{n-1}$ is 
infinitely factorizable. 
\end{thm} 

%%%%%%%%%%%%%%%%%%%%%%%%%%%%%
% Proof of Theorem 6.7(infinitely factorizable(
%%%%%%%%%%%%%%%%%%%%%%%%%%%%%
\begin{proof}
Statement (1) 
of Theorem \ref{thm:infinite_factorizable_clsd}  follows from  Corollary  \ref{cor:62} combining it  with the fact that all interleaving factors  $\psi_{j,m}(\sP)$ are path sets.
Statement  (2) follows  
from Corollary \ref{cor:63}, combining it  with  Theorem \ref{thm:intalgthm} to infer that $\sP$ is a path set.
\end{proof}

%%%%%%%%%%%%%%%%%%%%%%%%%%%%%%%%%%%%%%%%%%%%%%%%%%%%%%%%%%%%%
%
% SUBSECT 6.3.Size of minimal right-resolving presentations of interleaving factors
%
%%%%%%%%%%%%%%%%%%%%%%%%%%%%%%%%%%%%%%%%%%%%%%%%%%%%%%%%%%%%%
\subsection{Size of minimal right-resolving presentations for interleaving factors of path sets }\label{subsec:57N}

We now suppose that a path set $\sP$ has an $n$-fold interleaving factorization. 
The following bound on the size of minimal right-resolving presentations of
interleaving factors (which are necessarily principal decimations)  improves  on the upper  bound of   Theorem \ref{thm:decimation_alg} for general $n$-level decimations.

%%%%%%%%%%%%%%%%%%%%%%%%%%%%%%%%%
% THM 6.8 (
%%%%%%%%%%%%%%%%%%%%%%%%%%%%%%
\begin{thm}\label{thm:56}
{\rm ( Upper bound on minimal presentation size  of $n$-fold interleaving factors ) } 
Let   $\sP$ be a path set  having $m$ vertices in its minimal right-resolving presentation. 
Suppose that $\sP$ has   an  $n$-fold interleaving factorization
$\sP = (\pr)_{j=0}^{n-1} \sP_j$. Then  each $n$-fold interleaving factor $\sP_j = \psi_{j,n}(\sP)$
has a minimal right-resolving presentation  having
at most $m$  vertices.
\end{thm} 

\begin{proof}
According to  the equivalences in Theorem \ref{thm:minimal_presentation},  
It suffices to show that for each $j$, the number of distinct word path sets of the form $\sP_j^w$, where $w$ is an initial word of $\sP_j$, is no larger than the number of distinct 
word path sets 
of the form $\sP^{w'}$, where $w'$ is an initial word of $\sP$. 
For an initial word $w=w_0w_1\ldots w_{k-1}$ of $\sP_j$, let $z\in\sP$, and say $z=z^0\pr z^1\pr\ldots\pr z^{n-1}$, 
with $w$ the initial $k$-length word of $z^{j}$. Let $w'$ be the initial word of $z$ of length $nk+j$. Then the letter in $z$ immediately following $w'$ is from $z^j$.

We assert  that  $\psi_{0,n}(\sP^{w'})=\sP_j^w$. We show both inclusions hold. 
The set $\sP^{w'}$ is the set of all infinite words $x$ such that $w'x\in\sP$. Note that if $w'x\in\sP$, then $\psi_{j,n}(w'x)=wy$, 
where $y=\psi_{0,n}(x)$. Thus, the $(0,n)$ decimation of any infinite word following $w'$ in $\sP$ is an infinite word following 
$w$ in $\psi_{j,n}(\sP)=\sP_j$, so we have the inclusion  $\psi_{0,n}(\sP^{w'})\subseteq\sP_j^w$. 
For the other inclusion, note that for any $y\in\sP_j^w$, 
we have $wy\in\sP_j$. Hence if we define $\tilde{z}=z^0\pr \ldots \pr z^{j-1}\pr (wy) \pr z^{j+1}\pr\ldots \pr z^{n-1}$, then $\tilde{z}\in \sP$, and the initial word of 
$\tilde{z}$ of length $nk+j$ is $w'$, since the choice of $y$ does not affect the first $nk+j$ letters of $\tilde{z}$. 
Thus $\tilde{z}=w'x$ for some $x\in\sP^{w'}$, and $\psi_{0,n}(x)=y$. Hence $y\in\psi_{0,n}(\sP^{w'})$, and we get  
$\sP_j^w\subseteq\psi_{0,n}(\sP^{w'})$, proving the assertion.

Thus, every path set $\sP_j^w$ is the $(0,n)$-decimation of a path set $\sP^{w'}$. It follows that there are at least as many distinct path 
sets of the form $\sP^{w'}$ as there are of the form $\sP_j^{w}$. 
\end{proof}

%%%%%%%%%%%%%%%%%%%%%%%%%%%%%%%%%
% Remark 6.9
%%%%%%%%%%%%%%%%%%%%%%%%%%%%%%
\begin{rem}\label{rem:55} 
 There is nothing special about the use of $(0,n)$-decimations in this proof. If, after choosing $j$, $w$, and $z$, we had chosen the word $w'$ to be of length $nk+j-i$, for any $0\leq i\leq n-1$, then the letter in $z$ occurring $(i+1)$ steps after the last letter of $w'$ would be from $z^j$, and we could have shown that every path set $\sP^w$ is the $(i,n)$-decimation of a path set $\sP^{w'}$.  
\end{rem}

%%%%%%%%%%%%%%%%%%%%%%%%%%%%%%%%%%%%%%%%%%%%%%%%

%
% SECTION 7: INFINITELY FACTORIZATION
%
%%%%%%%%%%%%%%%%%%%%%%%%%%%%%%%%%%%%%%%%%%%%%%%%%%%%%%%%
\section{Structure of Infinitely Factorizable Path Sets} \label{sec:6}

This section classifies all infinitely factorizable path sets, in terms of the structure of
 their minimal right-resolving presentation. 
 It deduces an improved upper bound on
 the size of minimal right-resolving presentations of interleaving factors of a general path set
 (which are decimations)  than that derived for decimations in Theorem \ref{thm:decimation_alg2}.

 %%%%%%%%%%%%%%%%%%%%%%%%%%%%%%%%%%%%%%%%%%%%%%%%%%%%%%%
 %
% Section 7.1 Infinitely factorizable path sets
%
%%%%%%%%%%%%%%%%%%%%%%%%%%%%%%%%%%%%%%%%%%%%%%%%%%%%%%%
\subsection{Characterization of infinitely factorizable path sets} \label{subsec:62}

We characterize the path sets $\sP$ that are infinitely factorizable 
as having 
 a minimal  right-resolving presentation $(\sG, v)$
of a particularly simple kind.

%%%%%%%%%%%%%%%%%%%%%%%%%%%%%%
% Defn 7.2 Leveled Presnation
%%%%%%%%%%%%%%%%%%%%%%%%%%%%
\begin{defn}\label{defn:leveled} 
(Leveled presentation) 
 A   presentation 
 $(\sG, v)$ of a path set $\sP$ is {\em leveled} 
if  it is right-resolving and all infinite paths in  $\sG$ from the marked vertex $v$ visit exactly the same set of vertices in the same order; i.e., all
exit edges of $\sG$  from a vertex $v'$ necessarily go the same target vertex $v^{''}$ (depending on $v'$).
There  may  be multiple edges (with different symbol labels) between $v'$ and $v''$. 
We  say that a path set $\sP$ is {\em leveled } if it has such a  presentation; otherwise it is {\em non-leveled}. 
 \end{defn}

 Figure 7.1  gives an example of a leveled presentation.

%%%%%%%%%%%%%%%%%%%%%%
% Figure 6.1
%%%%%%%%%%%%%%%%%%%%%%
\begin{figure}[ht]\label{fig61}
	\centering
	\psset{unit=1pt}
	\begin{pspicture}(-150,30)(100,135)
		\newcommand{\noden}[2]{\node{#1}{#2}{n}}
		\noden{$v_0$}{-140,90}
		\noden{$v_6$}{0,125}
		\noden{$v_7$}{-60,105}
		\noden{$v_2$}{-80,45}
		\noden{$v_1$}{-140,45}
		\noden{$v_5$}{60,105}
		\noden{$v_4$}{80,45}
		\noden{$v_3$}{0,45}
		\bline{n$v_5$}{n$v_6$}{1}
		\eeline{n$v_6$}{n$v_7$}{3}{0}{1}
		\bline{n$v_7$}{n$v_2$}{3}
		\bline{n$v_2$}{n$v_3$}{1}
		\bline{n$v_4$}{n$v_5$}{3}
		\ddline{n$v_1$}{n$v_2$}{0}{2}
		\ddline{n$v_3$}{n$v_4$}{1}{2}
		\bline{n$v_0$}{n$v_1$}{1}
	\end{pspicture}
\newline
\hskip 0.5in {\rm Figure 7.1} Leveled presentation  of a   path set $\sP$. The marked vertex is $v_0$.
\newline
\end{figure}
% End Figure 7.1

%%%%%%%%%%%%%%%%%%%%%%%%%%%%%%%%%%%%%%%%%%%%%%%%%%%%%%%
% Theorem  7.3 
%%%%%%%%%%%%%%%%%%%%%%%%%%%%%%%%%%%%%%%%%%%%%%%%%%%%%%%

\begin{thm}\label{thm:63}
 A path set $\sP$ is infinitely factorizable if and only if it has a minimal right-resolving presentation $(\sG, v_0)$ that
 is leveled. 
 \end{thm}

\begin{proof}
To prove necessity, we must show  that if $\sP$ is infinitely factorizable, then its minimal right-resolving
 presentation is leveled. 
 We prove the 
 assertion by induction on the number of vertices reached from the inital vertex  in G,  
starting from the marked  vertex $v_0$.  
 Using Theorem \ref{thm:61},  condition (iii) for being infinitely factorizable says 
that 
$ \sP = \prod_{k=0}^{\infty} \mathcal{A}_k, $
where each $\mathcal{A}_k$   is a subset of the (finite) alphabet $\mathcal{A}$.
Each exit  edge from the marked  vertex $v_0$ of $(\sG,v_0)$  goes to a vertex $v'$ whose vertex path set $X(\sG, v')$  must be
$  \sP' = \prod_{k=1}^{\infty} \mathcal{A}_k. $
The (finite) vertex  follower set  $\mitF(\sG, v')$ is then  
$$
     \mitF(\sG, v') = \bigcup_{m=1}^{\infty} \, \prod_{i=1}^m \mathcal{A}_k.
$$
By Theorem \ref{thm:minimal_presentation}  all vertex  follower sets in $\sG$ are distinct.
Consequently there can be only one choice for $v''$,   and all exit edges from $v_0$ go to it.
If $v''=v_0$ we have a self-loop at $v$ and are done. Otherwise $v''$  is  a new
vertex; call it $v_1$. There may be multiple edges (with different labels) from $v_0$ to $v_1$;
 the labels are exactly the letters in $\sA_0$. 

The induction hypothesis on $v_j$ supposes that the 
vertex $v_j$ has associated vertex path set
$\sP_j = \prod_{k=j}^{\infty} \mathcal{A}_k.$ 
We next study exit edges  from $v_j$. They necessarily go to a vertex $v''$
whose vertex path $X(\sG, v^{"} )$ is
$  \sP_{j+1}^{'}= \prod_{i=2}^{\infty} \mathcal{A}_k$
and whose (finite) vertex follower set is 
$$
       \mitF(\sG, v'') = \bigcup_{m=2}^{\infty} \,  \prod_{i=1}^m \mathcal{A}_k.
$$
By uniqueness of a vertex in the minimal presentation having a  particular finite follower set, all exit edges
must go to the same vertex $v''$. If $v''$ is one of the previously found  vertices, we are done.
Otherwise we are at  a new vertex $v_{j+1}$. 
Since $\sP$ is a path set it has a presentation with finitely many
vertices so the process must terminate.  
The induction is complete, so $(\sG, v_0)$ is leveled. 

Suppose such the leveled presentation  $\sG$ has
$m$ vertices.  The  directed graph $\sG$  either has a unique vertex path that is an $m$-cycle
or else this graph has the appearance of a Greek letter $\rho$, with the  unique vertex path
having a preperiodic part $v_0 \to v_1 \cdots \to v_{s-1}$,  with $s$ vertices, followed by moving around a 
periodic part $v_{s} \to v_{s+1}, \ldots v_{s+p-1} \to v_{s}$, a period $p$ cycle, with $p= m -s$.  

To prove sufficiency, we must show that every path set $\sP$ having a leveled presentation $(\sG, v_0)$  is infinitely factorizable.
A leveled presentation is right-resolving, since the labels $\sA_k$ exiting from vertex $v_k$ are distinct. 
It is clear from the  internal structure of a leveled presentation (as a rho-graph)  that the associated path set $\sP=(\sG, v_0)$ necessarily has 
the form
$\sP = \prod_{k=0}^{\infty} \sA_k$
where for the first $m$ steps $\sA_k$ are the set of edge labels for vertices 
$v_0, v_1, \ldots , v_{m-1}$. After this point the edge labels repeat periodically
with a period $p=m-s$, where $\ell$ is the length of the  pre-period, having 
the equality of sets 
$\sA_{m+j} = \sA_{\ell +k}$ with 
$0 \le k \le p-1$ determined by the congruence $k \equiv  j (\bmod  p)$.
Since this presentation satisfies  Theorem \ref{thm:61}(iii), $\sP$ is 
infinitely factorizable.  Finally, the presentation is minimal using Theorem \ref{thm:minimal_presentation}, 
because all vertex path sets $X(\sP, v_i)$ fo $0 \le i \le m-1$
are distinct, whence all finite follower sets $\mitF(\sG, v)$ are distinct. 
\end{proof}

%%%%%%%%%%%%%%%%%%%%%%%%%%%%%%%%%%%%%%%%%%%%%%%%%%%%%%%
 %
 % Section 7.2 Bounds for number of  Infinitely factorizable path sets
%
%%%%%%%%%%%%%%%%%%%%%%%%%%%%%%%%%%%%%%%%%%%%%%%%%%%%%%%
\subsection{Bounds for the number of distinct factors of infinitely factorizable path sets} \label{subsec:63}

We upper bound the number of distinct interleaving factors $\fF(\sP)$ of a  infinitely factorizable  path set in terms of the size of
its minimal right-resolving presentation. Since all factors are decimations we  know by Theorem \ref{thm:decimation_set_bound}
that $\fF(\sP) \subseteq \ffD(\sP)$  is finite.

 %%%%%%%%%%%%%%%%%%%%%%%%%%%%%%%%%%%%%%%%%%%%%%%%%%%%%%%
% Theorem 7.4
%%%%%%%%%%%%%%%%%%%%%%%%%%%%%%%%%%%%%%%%%%%%%%%%%%%%%%%
\begin{thm} \label{thm:monfactfinprop} 
 Suppose that $\sP$ is an infinitely factorizable path set that has a right-resolving presentation $(\sG, v)$
 with $m$ vertices.

 (1) Each possible distinct factor occurs in some $n$-fold factorization having $n \le 2m-1$.

(2) The  cardinality $|\fF(\sP)|$ of the factor set $\fF(\sP)$ is at most  $m^2$.
\end{thm}
\begin{proof} 
It suffices to consider  the minimal right-resolving presentation $(\sG, v_0)$ of $\sP$,
which must be a leveled presentation, and which has at most $m$ vertices. Let $p$ be the period of the graph. 
There is a unique vertex path on $\sG$ starting from the initial vertex, which we consider to be the 0th vertex. 
Call the vertex reachable in 1 step from the 0th vertex the 1st vertex, and so on. Now $\mathcal{A}_k(\sP)$ is the set
 of symbols available at the $k$-th vertex, and by Theorem \ref{thm:61}, $\sP=\prod_{k=0}^{\infty}\mathcal{A}_k(\sP)$. 
 Then $\psi_{j,n}(\sP)=\prod_{k=0}^{\infty}\mathcal{A}_{kn+j}(\sP)$. Now, since there are only $m$ vertices (the 0th vertex
  through the $(m-1)$st vertex), we may choose $j'$ so that the $j'$th vertex is the $j$th vertex, 
  which also gives us that the $(j'+1)$st vertex is the $(j+1)$st vertex, and so on. Hence:

\[\psi_{j,n}(\sP)=\prod_{k=0}^{\infty}\mathcal{A}_{kn+j}(\sP)=\prod_{k=0}^{\infty}\mathcal{A}_{kn+j'}(\sP)=\psi_{j',n}(\sP)\]

Likewise, because there are only $m$ vertices, we may choose $n'<m$ so that the $j'+n'$th vertex is the $j'+n$th vertex. 
We wish to show that this will also imply that the
 $(j'+kn')$th vertex is the $(j'+kn)$th vertex for all $k$. If $n<m$, we may take $n'=n$. If $n\geq m$, then the $(j'+n)$th vertex is 
 in the periodic part of the graph, so the fact that the $(j'+n')$th vertex is the $(j'+n)$th vertex implies that $n'=n~(\bmod~p)$. 
 This in turn implies that the $(j'+kn)$th vertex is the $(j'+kn')$th vertex for all $k$. Hence we have:

\[\psi_{j,n}(\sP)=\prod_{k=0}^{\infty}\mathcal{A}_{kn+j'}(\sP)=\prod_{k=0}^{\infty}\mathcal{A}_{kn'+j'}(\sP)\]

Since $j'$ and $n'$ are both chosen between $0$ and $m-1$, there are at most $m^2$ distinct sets of this form, whence 
the cardinality of the factor set $\fF(\sP)$  is at most $m^2$, proving (2). 

However, we have not guaranteed that $j'<n'$, and so the set 
$\prod_{k=0}^{\infty}\mathcal{A}_{kn'+j'}(\sP)$ is not guaranteed to be one of the $n'$-fold interleaving factors of $\sP$. 
If $n<m$, then we have $j'\leq j<n=n'$, and so we are done. If $n\geq m$, then because the $(j'+n')$th vertex is in the periodic part of the graph,
 we may take $n''=n'+pr$ for any $r\geq 1$ and get that the $(j'+kn'')$th vertex is the $(j'+kn')$th vertex for all $k\geq0$. Since $j'<m$ and $p\leq m$,
It  is always possible to choose $r$ such that $j'<n''\leq 2m-1$. We then have:
\[\psi_{j,n}(\sP)=\prod_{k=0}^{\infty}\mathcal{A}_{kn'+j'}(\sP)=\prod_{k=0}^{\infty}\mathcal{A}_{kn''+j'}(\sP)=\psi_{j',n''}(\sP)\]
which proves (1).
\end{proof}

%%%%%%%%%%%%
% Remark 7.5  [Algorithm]
%%%%%%%%%%%%
\begin{rem} \label{rem:68}
The bound $2m-1$ in Theorem \ref{thm:monfactfinprop}  is sharp. Consider a circular graph with $m$ vertices, where the available alphabets 
at all vertices are distinct. Mark one of the vertices. If $\sP$ is the path set presented, then it can be shown that $\psi_{m-1,2m-1}(\sP)$ 
(which is  the full shift over the alphabet available at the $(m-1)$st vertex) is not an $n$-fold interleaving factor for any $n<2m-1$.
\end{rem}

%%%%%%%%%%%%%%%%%%%%%%%%%%%%%%%%%%%%%%%%%%%%%%%%%%%%%%%
 %
 % Section 7.4 Bounds for number of  Infinitely factorizable path sets
%
%%%%%%%%%%%%%%%%%%%%%%%%%%%%%%%%%%%%%%%%%%%%%%%%%%%%%%%
\subsection{Minimal right-resolving presentations for interleaving factors-part 2 }\label{subsec:64}

Theorem \ref{thm:56} showed that any path set $\sP$ having a minimal right-resolving presentation with $m$ vertices  
that has an $n$-fold interleaving
factorization $\sP = (\pr)_{i=0}^{n-1} \sP_j$ , necessarily has every factor $\sP_j= \psi_{j,n}(\sP)$ has
 a minimal right-resolving presentation having 
$m$ or fewer vertices. We now show that the equality case implies all these sets are leveled.
%imposing the additional hypothesis that $\sP$  has an $n$-fold interleaving factorization.

%%%%%%%%%%%
% Theorem 6.6
%%%%%%%%%%
\begin{thm} \label{thm:68}
Let  $\sP$ be a path set having a minimal right-resolving presentation with  $m$  vertices.
Suppose  that $\sP$ has an $n$-fold interleaving factorization $\sP = (\pr_n)_{i=0}^{n-1} \sP_j$
such that at least one  factor $\sP_j= \psi_{j,n}(\sP)$ has a minimal right-resolving
presentation with $m$ vertices.  
Then  the path set $\sP$ must be  leveled. 
and all of the factors $\sP_j$ for $0 \le j \le n-1$ are leveled. 
\end{thm}

\begin{proof}
Suppose $\sP=X(\Gr,v)$ and $\sP_i=X(\Gr_i,v_i)$ are minimal right-resolving presentations, where 
$\Gr$ and $\Gr_i$ both have $m$ vertices. 
We begin by showing that the hypotheses of the theorem imply the following:

\smallskip

{\bf Claim.}{\em ~All word path sets $\sP^w$ are determined by their $(0,n)$-decimations $\psi_{0,n}(\sP^w)$.}

\smallskip

Since $\Gr$ and $\Gr_i$ both have $m$ vertices with distinct follower sets, they both have $m$ distinct vertex path sets $X(\Gr, v')$ 
 and $X(\Gr_i, v_i^{'})$ that can be presented by choosing initial vertices. The proof of Theorem \ref{thm:56} established that
 every word path set $\sP_i^w$ is a set of the form $\psi_{0,n}(\sP^{w'})$, for another word $w'$.
  Our claim follows immediately from the fact that there are $m$ such distinct path sets, but also only $m$ distinct path sets $\sP^{w'}$. 

Now, we will show that $\sP_j$ is leveled, for all $0\leq j\leq n-1$, using only the key fact stated in our claim
 (in particular, we no longer need to use any special information about $\mathcal{P}_i$, and what follows works for $j=i$ as well as for $j\neq i$). 
 Let $w^{j,1}$ and $w^{j,2}$ be initial words of $\sP_j$, both of length $k$. 
Say they are the initial blocks of the infinite words $x^{j,1}$ and $x^{j,2}$, respectively. It will suffice to show 
the equality of word path sets $\sP_j^{w^{j,1}}=\sP_j^{w^{j,2}}$, 
since this equality is equivalent to the equality of the word follower sets  
$\mitF_{\sP_j} (w^{j,1})$ and $\mitF_{\sP_j} (w^{j,2})$ of $\mathcal{P}_j$,
using Theorem \ref{thm:initialblocks}, because these word follower sets are the initial
words $\sB^{I}( \sP_j^{w^{j,1}})$ and $\sB^{I}(\sP_j^{w^{j,2}})$, respectively. 
This equality of the word follower sets then implies the equality of the
vertex follower sets $\mitF(\sP_j, w^{j,1})$ and $\mitF(\sP_j, w^{j,2})$, which since the presentation of
$\sP_j$ is minimal right-resolving means that we arrived at the same vertex of $\sP_j$ following the
symbol paths for $w^{j,1}$ and $w^{j,2}$ from the initial vertex $v_j$ of  this presentation. The property of
being at the same vertex is exactly the desired leveling property of $\sP_j$. 

It remains to show that $\sP_j^{w^{j,1}}=\sP_j^{w^{j,2}}$.
Now for all $l\neq j$, let $x^l\in\sP_l$ be chosen arbitrarily. Then let  
$$
y=x^0\pr x^1\pr \ldots \pr x^{j,1}\pr x^{j+1}\pr  \ldots\pr x^{n-1} \,\, \mbox{and} \,\,
z=x^0\pr  x^1\pr \ldots \pr  x^{j,2}\pr x^{j+1}\pr\ldots\pr x^{n-1}.
$$
 Let $b^1$ and $b^2$ be the words made up of the first $(n(k-1)+j+1)$ entries of $y$ and $z$, respectively. 
 In particular, the last entry of $b^1$ is the last entry of $w^{j,1}$, and the last entry of $b^2$ is the last entry of $w^{j,2}$. 

We will show:

$$
\sP_j^{w^{j,1}}=\psi_{n-1,n}(\sP^{b^1})\,\, \mbox{and} \,\, \sP_j^{w^{j,2}}=\psi_{n-1,n}(\sP^{b^2}).
$$

by reasoning along similar lines as in the proof of Theorem \ref{thm:56}. Specifically, for the first equality, $\sP^{b^1}$ is the 
set of all infinite words $x$ such that $b^1x\in\sP$. and if $b^1x\in\sP$, then $\psi_{j,n}(b^1x)=w^{j,1}y$, where $y=\psi_{n-1,n}(x)$. 
(The reason that we have $(n-1,n)$-decimations  here instead of $(0,n)$-decimations is that the words $b^1$ and $b^2$ have length $n-1$
 less than $kn+j$, the length of the word used in the proof of Theorem \ref{thm:56}.) Thus, the $(n-1,n)$ decimation of any infinite word following
  $b^1$ in $\sP$ is an infinite word following $w^{j,1}$ in $\psi_{j,n}(\sP)=\sP_j$, and we have  $\psi_{n-1,n}(\sP^{w'})\subseteq\sP_j^{w^{j,1}}$. 
  On the other hand, for any $y\in\sP_j^{w^{j,1}}$, we have $w^{j,1}y\in\sP_j$. Hence 
  $\tilde{z}=x^0\pr \ldots \pr x^{j-1}\pr (w^{j,1}y) \pr x^{j+1}\pr\ldots \pr x^{n-1}\in \sP$, 
  and the initial word of $\tilde{z}$ of length $(n(k-1)+j+1)$ is $b^1$, 
  since the choice of $y$ does not affect the first $(n(k-1)+j+1)$ letters of $\tilde{z}$. 
  Thus $\tilde{z}=w'x$ for some $x\in\sP^{w'}$, and $\psi_{n-1,n}(x)=y$. Hence $y\in\psi_{n-1,n}(\sP^{w'})$, 
  and we get  $\sP_j^w\subseteq\psi_{n-1,n}(\sP^{w'})$, proving that $\sP_j^{w^{j,1}}=\psi_{n-1,n}(\sP^{b^1})$.
  By the same argument (replacing $b^1$ with $b^2$ and $w^{j,1}$ with $w^{j,2}$), we get $\sP_j^{w^{j,2}}=\psi_{n-1,n}(\sP^{b^2})$.

Thus, if we can show that $\mathcal{P}^{b^1}=\sP^{b^2}$, then we get the desired equality $\sP_j^{w^{j,1}}=\sP_j^{w^{j,2}}$.

 But our earler claim
tells us that $\sP^{b^1}$ and
 $\sP^{b^2}$ are determined by their $(0,n)$-decimations. Their $(0,n)$-decimations are determined by the choice of (the first $k-1$ letters of) $y^{j+1}$, which is the same for $b^1$ and $b^2$ by construction. 
 
 Thus, we have shown that all the interleaving factors of $\mathcal{P}$ are leveled, and so $\sP$ is leveled.
\end{proof}

%%%%%%%%%%%%%%%%%%%%%%%%%%%%%%%%%%%%%%%%%%%%%%%%%%%%%%%%%%%%%
%
% SECTION 8: FINITE FACTORIZATION
%
%%%%%%%%%%%%%%%%%%%%%%%%%%%%%%%%%%%%%%%%%%%%%%%%%%%%%%%%%%%%%
\section{ Finitely Factorizable Path Sets} \label{sec:factorsec}
%Our strategy for determining the factor set of a path set $\mathcal{P}$ is as follows. 

Finitely factorizable path sets coincide with non-leveled path sets, so they are
algorithmically recognizable. 

 %%%%%%%%%%%%%%%%%%%%%%%%%%%%%%%%%%%%%%%%%%%%%%%%%%%%%%%
 %
% Section 8.1 Finitely factorizable path sets
%
%%%%%%%%%%%%%%%%%%%%%%%%%%%%%%%%%%%%%%%%%%%%%%%%%%%%%%%
\subsection{Bounds for number of distinct $n$-fold interleaving factorizations} \label{subsec:71}

%%%%%%%%%%%%%%%%%%%%%%%%%%%%%%%%%%%%%%%%%%%%%%%%%%%%%%%
% Theorem 8.1
%%%%%%%%%%%%%%%%%%%%%%%%%%%%%%%%%%%%%%%%%%%%%%%%%%%%%%%
\begin{thm} \label{thm:non_leveled_fact} 
%{\rm (Factorization under interleaving)}
Let $\mathcal{P}$ be a path set having a right-resolving presentation $(\sG, v)$
having $m$ vertices. If $\mathcal{P}$ is finitely factorizable,  i.e., non-leveled, and has an $n$-fold interleaving factorization,
then $n \le m-1$. 
\end{thm}

\begin{proof}
To prove the bound, we will assume $\sP$ has an $n$-fold interleaving factorization, 
and show that a minimal right-resolving presentation of $\sP$ must have at least $n+1$ vertices. If so we must have $m \ge n+1$,
giving the result.

We are given the interleaving factorization 
$\sP = \sP_0 \pr \sP_1 \pr \cdots \pr \sP_{n-1}$.
We let $(\sG_i, v_i^0)$ for $0 \le i \le n-1$ be minimal right-resolving
presentations for each $\sP_i$. In particular, by Theorem \ref{thm:minimal_presentation},
each labeled graph $\sG_i$ has the property that
all  its vertices have distinct vertex follower sets.

One of the  $\sP_i$ must be  finitely factorizable. For if all of the $\sP_i$ were infinitely factorizable,
then by Corollary \ref{cor:63} their $n$-fold interleaving $\sP := \sP_0 \pr \sP_1 \pr \cdots \pr \sP_{n-1}$ would be infinitely factorizable,
a contradiction.  For definiteness suppose $\sP_{i_0}$ is finitely factorizable. Then $\sP_{i_0}$ is  non-leveled, by Proposition \ref{thm:63}. 
Because $\sP_{i_0}=X(\sG_{i_0},v_{i_0}^0)$ is non-leveled, the graph $\sG_{i_0}$ has a reachable vertex $w$ which 
has exit edges to two different vertices---call  them $w_1$ and $w_2$---which have  distinct vertex follower sets $\mitF(\sG_{i_0}, w_1)$ and $\mitF(\sG_{i_0}, w_2)$  by minimality of the presentation $X(\sG_{i_0},v_{i_0}^0)$.

We now study the $n$-fold interleaving graph product $(\mathcal{H}, \bv^0) := (\pr_n)_{i=0}^{n-1}(\sG_i,v_i^0)$ studied in Theorem \ref{thm:algthm}.
We will  show that $(\mathcal{H}, \bv^{0})$ has at least $n+1$ different vertex  follower sets, over all vertices reachable from $\bv^{0}$. 
If so, by Proposition \ref{onevertexforeachfollowerset} 
 $\sP = X(\mathcal{H}, \bv^{0})$ will have at least $n+1$ distinct  word follower sets, and Theorem \ref{thm:minimal_presentation} then
implies that the minimal right-resolving presentation $(\sG, v)$ of $\sP$ has at least $n+1$ vertices. Consequently $m \ge n+1$  and we are done.

Recall that the vertex set $V(\sH) = \bigcup_{i=1}^{n-1} \sV^{i}$, but that the graph $\sH$ need not be connected,
Our argument must establish that the $n+1$  vertex follower sets constructed are in the reachable component  of $\bv$.
We first consider a shortest directed path in $\sG_{i_0}$  from the initial vertex $v_{i_0}$  to the vertex $w$---call its length $k$---and denote it $v_{i_0}^0,v_{i_0}^1,v_{i_0}^2,\ldots,v_{i_0}^{k-1}$, with $v_{i_0}^{k-1}=w$.
  Then  let $v_i^0,v_i^1,v_i^2,\ldots$ denote an arbitrary vertex path from the initial vertex in $\mathcal{G}_i$.

We know that the initial vertex of $\mathcal{H}$ is $\bv^0=(v_0^0,v_1^0,\ldots,v_{n-1}^0) \in \sV^0$. 
By the construction in the proof of Theorem \ref{thm:algthm}, there is an edge from this vertex to the 
vertex $(v_1^0,v_2^0,\ldots,v_{n-1}^0,v_0^1) \in \sV^{1}$, from here to $(v_2^0,v_3^0,\ldots,v_{n-1}^0,v_0^1,v_1^1)\in \sV^2$,  eventually reaching after $kn+i_0$ steps the vertex

 $$
 \bv_0:= (v_{i_0}^{k-1} ,v_{i_0+1}^{k-1},\ldots v_{n-1}^{k-1},v_0^k,v_1^k,\ldots,v_{i_0-1}^k) \in \sV^{i_0} 
 $$
 where 
 $\sV^{i_0}= \sV_{i_0} \times \sV_{i_0+1} \times \cdots \times \sV_{n-1} \times \sV_0 \times \sV_1 \times \cdots \sV_{i_0-1}.$
 For the next step, we have two choices, moving $w \to w_{\ell}$ for $\ell=1,2$, and 
we can then reach, in sequence, the following $n$ vertices in $\sH$:

\begin{eqnarray*}\label{nvertices1}
&\bv_1(\ell) := (v_{i_0+1}^{k-1},\ldots v_{n-1}^{k-1},v_0^k,v_1^k,\ldots,v_{i_0-2}^k,v_{i_0-1}^k,w_{\ell}) \in \sV_{i_0+1}&
\\
&\bv_2(\ell) := (v_{i_0+2}^{k-1},\ldots v_{n-1}^{k-1},v_0^k,v_1^k,\ldots,v_{i_0-1}^k,w_{\ell},v_{i_0+1}^{k}) \in \sV_{i_0+2} &
\\
&.&\\
&\bv_n(\ell) := (w_{\ell},v_{i_0+1}^{k},\ldots v_{n-1}^{k},v_0^{k+1},v_1^{k+1},\ldots,v_{i_0-1}^{k+1})\,  \in \,\sV_{i_0} \quad \quad&
\end{eqnarray*}
The  last  $n-1$ of these steps will be  chosen to travel edges in $\mathcal{H}$ corresponding to identical edges for $\ell=1, 2$ in  the graphs $\sG_{i_0+j}$ for  $1 \le j  \ne n-1$.  
We have now  specified  a list  of $2n$  vertices in $\mathcal{H}$ that are reachable from its initial vertex $\bv^{0}$.

We will show  that among these $2n$ vertices there are at least $n+1$ distinct vertex follower sets in $\sH$. 
To tell vertex follower sets apart we use the following test.

\smallskip

%%%CLAIM%%%
{\bf Claim.} {\em Given two vertex follower sets $\sF( \sH, \by^1)$ and $\sF(\sH,\by^2)$ with vertices in $\sH$ in vertex subsets 
$\by^1\in \sV^{i_1}$ and $\by^2 \in \sV^{i_2}$, respectively,  a necessary condition for the vertex follower set equality 
\begin{equation}\label{eqn:follower-set-match}
\sF( \sH, \bv^1) = \sF(\sH,\bv^2)
\end{equation}
is that each of their projected vertices in  the individual graphs $\sG_i$ for $0 \le i \le n-1$
must have identical vertex  follower sets, 
\begin{equation}\label{eqn:follower-set-pre-match}
\sF( \sG_{i_1+ j} , \by^1(j) ) = \sF(\sG_{i_2+j} ,\by^2(j)), \quad \mbox{for} \quad 0 \le j \le n-1.
\end{equation} 
These conditions are equivalent, for each $j$,  to the path set equalities
%Another set of necessary conditions are the path set equalities
\begin{equation}\label{eqn:path-set-pre-match}
X(\sG_{i_1+j}, \by^{1}(j)= \psi_{j, n}(X( \sH, \by^1) ) = \psi_{j, n}(X( \sH, \by^2))=X(\sG_{i_2+j}, \by^{2}(j)).
\end{equation}
}
%%%END- CLAIM%%%

To prove  the claim, note that following an infinite path  in the graph $(\sH, \bv^{0})$ is the same thing as  following   separate independent paths   in  each of the 
$n$ graphs $(\sG_i, v_i)$, starting from their initial vertices.
Write $\bv^j = (v^j(0), v_j(1), \cdots, v_j(n-1))$
 The   finite follower set equality \eqref{eqn:follower-set-match} implies the path set equality
$X( \sH, \by^1) = X(\sH, \bv^2).$  Note that $\psi_{j,n}(X( \sH, \by^{\ell} )) = X(\sG_{i_1+j}, \bv^{\ell}(j)),$  for all $j$.
The vertex follower sets of the projections
are  completely determined by the property that they are  initial languages of the path sets $\psi_{j,n}(X(\sH, \by^{\ell}))$.
Therefore the equality $\sF( \sG_{i_1+j} , v_i (1) ) = \sF (\sG_{i_2+j},v_i (2))$  holds for all $0 \le j \le n-1$, proving the claim.

Thus, for example, for the vertex follower set of vertex
 $$(v_{i_0+1}^{k-1},\ldots v_{n-1}^{k-1},v_0^k,v_1^k,\ldots,v_{i_0-1}^k,w_2) \in \sV^{1_0+1}$$
  to be equal to the vertex follower set of 
$$(v_{i_0+2}^{k-1},\ldots v_{n-1}^{k-1},v_0^k,v_1^k,\ldots,w_1,v_{i_0+1}^{k}) \in \sV^{i_0+2},$$
we would need to have the row of path sets
\[
\left(
X(\mathcal{G}_{i_0+1},v_{i_0+1}^{k-1}),
X(\mathcal{G}_{i_0+2},v_{i_0+2}^{k-1}),
\ldots,
%X(\mathcal{G}_{i_0-2},v_{i_0-2}^{k}),
X(\mathcal{G}_{i_0-1},v_{i_0-1}^{k}),
X(\mathcal{G}_{i_0},w_2)
\right)
\]
to be identical with the row of path sets
\[
\left(
X(\mathcal{G}_{i_0+2},v_{i_0+2}^{k-1}),
X(\mathcal{G}_{i_0+3},v_{i_0+3}^{k-1}),
\ldots,
X(\mathcal{G}_{i_0},w_1),
X(\mathcal{G}_{i_0+1},v_{i_0+1}^{k})
%X(\mathcal{G}_{i_0+2},v_{i_0+2}^{k})
\right).
\]

We want to make sure that not too many coincidences of vertex follower sets $\mathcal{H}$ can occur in this manner. 
For each of these $2n$ vertices, consider the associated row of path sets as above.
 Considering all of these rows together gives us path sets filling in two arrays of the following schematic shape,
where the entries are path sets. Here $\bar{a} = X(\sG_{i_0}, w_1)$ and $\bar{b} = X(\sG_{i_0}, w_2)$.

\begin{equation*}
A=\left(
\begin{array}{ccccc}
x_{i+1} 	&	x_{i+2}	&	  \cdots 	&	x_{i-1}	&	\bar{a}	\\
x_{i+2}	&	x_{i+3}	&	 \cdots 	&	\bar{a}	&	y_{i+1}	\\
%x_{i+3}      &      x_{i+4}      &       \cdots         &       y_{i+1} &        y_{i+2}       \\
\vdots 	&	\vdots 	&	  \ddots 	&	\vdots 	&	\vdots 	\\
x_{i-1}	&	\bar{a}	& 	  \cdots 	&	 y_{i-3} 	&	y_{i-2}	\\
\bar{a} 	&	y_{i+1}	& 	 \cdots 	&	 y_{i-2} 	&	y_{i-1} \\
\end{array}
\right) \,\, 
%\quad \mbox{and} \quad
%
%\begin{equation*}
B=\left(
\begin{array}{ccccc}
x_{i+1} 	&	x_{i+2}	&	 \cdots 	&	x_{i-1}	&	\bar{b}	\\
x_{i+2}	&	x_{i+3}	&	 \cdots 	&	\bar{b}	&	y_{i+1}	\\
%x_{i+3}      &      x_{i+4}      &       \cdots         &       y_{i+1} &        y_{i+2}       \\
\vdots 	&	\vdots 	&	  \ddots 	&	\vdots 	&	\vdots 	\\
x_{i-1}	&	\bar{b}	& 	  \cdots 	&	 y_{i-3} 	&	y_{i-2}	\\
\bar{b} 	&	y_{i+1}	& 	 \cdots 	&	 y_{i-2} 	&	y_{i-1} \\
\end{array}
\right)
\end{equation*} 
We show that matrices of this schematic shape always  have at  least $n+1$
distinct rows, as  a special case  of the following  combinatorial lemma.

%%%%%%%%%%%%%%%%%%%%%
% Lemma 8.2
%%%%%%%%%%%%%%%%%%%%%
\begin{lem}\label{matlem}
For any set $X$, let $A$ and $B$ be $n \times n$ matrices with coefficients from $X$, and let $R_A$ be the set of distinct rows of $A$, $R_B$ be the set of distinct rows of B. If $A$ and $B$ have constant skew-diagonal entries  $a_{i, n+1-i}=\bar{a} $ and $b_{i,n+1-i}=\bar{b}$,  with  $\bar{a} \ne \bar{b}$,  and if in addition $a_{ij} = b_{ij}$ 
holds for all entries not on the skew-diagonal ( $i+j \ne n+1$), then $|R_A\cup R_B|\geq n+1$. 
\end{lem}

The assertion of Lemma \ref{matlem} completes the proof  of the theorem; we prove it below.
\end{proof}

%%%%%%%%%%%%%%%%%%%%%%
%  PRoof of  Lmma 8.2
%%%%%%%%%%%%%%%%%%%%%
\begin{proof}[Proof of Lemma \ref{matlem}]
We prove the result by induction on $n$. It  holds for  the base case $n=1$, for  $A$ has one entry: $\bar{a}$, and $B$ has a distinct entry: $\bar{b}$.

Now suppose the theorem holds for $n$, and let $A,B\in M^{(n+1)\times (n+1)}(X)$. 
Then we have two matrices of the following form:

\begin{equation*}
A=\left(
\begin{array}{ccccc}
x_{1,1}	&x_{1,2}&	\cdots 	&	x_{1,n}	 &	\bar{a}	\\
x_{2,1}	& x_{2,2} &		\cdots 	&	\bar{a}	 &	x_{2,n+1}	\\
\vdots 	&	\vdots &	\ddots 	&	\vdots 	 &	\vdots 	\\
x_{n,1}	&	\bar{a} &	\cdots 	&	x_{n,n} &	x_{n,n+1}	\\
\bar{a}	&	 x_{n+1,2}& 	\cdots 	&	x_{n+1,n}	 &	x_{n+1,n+1}
\end{array}
\right),  \quad
%\end{equation*}
%and
%begin{equation*}
B=\left(
\begin{array}{cccc}
x_{1,1} 	&		\cdots 	&	x_{1,n}	  &	\bar{b}	\\
x_{2,1}	&		\cdots 	&	\bar{b} 	  &	x_{2,n+1}	\\
\vdots 	&		\ddots 	&	\vdots 	  &	\vdots 	\\
x_{n,1}	&		\cdots 	&	x_{n,n}  &	x_{n,n+1}	\\
\bar{b}	&	 	\cdots 	&	x_{n+1,n}	  &	x_{n+1,n+1}
\end{array}
\right).
\end{equation*}

Let $R_A$ be the set of distinct rows of $A$, and $R_B$ be the set of distinct rows of B, and let $R=R_A\cup R_B$. 
We need to show $|R|\geq n+2$. Consider $A',B'\in M^{ n \times n}(X)$ defined as
the upper right $ n \times n$ corner of $A$ and $B$.
\begin{equation*}
A'=\left(
\begin{array}{cccc}
x_{1,2}	&		\cdots 	&	x_{1,n}	&    \bar{a}		\\
x_{2,2}	&		\cdots 	&	\bar{a}	&   x_{2,n+1}		\\
\vdots 	&	       \vdots 	&	\ddots 	&	\vdots 		\\
x_{n-1,2}   &               \cdots           &  x_{n-1,n}     &   x_{n-1,n+1} \\
\bar{a}	&	 	\cdots 	&   x_{n,n}	&	x_{n, n+1}		\\
\end{array}
\right), \,  
B'=\left(
\begin{array}{cccc}
x_{1,2}	&		\cdots 	  &  x_{1, n}     & 	\bar{b}		\\
x_{2,2}	&		\cdots 	  &  \bar{b}     & 	x_{2,n+1}			\\
\vdots 	&	        \vdots 	 &	\ddots 	  &	\vdots 		 	\\
x_{n-1,2}   &               \cdots           &  x_{n-1,n}     &   x_{n-1,n+1} \\
\bar{b}	&	 	\cdots         &	x_{n,n}   &	x_{n, n+1}			\\
\end{array}
\right)
\end{equation*}
By the induction hypothesis $A'$ and $B'$ have between them at least $n+1$ distinct rows. 
These rows will all remain distinct in $A$ and $B$ when we include  the first coordinate position.
The  $(n+1)$-st rows of $A$ and $B$ are identical  in their last $n$ entries:
\begin{equation*}
\bv_{n+1}= \left(
\begin{array}{cccc}
x_{n+1, 2}	 &	 	\cdots         &	x_{n+1,n}   &	x_{n+1, n+1} \\
\end{array}
\right) 
\end{equation*}	
 If this row $\bv_{n+1} $ with $n$ entries is different from all rows in $A'$ and $B'$, then we can include
the last row of $A$ with the rows of $A$ and $B$ corresponding to  the distinct rows of $A$ and $B$, to obtain
at least $n+1$ distinct rows in $R(A) \cup R(B)$.
 On the other hand if  $\bv_{n+1} $ already occurs in the set $R(A')\cup R(B')$ then 
 remove it from the list for $R(A')\cup R(B')$, obtaining at least $n$ 
 row positions  in $A$ and $B$ corresponding to distinct rows in  $R(A') \cup R(B')$ not equal to $\bv$.
Combine these row positions in $A,B$  with 
 the last row of both $A$ and $B$ (which differ in their $(n+1, 1)$ entry) 
 to obtain $n+2$ distinct rows in $R(A ) \cup R(B),$ completing the induction step. 
\end{proof}

%%%%%%%%%%N
% Remark 8.3
%%%%%%%
\begin{rem}\label{rem:73}
 The lower bound $n+1$ of  Lemma \ref{matlem}
is tight, on taking all $a_{i,j} = b_{i,j}= \bar{a}$ off the skew-diagonal.
\end{rem}

%%%%%%%%%%%%%%%%%%%%%%%%%%%%%%%%%%%%%%%%%%%%%%%%%%%%%%%
 %
% Section 8.2 Iterated Interleaving %
%%%%%%%%%%%%%%%%%%%%%%%%%%%%%%%%%%%%%%%%%%%%%%%%%%%%%%%
\subsection{Iterated interleaving and complete factorizations} \label{subsec:73}

We next consider {\em iterated interleaving factorizations}, in which we may continue to factorize a finitely factorizable path set $\sP$ 
 at  finitely factorizable factors $\sQ$ if they are decomposable.  
  An iterated factorization is said to be  {\em complete} if all factors are either indecomposable or are infinitely factorizable.

We associate to any (finite) iterated factorization a rooted tree   
 with root node $\sP$, and with leaf nodes corresponding to the factors in the iterated factorization,
 and with internal nodes labeled by factors in some intermediate factorization in the iteration process.
 The {\em depth} of a node in the tree is the number of edges traversed in the unique path to 
the root node. The root node has depth $0$. The {\em branching factor} at a node is the number of edges from it to nodes
at the next higher depth.

The iterated interleaving  process grows the tree, starting with a single root node $\sP$. 
Each iteration step  replaces one factor $\sQ$  in the current factorization by an 
   $n$-fold interleaving factorization of it (for some $n \ge 2$). This step corresponds to 
  adding  to this current  factor node $\sQ$ (which is a leaf node of the current tree)
$n$ new branches with the $n$-fold interleaving factors $\sQ_j = \psi_{j,n}(\sQ)$ as new leaves of $\sQ$, 
 so that  $\sQ$ becomes an internal node of the  new tree. 
 We treat any infinitely factorizable factors $\sQ$ encountered in the process as 
``frozen'' and do not factor them further. 
Figure 8.1
%\ref{fig71}
 exhibits such a tree.
%\newpage

 \begin{figure}[ht]\label{fig71}
	\centering
	\Tree[.$\sP$ 
$\sP_{0,4}$
[.$\sP_{1,4}$
[.$\mathcal{Q}_{0,2}$
$\mathcal{R}_{0,3}$
$\mathcal{R}_{1,3}$
$\mathcal{R}_{2,3}$
]
$\mathcal{Q}_{1,2}$
]
$\sP_{2,4}$
$\sP_{3,4}$
]
\newline

\vskip 0.2in 
 \hskip 0.2in {\rm FIGURE 8.1.} Iterated interleaving tree for
$\sP=\sP_{0,4}\pr((\mathcal{R}_{0,3}\pr\mathcal{R}_{1,3}\pr\mathcal{R}_{2,3})\pr\mathcal{Q}_{1,2})\pr\sP_{2,4}\pr\sP_{3,4}$.
\newline
\end{figure}

A factorization is said to be {\em complete} if each individual factor is either infinitely factorizable
or cannot be further factorized.

In the case of general closed sets $X \subset \sA^{\NN}$ the paper  \cite[Theorem 7.1]{ALS21} 
constructed examples where the recursive factorization procedure above can  go on forever
(with no infinitely factorizable factors), leading to
an infinite tree of factors of $X$. Furthermore these examples have no complete factorizations. 
The situation for  path sets is different:  any sequence of factorizations always terminates
in a complete factorization, see Theorem \ref{thm:complete_fact0} below.

The following  result will be used to establish termination of the  iterated interleaving process
for path sets.

%%%%%%%%%%%
% Theorem  8.4
%%%%%%%%%%
\begin{thm} \label{thm:75} 
Let $\sP$ be a finitely factorizable path set having $m$ vertices in its minimal 
right-resolving presentation. Then for $n \ge 2$ 
any  interleaving factor  $\sP_j= \psi_{j,n}(\sP)$ in an $n$-fold interleaving factorization of $\sP$
 has at most $m-1$ vertices in  its minimal right-resolving presentation.
\end{thm}

\begin{proof}
 Theorem \ref{thm:56} shows that
if $\sP$ has an $n$-fold  factorization $\sP= (\pr)_{i=0}^{n-1} \sP_j$ then each $\sP_j$ has
at most $m$ vertices in its minimal right-resolving presentation. Theorem \ref{thm:68} shows that if 
in addition some factor $\sP_j$
has exactly $m$ vertices in its minimal right-resolving factorization, then $\sP$ and all $\sP_j$
are leveled, hence infinitely factorizable, contradicting the hypothesis that
$\sP$ is finitely factorizable.  It follows that  all factors $\sP_j$ must have at most $m-1$
vertices in their minimal right-resolving presentation.
\end{proof}

Theorem \ref{thm:75} implies the existence of complete factorizations of finitely factorizable path sets.

%%%%%%%%%%%%%%%%%%%%%%%%%%%%%%%%%%%%%%%%%%%%%%%%%%%%%%%
% Theorem 8.5
%%%%%%%%%%%%%%%%%%%%%%%%%%%%%%%%%%%%%%%%%%%%%%%%%%%%%%%
\begin{thm} \label{thm:complete_fact0} 
{\rm (Complete Factorizations of Path Sets)}

(1) Every path set  $\sP$ has at least one iterated interleaving presentation that
is complete. 

(2) Every  path set $\sP$ 
 has  finitely many iterated interleaving factorizations 
 (with ``frozen'' infinitely factorizable factors) that are complete.
 
 (3)  If  $\sP$ has   $m$ vertices in its minimal right-resolving
 presentation,  then each such complete factorization  has at most $(m-1)!$ factors.
 In addition the associated tree has  depth at most $m-1$,
 and has branching factor at internal nodes of depth  $j$  of at most $m-j-1$ . 
\end{thm}

\begin{proof}
The path set $\sP$ has $m$ vertices in its minimal right-resolving presentation.   
We take any nontrivial $n$-fold interleaving factorization of $\sP$, and we know at least one factor in it will
be finitely factorizable.  
 If any finitely factorizable $\sQ= \sP_i$ appearing in it
is decomposable, we may further factorize i$\sP_i$  (as an iterated interleaving factorization), each new factorization having
at least one finitely factorizable factor. Theorem \ref{thm:non_leveled_fact} says that any $n$-fold factorization of  $\sQ$  necessarily has $n < m(\sQ)$,
where $m (\sQ)$ is the number of vertices in a minimal right-resolving presentation of $\sQ$. Let $\sQ=(\pr)_{i=0}^{n-1} \sQ_{j,n}$ 
 Theorem \ref{thm:75} then says the number of
vertices $m(\sQ_{j,n})$ in the minimal right-resolving presentation of $\sQ_{j,n}$ has $m(\sQ_{j,n}) < m(\sQ)$. 
Call the {\em level} of a node its distance from the root node, where
  the root node is assigned level $0$. Then any level $k$ node
in the tree that is finitely factorizable corresponds to a factor  $\sQ$ that has at most $m-k$ vertices in its minimal right-resolving presentation. 
 It follows from this fact that the maximal possible  level  of any node in the tree is   $m-1$,  bounding the maximum nesting level of parentheses in
 the iterated interleaving factorization. 
 In addition  the maximal
branching factor possible at level $k$ is $m-k-1$, applying  Theorem \ref{thm:non_leveled_fact}.
We conclude that all possible such  trees are finite, that they can  have one node at level $0$, at most  $m-1$ nodes at level $1$, at most $(m-1)(m-2)$ nodes at level $2$, up
to $(m-1)!$ nodes at level $(m-1).$ Now the maximal  number of leaves possible  in such a tree, terminating this process on any level, is $(m-1)!$.
The total number of nodes, counting internal nodes, is $(m-1)! ( 1 +\frac{1}{2} + \frac{1}{3!} +\cdots + \frac{1}{(m-1)!} \le e(m-1)!$. 
\end{proof} 

%%%%%%%%%%%%%%%%%%%%%%
% Remark 7.6
%%% %%%%%%%%%%%%%%%%%%%%
\begin{rem}\label{rem:76}
We do not settle the question whether there is  uniqueness of 
the indecomposable factors in a complete factorization.
 The  infinitely factorizable factors are non-unique.
\end{rem}

%

%%%%%%%%%%%%%%%%%%%%%%%%%%%%%%%%%%%%%%%%%%%%%%%%%%%%%%%%%%%%%
%
%  SECTION  9 Conclusion
%
%%%%%%%%%%%%%%%%%%%%%%%%%%%%%%%%%%%%%%%%%%%%%%%%%%%%%%%%%%%%%
\section{Concluding Remarks} \label{sec:concluding}

%%%%%%%%%%%%%%%%%%%%%%%%%%%%%%%%%%%%%%%%%%%%%%%%%%%%%%%
 %
% Section 9.1 
%
%%%%%%%%%%%%%%%%%%%%%%%%%%%%%%%%%%%%%%%%%%%%%%%%%%%%%%%
\subsection{Automatic sequences associated to  path sets} \label{sec:83}

 Theorem \ref{thm:decimation_set_bound} has the consequence that
 every path set is $n$-automatic  for all $n \ge 1$ in the sense given in \cite{AL14a}. 
 One may ask whether the are many integer counting statistics of an arbitrary path set $\sP$ are
 $n$-automatic sequences in the sense of Allouche and Shallit \cite{AS92}, \cite{AS03}. 
 A statistic of particular interest is  the function $f(k) = N_k^{I}(\sP)$ that counts the number of initial words of length $k$ in the path set $\sP$.
 
 Decimations and interleavings together define an infinite collection
  of closure operations $ X \mapsto X^{[n]}$ on $\sA^{\NN}$. These are idempotent operations that preserve the property of being
 closed sets in the symbol topology. One may define $\sC(\sA)^{[n]}$ to be the subclass of
 path sets that is invariant under the $n$-th closure operation.
 (It is not clear whether these classes of sets are closed under union or intersection, however.) 
  These closure operations are effectively computable on
 path sets, using presentations. One can ask if there are automata-theoretic characterizions
 of the class of path sets that is invariant under a given closure operation. 
 
%Appendices 
\appendix
%%%%%%%%%%%%%%%%%%%%%%%%%%%%%%%%%%%%%%%%%%%%%%%%%%%%%%%%%%%%%%%%%%%%%%%%%%%%%%%%
% SECT  10. Appendix A: Path Sets in Automata Theory
%%%%%%%%%%%%%%%%%%%%%%%%%%%%%%%%%%%%%%%%%%%%%%%%%%%%%%%%%%%%%%%%%%%%%%%%%%%%%%%
%%%%%%%%%%%%%%%%%%%%%%%%%%%%%%%%%%%%%%%%%%%%%%%%%%%%%%
%
% Section A.0
%
%%%%%%%%%%%%%%%%%%%%%%%%%%%%%%%%%%%%%%%%%%%%%%%%%%%%%%%

\section{Path sets in Automata Theory}\label{sec:A0}

Path sets have an important characterizations in  automata theory. 
We recall basic definitions and terminology in automata theory, following Eilenberg \cite{Eilenberg74}
and, for infinite words, Perrin and Pin \cite{PP04}.
 Recall that $\sA^{\star}$  denotes the set of all finite words in the alphabet $\sA$, including the empty word $\emptyset$.
 We let $A^{\omega}$ denote all infinite words $a_0a_1a_2 \cdots$ in the language with alphabet $\sA$ (rather than $\sA^{\NN}$, which we used in the main text.)

%%%%%%%%%%%%%%%%%%%%%%%%%%%%%%%%%%%%%%%%%%%%%%%%%%%%%%
%
% Section A.0.1
%
%%%%%%%%%%%%%%%%%%%%%%%%%%%%%%%%%%%%%%%%%%%%%%%%%%%%%%%

\subsection{Automata and languages}\label{sec:A0.1}

 %%Defn A.1 
 \begin{defn}\label{def:FA1}
%{\em  
A {\em finite automaton} on an alphabet $\sA$ is denoted $\bfA := (Q,  I, T)$ ( in full $(Q, \sA, E, I, T)$)  which has a finite directed labeled graph
in which  $Q$ denotes the (finite) 
set of its states, with specified subsets  $I$ of {\em initial states}  and  $T$ of {\em terminal states} (or {\em final states}). (The sets $I$ and $T$ may overlap.)
Additional data speifying the automaton consists of  labeled  edge data  $E \subset Q \times \sA \times Q$, 
writing $ e=(v_1, a, v_2)$, for the directed edge from state $v_1$ to state  $v_2$, 
carrying  a label $a \in \sA$.

\end{defn} 

The alphabet $A$ and the labeled edge  data  $E$ are traditionally omitted from the notation for $\bfA$. 

%%Defn A.2
\begin{defn}\label{def:DET1}
%{\em
A finite automaton is {\em deterministic} if $I$ contains one element and each state $v \in V$ has
for each symbol $a \in \sA$ at most one exit edge labeled with this symbol.
Otherwise it is {\em nondeterministic}.
%}
\end{defn}

A deterministic automaton is characterized by the property that for each finite path the symbolic path label data plus the initial state on the path uniquely determine 
the path; i.e., the sequence of states visited in following the path. 

%%defn A.3
\begin{defn}\label{def:FL1}
%{\em 
The {\em formal language $L(\bfA) \subset \sA^{\star}$}  associated to a finite automaton $\bfA$ is the set of all finite symbol sequences obtained 
as labels following
some directed path starting from an initial vertex $v \in I$ and ending at some terminal vertex $w \in T$.
The empty sequence (denoted $1$) is included  in   $L(\bfA)$  if some initial state is a terminal state. 
%}
\end{defn}

%%defn A.4
\begin{defn}\label{def:BA}
%{\em 
A  finite {\em B\"{u}chi automaton}
%\footnote{B\"{u}chi's theory was developed to treat automata having a countably infinite set of states.}
has the same automaton $\bfA$, but the (B\"{u}chi)  language 
$ L^{\omega}(\bfA)$ {\em recognized} by such an  automation is the set of  all infinite words $a_0 a_1 a_2\ldots \in \sA^{\NN}$ which
are labels of an  infinite directed path starting at an initial state and passing through terminal states at infinitely  many times.
%}
\end{defn}

The theory of B\"{u}chi \cite{Buchi62} as originally developed allowed  automata having a  countably infinite number of states.
 
%%Defn A. 5
\begin{defn}\label{def:RECOG} 
%{\em 
For fixed $\sA$ the  set of languages recognized by some finite  B\"{u}chi automaton, allowing nondeterminism and arbitrary sets of initial and terminal vertices, are called
{\em recognizable languages}  on alphabet $\sA$. (\cite[page 25]{PP04}.)
%}
\end{defn}

 The class of recognizable languages are characterized as coinciding with the class of $\omega$-rational languages.
(\cite[Chapter I. Theorem 5.4]{PP04}).

%%%%%%%%%%%%%%%%%%%%%%%%%%%%%%%%%%%%%%%%%%%%%%%%%%%%%%
%
% Section A.0.2 
%
%%%%%%%%%%%%%%%%%%%%%%%%%%%%%%%%%%%%%%%%%%%%%%%%%%%%%%%

\subsection{Automata-theoretic characterization of  path sets }\label{sec:A0.2}

To a presentation $(\sG, v)$ of a path set $\sP$ on a finite alphabet $\sA$ we canonically associate 
the finite automaton $\bfA_{\sG}=(Q, I, T)$ where $Q= V(\sG)$, the initial state set $I= \{ v\}$ and the  terminal state set $T = V(\sG)$ 
consisting every state of $\sG$, and with edge label data specified  by $\mathcal{E}$.
A path set $\sP = X(\sG, v)$ then coincides with the B\"{u}chi language $ L^{\omega}(\bfA)$ recognized by the associated B\"{u}chi automaton 
$\bfA$ associated to $(\sG, v)$. 
We have the following characterization. 

%%%%%%%%%%%%
% Theorem A.6 
%%%%%%%%%%%%%
\begin{thm}\label{thm:PSA}
{ \rm (Automata Characterization of Path Sets)}
The following properties are equivalent, for a finite alphabet $\sA$.
\begin{enumerate}
\item[(1)]
$\sP$ is a path set with alphabet $\sA$.
\item[(2)]
$\sP$ is a recognizable language that is  a closed set in $\sA^{\NN}$.
\item[(3)] 
$\sP$ is recognized by a finite B\"{u}chi automaton in which every state is terminal.
\item[(4)] 
$\sP$ is recognized by a finite deterministic B\"{u}chi automaton in which every state is terminal.
\end{enumerate}
\end{thm}

\begin{proof}
The path set definition $(1)$ is equivalent to  $(3)$.
The equivalence of  $(2), (3), (4)$ and $(5)$ is shown in  Proposition 3.9 in Chapter III of Perrin and Pin \cite{PP04}.  The reduction from (3) to (4) uses the fact that every state of the automaton is a terminal state. 
\end{proof}

It is known that the set  of  languages recognizable languages
is strictly larger that those recognized by deterministic B\"{u}chi automata. The
set of  recognizable languages
are closed under complement, while the languages recognized by deterministic B\'{u}chi automata
(allowing arbitrary subsets of $Q$ of terminal states) are not closed under complement. 
However the set of detrerminisitic B\"{u}chi languages has a nice characterization.
given  in \cite[Chapter III, Corollary 6.3]{PP04}  

The  class  $\sC(\sA)$ of path sets forms  a strict subset of the  languages recognized by some
deterministic B\"{u}chi automaton; i.e., there are non-closed sets of $A^{\NN}$ recognized by some deterministic B\"{u}chi automaton, as shown by the following example.

%%%%%%%%%%
% Example A.7
%%%%%%%%%%
\begin{exmp}\label{exam:non-closed} 
Consider the directed labeled graph with two states pictured, on  alphabet $\sA= \{0.1\}.$ 

%[Figure to be added.]

%%%%Figure A.1%%%%%%%
\begin{figure}[ht]\label{fig0}
	\centering
	\psset{unit=1pt}
	\begin{pspicture}(-130,-30)(130,20)
		\newcommand{\noden}[2]{\node{#1}{#2}{n}}
		\noden{$v_0$}{-30,0}
		\noden{$v_1$}{30,0}
		\bcircle{n$v_0$}{90}{1}
		\bcircle{n$v_1$}{270}{1}
		\bline{n$v_0$}{n$v_1$}{0}
	\end{pspicture}
\newline
\hskip 0.5in {\rm FIGURE A.1.} Graph for Example ~\ref{exam:non-closed}. The marked initial vertex is $v_0$.
\newline
\end{figure}
%%%END Figure A.1

Let $I =\{v_0\}$ be the initial state and let $\bfA_1$ denote the automaton with every state terminal
$T =\{ v_0, v_1\}$ and let $\bfA_2$ denote the automaton with $T = \{v_1\}$.
Then
$$
L^{\omega}(\bfA_1) = \{ 1^k 0 1^{\infty}: \,  k \ge 0\}\, \cup\, \{ 1^{\infty}\}
$$
while 
$
L^{\omega}(\bfA_2) = \{ 1^k 0 1^{\infty}: \,  k \ge 0\}.
$
The set $L^{\omega}(\bfA_1)$ is a closed set in $\sA^{\omega}$, and is a path set.
The set  $L^{\omega}(\bfA_2)$ is not closed in $\sA^{\omega}$ because it omits the limit point
$\{ 1^{\infty}\}$. 
\end{exmp} 

%%%%%%%%%%%%%%%%%%%%%%%%%%%%%%%%%%%%%%%%%%%%%%%%%%%%%%
%
% Section A.0.3
%
%%%%%%%%%%%%%%%%%%%%%%%%%%%%%%%%%%%%%%%%%%%%%%%%%%%%%%%

\subsection{Minimal deterministic automata}\label{sec:A0.3}

A  finite automaton $\bfA$  is  {\em deterministic} if  
the initial state plus the symbols sequence of some legal  path uniquely determine the set of states 
that path  passes through. This concept is equivalent to {\em right-resolving} in
symbolic-dynamics terminology.

%%%%%Defn A.8
\begin{defn}\label{def:presentation}
(1) 
An automaton $(Q, I, T)$ is {\em accessible} if every state can be reached by a directed path from some initial state.

(2) An automaton is {\em co-accessible} if every state has a directed path to a terminal state.

(3) An automaton that is accessible and co-accessible is called {\em trim}.
\end{defn}

For a finite automaton $\bfA$ with one initial state $v$,   the property of {\em  reachability}  is equivalent to it  being {\em accessible.}
For a finite  automaton $\bfA$ in which   every state is  terminal, the property of being {\em pruned} 
 is equivalent to being {\em co-accessible.} 
For  finite automaton the property of right-resolving is equivalent to being {\em deterministic}. 
A deterministic automaton is {\em complete} if each state has   an exit edge with label for 
  each letter of the label alphabet $\sA$.

%%%%THM A.9%%%%
\begin{thm}\label{thm:deterministic} 
Every finite (non-deterministic) automaton $\bfA= (Q, I, T)$ has an equivalent deterministic  automaton $\bfA^{\ast}= (Q^{\ast}, I^{\ast}, T^{\ast})$  that
has $L(\bfA) = L(\bfA^{\ast}).$
\end{thm} 

\begin{proof}
This is proved using the subset construction \cite[Chapter III, Sect. 2]{Eilenberg74}.
\end{proof}

%%%Theorem A.10%%%%%%%
\begin{thm}
(1) Every finite  automaton $\bfA$ has a minimal deterministic presentation $\bfA_{min}$ giving the same language $L(\bfA)$.This presentation is
unique up to isomorphism of automata. 
(2) The  minimal automaton is deterministic and accessible. 
The behavior of accepting paths from each state of the minimal automaton is different. 
The minimal automaton  is trim unless there is a state
which has no accepting exit paths (``null state.'') 
%In the minimal automaton, the languages are different. 
(3) The minimal automaton is not complete unless $L(\bfA) = \sA^{\NN}$ is the full shift.
\end{thm} 

\begin{proof} See Eilenberg \cite[Chapter III, Sect. 5]{Eilenberg74}. Theorem 5.2 gives (1) and (2), and (3) is mentioned in this proof.
The proof is by construction.
%[ADD; Minimal trim automaton.]
%%defn
\end{proof} 

Given a finite automaton, there is an  effectively computable procedure to obtain the minimal deterministic automaton,
see the discussion in Eilenberg \cite[Chapter III] {Eilenberg74}.
Some authors study  {\em complete minimal automata}, which are obtained by adding an extra ``null state'' to the automaton
which is not co-accessible, e.g. Sakarovich \cite[Chapter I.4.2]{Sakarovich:90}. 

%%%%%%%%%%%%%%%%%%%%%%%%%
% Remark
%%%%%%%%%%%%%%%%%%%
\begin{rem}
 Every finite B\"{u}chi automaton  recognizes a language recognizable by a finite trim B\'{u}chi automaton
(\cite[Chapter I, Proposition 5.2]{PP04}. If the automaton is deterministic then the corresponding trim automaton is
deterministic.  But for general automata the  set of  terminal states may be altered under the transformation.
\end{rem}

%%%%%%%%%%%%%%%%%%%%%%%%%%%%%%%%%%%%%%%%%%%%%%%%%%%%%%
%
% Section A.0.4
%
%%%%%%%%%%%%%%%%%%%%%%%%%%%%%%%%%%%%%%%%%%%%%%%%%%%%%%%

\subsection{Path set languages and closures of rational languages}\label{sec:A0.4}

% Defn A,11
\begin{defn}\label{def:PSL}
{\em 
A {\em path set language} $L^{I}(\sP) \subset \sA^{\star}$ of a path set $\sP$  is the set of finite initial prefixes of words in
the path set.
%%%%%%%%%%%%%%%%%% 
%It always  includes the empty word.
%%%%%%%%%%%%%%%%%%%%
}
\end{defn}

Eilenberg \cite[Chapter XIV, page 379]{Eilenberg74} defines a closure operation on a  language $L \subset \sA^{\ast}$,
which defines $\bar{L} \subset A^{\omega}$ to be the set of all infinite words $a_0a_1a_2 \ldots$ which has infinitely 
many prefixes $a_0a_1 \ldots a_n$ being a word in $L$. 

% Defn A.12
\begin{defn}\label{def:prefix-closed}
A language $L \subset A^{\star}$ is called {\em prefix-closed} if all prefixes of each word in $L$ belongs to $L$.
\end{defn}

A path set language $C= L(\sP)$  is prefix-closed. One also has
$\overline{L(\sP)} = \sP$. 

%%%%%%%%%%%%%%%%%%%%%%%%%%%%%
%ThM-Path Set Languages 
%Theorem A.14
%%%%%%%%%%%%%%%%%%%%%%%%%%%%%%%%%%
\begin{thm}\label{thm:PSL}
{\rm (Path Set Language Characterization)} 
The following properties are equivalent.
\begin{enumerate}
\item[(1)]
$\sL$ is a path set language $L^{I}(\sP)$ for some $\sP$ using the finite alphabet $\sA$.
\item[(2)]  
$\sL$ is a rational language which  is prefix-closed in the sense that  all the prefixes of 
each word $w \in L$ also belong to $L$.
\item[(3)]
$\sL$ is  recognizable by a trim finite deterministic automaton,
in which all states are terminal. 
\end{enumerate}
For any such language the associated path set  $\sP = \overline{L}$  consists of 
all infinite words which have infinitely many prefixes belonging to $L$.
Any other rational languages $L_1$ which have $\sP$ as their closure have $L_1 \subset L$.
\end{thm}

\begin{proof}
The equivalence of (2) and (3) appears as Exercise 5.1 in \cite[page48]{Eilenberg74}. 
\end{proof}

One can define a closure operation on regular languages $L$, as follows.  One first passes  to the infinite language
$\bar{L}$, which is a path set.  
Then one passes to the initial prefix language ${\sL}^{i}( \bar{L})$ of this path set, which is a prefix-closed language.
This relation is idempotent because the closure of ${\sL}^{i}( \bar{L})$ is again $\bar{L}.$
If $L$ is recognized by a co-accessible automaton then $L \subseteq {\sL}^{i}( \bar{L})$.\\

%%%%%%%%%%%%%%%%%%%%%%%%%%%%%%%%%%%%%%%%%%%%%%%%%%%%%%%
 %
% Section Appendix B.  Self-interleaving 
%
%%%%%%%%%%%%%%%%%%%%%%%%%%%%%%%%%%%%%%%%%%%%%%%%%%%%%%%
\section{Self-interleaving criterion} \label{sec:B0}

We give a sufficient condition   on a presentation of  a path set to guarantee that  all  interleaving factorizations
of it are   self-interleavings (possibly of other path sets).

%%%%%%%%%%%%%%%%%%%%%%%%%%%%%%%%%%%%%%%%%%%%%%%%%%%%%%%
% Theorem 8.1
%%%%%%%%%%%%%%%%%%%%%%%%%%%%%%%%%%%%%%%%%%%%%%%%%%%%%%%

\begin{thm}\label{thm:selfloop}
{\rm (Sufficient condition for only  self-interleaving factorizations )}
Let $\sP= X(\sG,v)$ be a path set having a right-resolving presentation
in which  the  graph $\sG$ has a self-loop at
the initial vertex $v$.
Then  for all $n \ge 1$ any $n$-fold interleaving factorization $\sP= \sP_0 \pr \sP_1 \pr \ldots \pr \sP_{n-1}$
must be a  self-interleaving $\sP = \sQ^{(\pr n)}$ of a single factor $\sQ$, where $\sQ$ depends on $n$.  
\end{thm} 

\begin{proof}
 Suppose that the presentation $(\sG, v)$ has a self-loop at vertex $v$ with label $a_0 \in \mathcal{A}$.
Then  $x \in \sP$ implies that $(a_0)^kx \in \sP$ for each $k \ge 1$. 
In consequence:  
\begin{enumerate}
\item 
We have  $\sP \subseteq S(\sP)$, where $S$ denotes the left-shift operator, because $S(a_0x) = x$.  
\item
The $n$-decimation sets $\sP_j= \psi_{j,n}(\sP)$  satisfy the inclusions
$$
\psi_{j,n}(\sP) \subseteq \psi_{j+1, n}(\sP)
$$
for all $j \ge 0$. Indeed if $y \in \psi_{j, n}(\sP)$ then there exists $x \in \sP$ with $\psi_{j,n}(x) =y$.
But now  $a_0 x \in \sP$ and $y = \psi_{j+1, n}(a_0 x)$ so $y \in \psi_{j+1,n}(\sP)$.
\end{enumerate} 

%%%%
By Theorem ~\ref{thm:initialblocks}, a path set is characterized by its set $\sB^{I}(\sP)$ 
of initial blocks. If we can show that $\mathcal{P}_i$ has the same set of initial blocks as 
$\mathcal{P}_0$ for all $1\leq j \leq n-1$, then
 we will have proven the result, taking $\sQ= \sP_0$.
  By hypothesis there is a right-resolving presentation $(\mathscr{G},v)$ of $\mathcal{P}$ having a self-loop labeled $a_0$ at $v$.
Let $x=x_0x_1x_2\ldots x_m$ be an initial block in  $\mathcal{P}_0$. Then there are words $w_0,w_1,\ldots w_m$, each of length $n-1$, such that 
$x_0w_1x_1w_2x_2\ldots w_mx_m$ is an initial block in $\mathcal{P}$. Then for all $0\leq i\leq n-1$, $a^ix_0w_1x_1w_2x_2\ldots w_mx_m$ is an initial block in $\mathcal{P}$. 
By the definition of interleaving, $x$ is an initial block of $\mathcal{P}_i$ for each $0\leq i\leq n-1$. Thus, any initial block of $\mathcal{P}_0$ is an initial block of all other $\mathcal{P}_j$ with $1 \le j\leq n-1$. 

Now, for some $0\leq i\leq n-1$, let $y=y_0y_1\ldots y_m$ be an initial block of $\mathcal{P}_i$. Since $\mathscr{G}$ has a self loop labeled $a_0$, 
a path traversing the loop infinitely many times presents the point $(a_0)^{\infty}\in\mathcal{P}$.
 Thus, each $\mathcal{P}_j$, $0\leq j\leq n-1$, contains the point $(a_0)^{\infty}$. Using this point from each $\mathcal{P}_j$ except $\mathcal{P}_i$, 
 we can form a point in $\mathcal{P}$ with initial word $(a_0)^{i-1}y_0(a_0)^{n-1}y_1(a_0)^{n-1}y_2\ldots (a_0)^{n-1}y_m$. Since $\mathscr{G}$ is right-resolving, the path presenting this point begins by traversing the $a_0$-labeled self loop $n-1$ times, thus ending at the initial vertex. Therefore, there is a path beginning at the initial vertex and presenting 
 $y_0(a_0)^{n-1}y_1(a_0)^{n-1}y_2\ldots (a_0)^{n-1}y_m$. Therefore, $y$ is an initial block of $\mathcal{P}_0$. Therefore $\sP_0 =\sP_i$ for all $i$. 
 since they have the same initial block set.
 \end{proof} 

%%%%%%
% Remark
%%%%%%
\begin{rem}
Certain examples of path sets (denoted $X(1, M)$) studied in
Abram et al  \cite[Section 3.4 and Proposition  5.1]{ABL17}  exhibited interleaving factorizations that were self-interleavings.
The presentations of such $X(1,M)$ have self-loops at the initial vertex, and Theorem \ref{thm:selfloop} applies to
give this result.These path sets arise in the study of intersections of $p$-adic path set fractals, as defined and studied in \cite{AL14b}, and arose in consideration of a problem of Erd\H{o}s, c.f. \cite{Lagarias09}.
\end{rem}
%%%%%%%%%%%%%%%%%%%%%%%%%%%%%%%%%%%%%%%%%%%%%%%%%%%%%%%%
%% REFERENCES
%%%%%%%%%%%%%%%%%%%%%%%%%%%%%%%%%%%%%%%%%%%%%%%%%%%%%%%%


\begin{thebibliography}{9}
\bibitem{ABL17} 
% sec. 7
W. C. Abram, A. Bolshakov, and J. C. Lagarias,
Intersection of multiplicative translates of 3-adic Cantor sets II: Two infinite families.
Experimental Mathematics {\bf 26} (2017),  no.4, 468--489.

\bibitem{AL14a} 
%sec. 1,
W. C. Abram and J. C. Lagarias,
Path sets in one-sided symbolic dynamics.
Adv. in Appl. Math. {\bf 56} (2014), no. 1, 109-134.

\bibitem{AL14b} 
%sec 1
W.C. Abram and J. C. Lagarias,
$p$-Adic path set fractals and arithmetic.
J. Fractal Geom. {\bf 1}   (2014), no. 1, 45-81.

\bibitem{AL14c} 
% sec 7
W. C.Abram and J. C. Lagarias,
Intersections of multiplicative translates of 3-adic Cantor sets.
J. Fractal Geom. {\bf 1}  (2014), no. 4, 349-390.

\bibitem{ALS21}
W. C.Abram, J. C. Lagarias  and D. Slonim
Decimation and interleaving  in one-sided symbolic dynamics, 
Advances in Applied Mathematics, to appear. 



\bibitem{AS92}
% new 191104
J.P. Allouche, J. O. Shallit,
{The ring of $k$-regular sequences,}
Theor. Comp. Sci. { 96} (1992) \,163--197.



\bibitem{AS03}
% new 191104
J.P. Allouche, J. O. Shallit,
{Automatic Sequences: Theory, Applications, Generalizations},
Cambridge University Press: Cambridge 2003.

\bibitem{AKS92}
J. Ashley, B. Kitchens, M. Stafford,
{Boundaries of Markov Partitions,}
Trans. Amer. Math. Soc. { 333} (1992) \, 177--201.

\bibitem{DR02}
A. Baes and M. Rigo,
{More on generalized automatic sequences},
J. Autom Lang. Comb. {\bf 7} (2002), no. 3,  351--376.

\bibitem{BC16}
% sec. 1
 Ban, Jung-Chao; Chang, Chih-Hung, 
Solution structure of multi-layer neural networks with initial condition. 
J. Dynam. Differential Equations {\bf 28}  (2016), no. 1, 69-92.


\bibitem{Buchi62}
J. R.  B\"{u}chi, 
On a decision method in restricted second order arithmetic,
in {\em Logic, Methodology and Philosophy of Science: Proc. 1960 Interational  Congress}
pp. 1-11 (Stanford University Press, 1962).

\bibitem{Co71}
% Appendix 2
J. H. Conway,
{\em Regular Algebra and Finite Machines},
Chapman and Hall, Inc.: London 1971.

\bibitem{DeB:81a}
N. G. de Bruijn,
\emph{Sequences of zeros and ones generated by special production rules},
Indag. Math.  {\bf 84} (1981), 27--37. 

\bibitem{DeB:81b}
%new
N. G. de Bruijn,
\emph{Algebraic theory of Penrose's non-perioidic tilings of the plane I,},
Indag. Math.  {\bf 84} (1981), 39--52.

\bibitem{DeB:81c}
%new
N. G. de Bruijn,
\emph{Algebraic theory of Penrose's non-perioidic tilings of the plane II,},
Indag. Math.  {\bf 84} (1981), 53--66.

\bibitem{DeB:89}
%new
N. G. de Bruijn,
\emph{Updown generation of Beatty sequences,}
Indag. Math.  {\bf 92} (1989), 385--407. 

\bibitem{DFLL:01}
G. Duchamp, M. Flouret, \'{E}. Laugerotte, J.-G. Luque,
Direct and dual laws for automata with multiplicities,
Theoretical Computer Science {\bf 267} (2001), 105--120.

\bibitem{Eilenberg74}
S. Eilenberg,
\emph{Automata, languages and machines, Volume A.}
Pure and Applied Mathematics, Vl. 59A,
Academic Press, New York 1974.

%\bibitem{Eilenberg76}
%S. Eilenberg,
%\emph{Automata, languages and machines, Volume B.}
%(with two chapters by Bret Tilson), 
%Pure and Applied Mathematics, Vl. 59B,
%Academic Press, New York 1976.

\bibitem{Erdos79} 
%Sec. 7
P. Erd\H{o}s, 
Some unconventional problems in number theory. 
Math. Mag. {\bf 52} (1979), 67-70.

\bibitem{Guy04}
R. K. Guy,
{\em Unsolved Problems in Number Theory: Third Edition},
Springer-Verlag: New York 2004. 

\bibitem{KMRRS:09}
D. Krieger, A. Miller, N. Rampersad, B. Ravikumar, J. O.  Shallit,
Decimations of languages and state complexity,
Theor. Comput. Sci. {\bf 410} (2009), no. 24-25, 2401--2409. 

\bibitem{Lagarias09} 
% Sec 1.0, Sec. 7
J.C. Lagarias, 
Ternary expansions of powers of $2$.
 J. London Math. Soc. {\bf 79} (2009), no. 3, 562-588.
 
\bibitem{LM95}
% many 
D. Lind and B. Marcus,
\emph{An Introduction to Symbolic Dynamics and Coding}.
Cambridge University Press, New York, 1995. (Reprinted 1999 with corrections.)

\bibitem{Lo03}
% not cited
M. Lothaire,
\emph{ Algebraic Combinatorics on Words,}
Cambridge University Press: Cambridge 2003.


\bibitem{McN62}
% New 191104
R. McNaughton,
Testing and generating infinite sequences by finite automata,
Information and Control {\bf 9} (1966), 521--530.

\bibitem{OTW14}
% sect 1.2
W. Ott, M. Tomforde, and P. N. Willis,
\emph{One-sided shift spaces over infinite alphabets,}
NYJM Monographs, No. 5.
State University of New York, University at Albany, Albany, NY.
2014, 54 pages.

\bibitem{PP04}
% sect 2.2
D. Perrin and J-E. Pin,
\emph{Infinite Words: Automata, Semigroups, Logic and Games,}
Elesevier: Dordrecht 2004.

\bibitem{Rigo:1999}
M. Rigo,
\emph{Generalization of automatic sequences for numeration systems on
a regular language},
Theor. Comput. Sci. {\bf 244} (2000), 271--281. 

\bibitem{Sakarovich:90} 
J. Sakarovich,
{\em Elements of Automata Theory,}
Cambridge University Press: Cambridge 2009.
(English Translation of: {\em \'{E}lements de th\'{e}orie des automates.} Vuibert, Paris: 2003). 


\bibitem{Shallit:88}
J. Shallit,
\emph{A generalization of automatic sequences,} 
 Theor. Comput. Sci. {\bf 61} (1988), 1--16.


\bibitem{VO89}
% Sec 1.0 [sect. 7.5 in book]
S. A. Vanstone and P. van Oorschot,
\emph{An Introduction to Error Correcting Codes with Applications}.
Kluwer Academic Publishers, Boston, 1989.

%\bibitem{MVA97}
%A. J. Menezes, P. C. van Oorschot, and S. A. Vanstone,
%\emph{Handbook of applied cryptography.}
%CRC Press Series on Discrete Mathematics and its Applications,
%CRC Press: Boca Raton, FL 1997.
\end{thebibliography}
\end{document}